\documentclass[times,final,11pt]{elsarticle}
\usepackage{jcomp}
\usepackage{amsmath}
\usepackage{amssymb}
\usepackage{graphicx}
\usepackage{algorithm}
\usepackage{booktabs}
\usepackage{amssymb,amsthm,amsmath,bm,amsfonts}

\usepackage{natbib}
\usepackage{hyperref}

\usepackage{pgfplots}
\usepgfplotslibrary{external}
\newcommand{\R}{\mathbb{R}} 
\newcommand{\A}{\mathcal{A}} 
\newcommand{\B}{\mathcal{B}} 
\newcommand{\D}{\mathcal{D}}

\usepackage{hyperref}
\usepackage{subcaption}

\newtheorem{remark}{Remark}
\newtheorem{definition}{Definition}
\newtheorem{theorem}{Theorem}[section]

\begin{document}

\begin{frontmatter}
	
       \title {{\bf A relaxation system  for the  low-Mach limit of kinetic equations: Stability and higher order asymptotic preserving scheme}\tnoteref{tnote1} 
	}

	\tnotetext[tnote1]{The work of Theresa K\"{o}fler has been supported by the European Union's Horizon Europe research and innovation programme under grant agreement No. 101086214 (DATAHYKING project).
	The work of Axel Klar and Sudarshan Tiwari has been supported by German Research Foundation (DFG) through grant  KL-1105-34,  ``Asymptotic preserving higher-order  meshfree schemes for kinetic equations".
	The work of Lorenzo Pareschi was partially supported by the Royal Society under the Wolfson Fellowship ``Uncertainty quantification, data-driven simulations and learning of multiscale complex systems governed by PDEs" and by the FIS2023-01334 Advanced Grant ``Tackling complexity: advanced numerical approaches for multiscale systems with uncertainties" (ADAMUS). The work of Giacomo Dimarco has been supported by by the Italian Ministry of University and Research (MUR) through the PRIN 2022 project (No. 2022KKJP4X) ``Advanced numerical methods for time dependent parametric partial differential equations with applications". This work has been written within the activities of GNFM and GNCS groups of INdAM (Italian National Institute of High Mathematics). Program codes can be found on \url{https://github.com/sudarshan825/}.}
	
	\author[FE,CMCS]{Giacomo \snm{Dimarco}}
	\ead{giacomo.dimarco@unife.it}

	\author[KA]{Axel \snm{Klar}}
	\ead{klar@rptu.de}

	\author[FE,KA]{Theresa \snm{K\"{o}fler}}
	\ead{theresa.kofler@unife.it}

	\author[FE,HW]{Lorenzo \snm{Pareschi}}
	\ead{l.pareschi@hw.ac.uk}
	
	\author[KA]{Sudarshan \snm{Tiwari}}
	\ead{stiwari@rptu.de}
	
	\author[AA]{Yizhou \snm{Zhou}}
	\ead{zhou@igpm.rwth-aachen.de}

	\address[KA]{Department of Mathematics, RPTU Kaiserslautern-Landau, Kaiserslautern, DE}

	\address[FE]{Department of Mathematics and Computer Science, University of Ferrara, Italy}
	\address[CMCS]{Center for Modeling, Computing and Statistics, University of Ferrara, Italy}
	\address[HW]{Maxwell Institute for Mathematical Sciences and Department of Mathematics, Heriot-Watt University, Edinburgh, UK}
	\address[AA]{Department of Mathematics, RWTH Aachen, Aachen, DE}

\begin{abstract} 
	This work introduces a  hyperbolic relaxation system designed for the simulation of kinetic equations in the low-Mach number limit. The methodology is built upon a micro-macro decomposition of the scaled BGK model, which reformulates the kinetic distribution function into a coupled system consisting of a macroscopic equilibrium part and a microscopic non-equilibrium remainder. By projecting the microscopic deviations onto a set of orthogonal polynomials, we derive a closed moment-relaxation system.
	The resulting system is a version of Grad's 13 moment system with a linear hyperbolic part and a relaxation adapted to the incompressible limit.
	%A central contribution of this work is the rigorous analytical investigation of the model’s structural stability; we prove that the system is symmetrizable and satisfies the necessary dissipative properties required for formal convergence to the incompressible Navier–Stokes equations. 
	We prove the model's structural stability in  the incompressible Navier–Stokes limit. Moreover, 
	we develop a high-order Asymptotic-Preserving (AP) numerical framework using Implicit-Explicit (IMEX) Runge–Kutta schemes for temporal accuracy and finite difference WENO reconstructions as well as central difference approximations for high-order spatial resolution. The proposed scheme ensures uniform stability and consistency across different physical regimes, automatically degenerating into a consistent high order discretization of the incompressible thermal Navier--Stokes limit  as the scaling parameter vanishes. Numerical experiments in one and two dimensions corroborate the theoretical findings, demonstrating significant reductions in computational cost and robust performance across a wide range of Knudsen and Mach numbers.
\end{abstract}
	
	
\end{frontmatter}

\section{Introduction}

Kinetic equations, like the Boltzmann equation and its simplified BGK variant, act as a bridge from the microscopic dynamics of particles to macroscopic quantities \cite{Cer,Levermore91}. These macroscopic limits are obtained when a system is dense enough, meaning the number of particles is large and interactions occur frequently relative to the scale of interest. This rapid interaction allows tiny regions within the system to reach local thermodynamic equilibrium almost instantly. This occurs when the Knudsen number, representing the ratio of the mean free path of particles between collisions to a characteristic length of the flow, approaches zero. Assuming a fixed Reynolds number, the Mach number proportionally decreases as the hydrodynamic limit is approached, asymptotically recovering the incompressible Navier-Stokes (INS) equations \cite{GSR,Saint2003}. However, solving the full kinetic equation in this limit is computationally demanding due to high dimensionality and extreme numerical stiffness. 

From a numerical perspective, the first issue arises from the discretization of the additional dimensions of the velocity space. Based on discrete kinetic models, the Discrete-Velocity-Method (DVM) relies on replacing the continuous velocity space with a finite set of discrete velocities defined on a fixed grid \cite{Mieussens,Platkowski1988}. This transformation converts the original kinetic equation into a large system of coupled partial differential equations. High-resolution requirements for the velocity grid lead to a computational burden that often becomes prohibitive, especially for multidimensional problems. Starting with the seminal paper by Grad \cite{grad1949kinetic}, this issue is frequently addressed by deriving approximate equations for the moments of the distribution function of the kinetic equation. A key challenge here is the so-called closure problem: deriving a finite set of equations from the infinite moment hierarchy while preserving essential physical properties. A prominent framework for addressing this is Extended Thermodynamics, as comprehensively formulated by Müller and Ruggeri \cite{MuellerRuggeri}, which provides a rigorous, thermodynamically consistent basis for closing these macroscopic systems. The literature on the moment approach is extensive; as further examples, we cite \cite{struchtrup2003regularization, sarna2020convergence,frank2007approximate,brunner2005one,frank2007fast,garrett2013comparison,alldredge2014adaptive,alldredge2019regularized} among many others, and the references therein. Usually, such systems fall into the class of hyperbolic relaxation systems, which have emerged as a powerful framework to approximate complex multiscale dynamics. Generally, these systems consist of first-order hyperbolic equations, where a stiff source term drives the state toward a specific equilibrium manifold \cite{Yong2001,JinLevermore1996,NaldiPareschiToscani2002}. In the specific context of kinetic equations, relaxation systems are often designed to capture the transition from the rarefied regime to the fluid dynamic limit, ensuring stability in the limit and avoiding the oscillations frequently associated with classical macroscopic approximations \cite{Cor-Yong}.

The second numerical issue mentioned above---the stiffness of the equations in the fluid dynamic limit---remains present in moment systems. So-called asymptotic preserving (AP) methods can be designed for either the kinetic equation or the moment system to ensure stability and accuracy across different regimes, eliminating the need to refine discretization parameters according to the small scales of the problem. Specifically, AP schemes are formulated such that, as the scaling parameter vanishes, the numerical method automatically degenerates into a consistent discretization of the limiting macroscopic equations. This property makes them exceptionally effective for multiscale problems where kinetic and fluid regimes coexist and interact. However, it is crucial to distinguish between the numerical AP property and the physical validity of the underlying mathematical model. While AP numerical frameworks are theoretically designed to be stable and to function across the entire range of the scaling parameter $\varepsilon$, the macroscopic relaxation system derived in this work is physically limited in this regard. The proposed moment-relaxation system is explicitly tailored to approximate the dynamics in the asymptotic regime of small $\varepsilon$ (the low-Mach and continuum limit) and is not designed to accurately capture highly rarefied gas dynamics where $\varepsilon$ is large. Therefore, while the AP scheme ensures a seamless and efficient numerical transition to the incompressible Navier-Stokes limit without grid-induced stiffness issues, the physical accuracy of the solution for large $\varepsilon$ remains strictly constrained by the inherent truncation of the underlying moment model. 

In recent years, relaxation-based AP schemes have been successfully applied to a wide range of applications, proving to be a versatile alternative to direct kinetic solvers \cite{JinPareschiToscani1998,JinPareschi2000,Russo,JinAP3,JinAP2,JinAP1,Klar99,DegondTang2011,HuangXingXiong2025,DimarcoLoubereVignal2017}. For the specific problem of the numerical incompressible Navier-Stokes limit of the BGK equation, there is a relatively small number of numerical works; we refer to \cite{K992} and the recent work in \cite{2ndPaper}. This stands in contrast to the extensive literature available on the numerical low-Mach limit of compressible fluid dynamics, such as \cite{Klein,DegondTang2011,HuangXingXiong2025,DimarcoLoubereVignal2017}. A powerful framework to address these multiscale challenges is the micro-macro decomposition, a mathematical technique used to reformulate the kinetic distribution function by explicitly separating the equilibrium and non-equilibrium dynamics \cite{BennouneLemou,LemouMieussens}. This partition is particularly effective in the low-Mach limit, allowing the pressure term to emerge naturally as a Lagrange multiplier for the incompressibility constraint while maintaining a clear description of the viscous and thermal dissipation terms \cite{CrouseillesLemou}. By focusing on the evolution of these decomposed components, it is possible to design numerical methods that transition uniformly from the rarefied to the continuum limit without a loss of accuracy or stability.

Among the various strategies to implement AP methods for these systems, Implicit-Explicit (IMEX) Runge-Kutta schemes have emerged as a powerful tool for managing multiscale stiffness \cite{AscherRuuthSpiteri1997, DimarcoAP1,DimarcoAP2,IMEX_PR,Ruuth,Multis,Bosc1,BoscarinoPareschiRusso2024}. The core idea behind IMEX methods is to split the governing equations into stiff and non-stiff components. In the context of relaxation systems, the stiff collision/relaxation terms are treated implicitly to ensure stability without the need for excessively small time steps, while the non-stiff transport terms are handled explicitly to maintain computational efficiency and avoid the complexity of solving high-dimensional, non-linear algebraic systems. Furthermore, IMEX methods inherit the AP property, guaranteeing a smooth transition from the kinetic to the hydrodynamic regime without a loss of accuracy \cite{Bosc2,BOSCARINO,Zeifang2020,BoscheriDimarcoLoubere2020}. These time-stepping strategies, when combined with high-resolution spatial discretizations---such as weighted essentially non-oscillatory (WENO) \cite{Shu} or central difference schemes---enable the simulation of complex flows across a wide range of Knudsen and Mach numbers with a computational effort comparable to that of traditional fluid solvers.

The present work complements the discrete-velocity formulation developed in \cite{1stPaper2025} and the full kinetic micro--macro AP framework introduced in \cite{2ndPaper}. 
First, we derive a reduced moment system tailored to the incompressible limit and verify its structural stability conditions. Unlike traditional LB-based discrete-velocity models that are typically restricted in spatial dimension \cite{1stPaper2025}, the present formulation applies to arbitrary spatial dimensions $d$. By avoiding a full discretization of the velocity space, it achieves high computational efficiency, albeit with accuracy restrictions in the highly rarefied regime. 
%Second, adapting the AP framework from \cite{2ndPaper}, we develop an efficient simulation method using a micro-macro decomposition combined with an arbitrarily high-order AP numerical scheme based on IMEX Runge-Kutta time integrators, WENO spatial reconstructions, and high-order central difference approximations. In the asymptotic limit, these methods reduce to high-order projection methods for the INS equations, solved on a collocated grid.
Second, we describe a computationally efficient framework for the simulation of the moment system in the low-Mach and small-Knudsen number limits using a micro-macro decomposition. Finally, building on earlier ideas \cite{2ndPaper}, we develop an arbitrarily high-order AP numerical scheme based on IMEX Runge-Kutta time integrators, WENO spatial reconstructions, and high-order central difference approximations, enabling robust simulations across a wide range of physical regimes. In the asymptotic limit, the proposed methods become high-order projection methods for the INS equations, solved on a collocated grid.

%%%%%%%%%%%%%%
The remainder of the paper is structured as follows. Section~\ref{sec2} details the mathematical derivation of the relaxation system, starting from the scaled BGK model and utilizing a micro-macro decomposition to identify the governing moment equations. Section~\ref{sec3} is dedicated to the structural stability analysis of the system, where we prove its symmetrizability and verify the dissipative properties of the source terms. Section~\ref{sec4} introduces the first-order AP numerical algorithm and describes the semi-implicit time discretization designed to handle the low-Mach stiffness. Section~\ref{sec5} extends the numerical framework to higher-order time discretizations using the IMEX Runge-Kutta approach. Section~\ref{sec6} describes the spatial discretization strategy, including the implementation of finite difference and WENO schemes. Section~\ref{sec7} presents a comprehensive series of numerical tests in both 1D and 2D. We provide a detailed comparison between our relaxation model and full kinetic simulations, highlighting the accuracy of the proposed method and the substantial gains in computational efficiency. Finally, Section~\ref{sec8} draws concluding remarks and outlines directions for future research.

\section{Kinetic equations and the low-Mach limit}\label{sec2}
%We use the notation from \cite{SR} and start with the nonlinear kinetic BGK model in $\R^d$
Adopting the notation from \cite{Saint2003}, we begin with the nonlinear kinetic BGK model in $\R^d$,
\begin{align}
	\label{BGKoriginal}
\partial_t F +  \frac{1}{\epsilon} v \cdot \nabla_x F= \frac{1}{\epsilon^2} C(F)
\end{align}
%with 
where the collision operator is given by
$$
C(F) = \frac{1}{\tau}  ( M_F -F),
$$
%where 
and $M_F$ represents the local Maxwellian distribution:
\begin{align}
	\label{Maxwelloriginal}
	M_F= \frac{\rho_F}{(2 \pi T_F)^{\frac{d}{2}}}exp(- \frac{\vert v-u_F \vert^2}{2 T_F}) =M_{(\rho_F,u_F,T_F)}.
\end{align}
%with 
The macroscopic quantities (density, velocity, and temperature) are determined by the velocity moments:
$$
\rho_F = \int F(x,v ) dv , \;  \rho_F u_F = \int v F(x,v ) dv , \; \rho_F u_F^2 + d \rho_F T_F = \int \vert v \vert^2 F(x,v ) dv .
$$
%Following \cite{SR,Levermore91} and using the expansion 
Following the standard scaling analysis \cite{Saint2003,Levermore91}, we consider small perturbations around a global equilibrium $M$ (with unit density, zero mean velocity, and unit temperature) via the ansatz
$$
F= M(1+\epsilon f).
$$
%where  $M$ is the standard Mawellian with density $1$, mean zero and variance $1$,
%we have
This leads to the scaled kinetic equation for the perturbation $f$:
\begin{align}
\partial_t  f +  \frac{1}{\epsilon}  v \cdot \nabla_x  f = \frac{1}{\epsilon^3} \frac{1}{M}C(M(1+\epsilon f)).
\end{align}
%Note that 
The underlying macroscopic variables are related to $f$ through:
\begin{align*}
\rho_F = 1+ \epsilon <f>, \; \rho_F u_F = \epsilon < vf>, \; \rho_F \vert u_F\vert^2 + d \rho_F T_F= d+ \epsilon < \vert v \vert^2 f>
\end{align*}
with
\[
\langle f \rangle = \langle f \rangle_{L^2(M dv )} = \int_{\mathbb{R}^d} f(v) M(v) \, dv.
\]
%Therefore, 
Accordingly, we can express the macroscopic fluctuations as:
\begin{align}
	\label{relation}
\rho_F = 1+ \epsilon \rho, \; \rho_F u_F = \epsilon u, \;  T_F= 1+ \epsilon \frac{T}{1+\epsilon \rho} - \epsilon^2 \frac{\vert u\vert^2}{d(1+\epsilon \rho)^2}=1+ \epsilon T + \mathcal{O}(\epsilon^2)
\end{align}
using the notation 
$  \rho = <f>  $, $ u = <v f>$ , $  T = <\frac{\vert v\vert^2-d}{d}f> $. 
%Then,  using the  projection
By defining the orthogonal projection operator onto the space collision invariants,
\begin{align}
	\label{projection}Pf=  \rho +  v \cdot u+ \frac{\vert v\vert^2-d}{2} T 
\end{align}
% a Taylor expansion of $C$ gives 
a Taylor expansion of the collision operator $C$ around the global Maxwellian reveals:
\begin{align}
	\label{expansion}
\frac{1}{M} C(M(1+\epsilon f)) = \epsilon L f + \epsilon^2 Q(f) 
\end{align}
where
$$
Lf = L(I-P)f= \frac{1}{\tau}(P-I) f \; 
$$
%and $Q$ a remainder term. For the following we do not need details on $Q$, but note that $P Q(f) =0$. 
and $Q(f)$ denotes a higher-order remainder satisfying $P Q(f) = 0$.
%One obtains the following scaled kinetic equation for $f$
Consequently, the scaled BGK equation takes the form:
\begin{align}
\label{kinscaled}
\partial_t f +  \frac{1}{\epsilon} v \cdot \nabla_x f= \frac{1}{\epsilon^2} Lf+ \frac{1}{ \epsilon } Q (f).
\end{align}

%One observes \cite{Levermore91,SR}, that for $\epsilon \rightarrow 0$ the distribution function $f$ converges towards 
%a function $f_0$ of the form
As established in \cite{Levermore91,Saint2003}, in the limit $\epsilon \rightarrow 0$, the perturbation $f$ approaches a local equilibrium state $f_0$:
%The solvability conditions for these equations yield
\[
f_0 = \rho + v \cdot u + \frac{|v|^2 - d}{2} T.
\]
%Moreover, compare \cite{SR}, the moments of the remainder term $Q (f)$ with respect to the polynomials 
Furthermore, the moments of the non-linear remainder $Q(f)$ evaluated against the traceless tensor polynomial
$$
\mathcal{A}(v) = v \otimes v -\frac{\vert v \vert^2 }{d}I
$$
and the vector polynomial
$$
\mathcal{B}(v) = v (\vert v \vert^2 -(d+2))
$$
%converge for $\epsilon \rightarrow 0$ as
exhibit the asymptotic behaviour \cite{Saint2003}:
%<B(v) Q^\epsilon>\rightarrow  \frac{1}{\tau} (d+2) (u T)
%$$
%and
\begin{align}
\label{convQA}
<\mathcal{A}(v) Q (f)>\xrightarrow{\epsilon \rightarrow 0}   \frac{1}{\tau} \mathcal{A}(u)
\end{align}
and 
\begin{align}
\label{convQB}
<\mathcal{B}(v) Q (f)>\xrightarrow{\epsilon \rightarrow 0}   \frac{1}{\tau} (d+2) (u T).
\end{align}

Taking the limit $\epsilon \to 0$ rigorously (see \cite{Saint2003} for the BGK operator and \cite{GSR} for the Boltzmann setting), one recovers the incompressible Navier-Stokes equations with Boussinesq coupling for density and temperature:
%\begin{align}
%\label{limit0}
%\nabla_x \cdot  u =0 \\
%\nabla_x (\rho+T) =0 \nonumber\\
%%\partial_t \rho + \nabla\cdot \bar u =0 \nonumber \\
%\partial_t u + \nabla_x \cdot u \otimes u + \nabla_x \tilde{p}= \tau  \Delta_x  u\nonumber\\
%%\partial_t T +\nabla_x \cdot< v \frac{v^2-3}{3} f_1 >= 0\\
%\partial_t T +\nabla_x \cdot (uT) =  \tau \Delta_x T,\nonumber
%\end{align}
%where $\tilde{p} = \frac{1}{\epsilon}(\rho + T)$ acts as the hydrodynamic pressure.

\begin{align}
\label{limit0}
\nabla_x \cdot  u &= 0 \\
\nabla_x (\rho+T) &= 0 \nonumber\\
%\partial_t \rho + \nabla\cdot \bar u &=0 \nonumber \\
\partial_t u + \nabla_x \cdot u \otimes u + \nabla_x \Pi &= \tau \Delta_x u \nonumber\\
%\partial_t T +\nabla_x \cdot< v \frac{v^2-3}{3} f_1 > &= 0\\
\partial_t T +\nabla_x \cdot (uT) &= \tau \Delta_x T, \nonumber
\end{align}
where $\Pi = \tilde{p} - \frac{|u|^2}{d}$ acts as the effective hydrodynamic pressure, absorbing the dynamic pressure contribution, and $\tilde{p} = \frac{1}{\epsilon}(\rho + T)$ represents the thermodynamic pressure fluctuation.
%In passing we note that a rigorous proof of convergence for the BGK collision operator  has been obtained in \cite{SR} and for the full Boltzmann collision operator in \cite{GSR}.
\subsection{Derivation of the relaxation system}
	
%Following the procedure in \cite{2ndPaper}, we start by rewriting the scaled kinetic equation separating macroscopic and remaining kinetic terms
%	using a micro-macro decomposition.
%	We decompose $f$ in a macroscopic and a microscopic  part using the projection $P$ given, for example, in (\ref{projection}) and 
% $f = P f + (I-P)f$.
 To establish the relaxation model, we apply a micro-macro decomposition strategy similar to \cite{2ndPaper}. We partition the distribution function into its hydrodynamic component $Pf$ and kinetic fluctuation $(I-P)f$. Applying $P$, given in (\ref{projection}), and $(I-P)$ to the scaled kinetic equation and utilizing $PL(f) = PQ(f) = 0$ yields:
%We use  again the notation
% $\rho = <f> = \int f M d v , u = <v f> $ and $T = <\frac{\vert v\vert^2-d}{d}f> $. Moreover, define
% $S = <\vert v\vert^2-(d+2)f> =d T -2 \rho$.
%
%The corresponding   projection  of the kinetic equations yields 
\begin{align}
	\label{pro1}
&&\partial_t Pf + \frac{1}{\epsilon}\nabla \cdot P(v  f )&=  0\\
\label{pro2}
&& \partial_t (I-P)f+  \frac{1}{\epsilon}\nabla_x \cdot  (I-P) (v  f)&=\ -\frac{1}{\epsilon^2 \tau }(I-P) f+ \frac{1}{\epsilon}Q(f)
\end{align}
%using   $PL(f)=PQ(f)=0$.
%For further use we collect some properties of the polynomials
%$\mathcal{A}$ and $\mathcal{B}$.
%%and 
%%$$
%%B(v) = v (\vert v \vert^2 -(d+2)).
%%$$
%That is first, that $\mathcal{A}$ and $\mathcal{B}$ are orthogonal polynomials in the orthogonal complement of the collision invariants.
%Moreover,
Before proceeding, we summarize essential orthogonality and tensor algebraic properties of $\mathcal{A}$ and $\mathcal{B}$. Both polynomials reside in the orthogonal complement of the kernel of $L$. Furthermore:
 $$
< \mathcal{B}(v) \otimes \mathcal{B}(v) > = 2 (d+2) I ,
$$
%and for $\Theta \in \R^d \times \R^d$
%$$
%<\mathcal{A}(v) \otimes \mathcal{A}(v) >  :\Theta = \Theta + \Theta^T  - \frac{2}{d}  tr(\Theta) I = \mathcal{D} (\Theta)
%$$
%or especially
$$
<\A(v) \otimes \A(v) >  : u \otimes u = 2 \A(u)  
$$
and
$$
<\A(v) \otimes \A(v) >  : \nabla_x u  = \nabla_x u +(\nabla_x u )^T  -  \frac{2}{d} \nabla_x \cdot  u I = \mathcal{D} (\nabla_x u).
$$
%A direct calculation gives 
Taking the divergence gives:
$$
\nabla_x \cdot (<A(v) \otimes A(v) >  : \nabla_x u)  = \Delta_x u   +  \frac{d-2}{d} \nabla_x (\nabla_x \cdot  u).
$$
%Moreover,
%$$
%<B(v) Q>= \frac{1}{\tau} (d+2) (u T)
%$$
%and
%$$
%<A(v) Q>=  \frac{1}{\tau} A(u)
%$$
%\begin{remark}
%	Note that this derivation corresponds to a fluid with compressible viscosity coefficients where the shear viscosity is $\mu = 1$ and the second viscosity coefficient is $\lambda = -2/d$. Under these conditions, the bulk (volume) viscosity is zero. To consider more complex physical situations, such as polyatomic gases with non-zero bulk viscosity, a suitably modified BGK model must be employed.
%\end{remark}

%Rewrite now equation (\ref{pro1}) using  moments of $f$ 
%as
We now recast the $Pf$ equation \eqref{pro1} in terms of macroscopic moments:
%\begin{align*}
%\partial_t <f> + \frac{1}{\epsilon}\nabla_x \cdot <v  f> =0 \\
%\partial_t <vf> + \frac{1}{\epsilon}\nabla_x \cdot <v \otimes v f> = 0\\
%\partial_t <(\vert v\vert^2-(d+2))f> +\frac{1}{\epsilon}\nabla_x \cdot< v (\vert v\vert^2-(d+2)) f >= 0\\
%\end{align*}
%or
%gives
%\begin{align*}
%\partial_t \rho + \nabla_x \cdot <v  f_1> =0 \\
%\partial_t u + \nabla_x \cdot <v \otimes v f_1> = 0\\
%\partial_t <\frac{v^2-5}{2}f_0> +\nabla_x \cdot< A(v) f_1 >= 0\\
%\end{align*}
%Rewriting gives
%\begin{align*}
%\partial_t \rho + \nabla_x \cdot \hat u =0 \\
%\partial_t u + \nabla_x \cdot <A(v)  f_1> + \nabla_x p= 0\\
%%\partial_t T +\nabla_x \cdot< v \frac{v^2-3}{3} f_1 >= 0\\
%\partial_t S +\nabla_x \cdot< B(v) f_1 >= 0
%\end{align*}
%
%Then
%\begin{align*}
%\partial_t \rho + \nabla_x \cdot \hat u =0 \\
%\partial_t u + \nabla_x \cdot <L^{-1}(B(v))  L f_1> + \nabla_x p = 0\\
%%\partial_t T +\nabla_x \cdot< v \frac{v^2-3}{3} f_1 >= 0\\
%\partial_t S +\nabla_x \cdot< L^{-1}(A(v)) L f_1 >= 0
%\end{align*}
 %for the BGK model
\begin{align}
%\label{Pf0}
\epsilon\partial_t \rho + \nabla_x \cdot  u =0 \\
\epsilon\partial_t u + \nabla_x \cdot <\mathcal{A}(v)  (I-P) f> + \nabla_x (T+\rho) = 0 \nonumber\\
%\partial_t T +\nabla_x \cdot< v \frac{v^2-3}{3} f_1 >= 0\\
\epsilon \partial_t S +\nabla_x \cdot< \mathcal{B}(v) (I-P) f >= 0\nonumber
\end{align}
with $S = <\vert v\vert^2-(d+2)f> =d T -2 \rho$. Re-expressing the equation for $S$ yields:
\begin{align}
\epsilon \partial_t  T + \frac{2}{d}	\nabla_x \cdot u  +  \frac{1}{d}\nabla_x \cdot < \B(v) (I-P)f  >= 0
\end{align}

%Note that $S \rightarrow (d+2) T$ for $\epsilon \rightarrow 0$.
%The $(I-P)f_0$ equation gives 
%\begin{align}
%%\label{irrelevant}
%\partial_t (I-P) f_0 + \nabla_x \cdot (I-P)(v  f_1) \\
%= -\frac{1}{\epsilon^2} (I-P)f_0+\frac{1}{\epsilon } [Q(f_0+\epsilon f_1)-Q(f_0)]
%\nonumber
%\end{align}
%Note that $(I-P)f_0 =  \mathcal{O}(\epsilon^2)$.
%
%
%The $Pf_1$ equation gives 
%\begin{align*}
%\partial_t < f_1> + \frac{1}{\epsilon^2} \nabla_x \cdot <v    f_0> =0 \\
%\partial_t <v f_1> + \frac{1}{\epsilon^2} \nabla_x \cdot <v \otimes v   f_0> =0 \\
%\partial_t <\frac{v^2}{d} f_1> + \frac{1}{\epsilon^2} \nabla_x \cdot <v\frac{v^2}{d}   f_0> =0
%\end{align*}
%or 
%\begin{align}
%\partial_t \hat \rho + \frac{1}{\epsilon^2} \nabla_x \cdot u  =0 \\
%\partial_t \hat u + \frac{1}{\epsilon^2} \nabla_x  (\rho+T)+\frac{1}{\epsilon^2} \nabla_x \cdot <v \otimes v   (I-P) f_0>) =0 \nonumber\\
%\partial_t p+ \frac{1}{ \epsilon^2} \nabla_x \cdot  (\frac{d+2}{d}u+<v\frac{v^2}{d}   (I-P)f_0> ) =0 \nonumber
%\end{align}
%This gives divergence free constraint and Boussinesq in the limit.
%Equation (\ref{pro2}) gives 
Next, the $(I-P)f$ equation \eqref{pro2} is rewritten as:
\begin{align}
\epsilon^2 \partial_t (I-P)f+  \epsilon \nabla_x \cdot  (I-P) (v  f)= -\frac{1}{\tau}(I-P) f+  \epsilon Q(f)
\end{align}

%	We evaluate the second term on the left as
Splitting the advection term into macroscopic and microscopic contributions gives:
\begin{align*}
(I-P)  (v f) = (I-P)(v P(f))+ (I-P)(v (I-P) (f))
\end{align*}
%with 
where the macro contribution simplifies directly to:		
\begin{align*}
(I-P)(v P(f))%=	(I-P)(v [<f> +  v <v f >+ \frac{v^2-d}{2} <\frac{v^2-d}{d}f>])\\
%=(I-P)(v <f> +  v \otimes v  <v f >+ v \frac{v^2-d}{2} <\frac{v^2-d}{d}f>)\\
%= (I-P)(  (A(v) + \frac{\vert v \vert^2}{d} ) <v f >+ (\frac{1}{2} B(v) + v )  <\frac{v^2-d}{d}f>)\\
%= (I-P)(  A(v)  <v f >+ \frac{B(v)}{2}   <\frac{v^2-d}{d}f>)\\
=   \A(v)  <v f >+ \frac{\B(v)}{2}   <\frac{v^2-d}{d}f>
\end{align*}
by using the properties of the polynomials $\A$ and $\B$.	
%Moreover, we compute 
%\begin{align*}
%(I-P)(v (I-P) f=	v (I-P)(f) - P(v(I-P )f )  \\
%%=
%%v (I-P)(f) - [<v (I-P)f> +  v <v \otimes v (I-P)f >+ \frac{v^2-d}{2} <v \frac{\vert v\vert^2-d}{d}(I-P)f>]\\
%%=v (I-P)(f) -  <(v \otimes v - \frac{\vert v \vert^2}{d} I )(I-P)f > v \\- \frac{\vert v \vert^2-d}{2} <v \frac{\vert v \vert^2-(d+2)}{d}(I-P)f>\\
%=v (I-P)(f) -  <A(v)(I-P)f > v - \frac{\vert v \vert^2-d}{2 d } <B(v)(I-P)f>\\
%\end{align*}			
%and therefore 	alltogehter 		
%\begin{align*}
%(I-P)  (v f) = A(v)  u + \frac{B(v)}{2}   T +  v (I-P)(f) \\-  <A(v)(I-P)f > v - \frac{\vert v \vert^2-d}{2 d } <B(v)(I-P)f>
%\end{align*}
%This gives the equation
This allows us to write:
\begin{align*}
\epsilon^2 \partial_t (I-P)f+  \epsilon \nabla_x \cdot  [\A(v)  u + \frac{\B(v)}{2}   T +  (I-P)(v (I-P) f)]= -\frac{1}{\tau}(I-P) f+  \epsilon Q
 \end{align*}

Recognizing that $(I-P)f = \mathcal{O}(\epsilon)$ in the diffusive regime, we introduce the rescaled micro-variable $g = \frac{1}{\epsilon}(I-P)f$, obtaining:
\begin{align*}
\epsilon^2 \partial_t g+   \nabla_x \cdot  [\A(v)  u + \frac{\B(v)}{2}   T + \epsilon (I-P)(v g)]= -\frac{1}{\tau} g +   Q.
\end{align*}
Collecting these results, we obtain the full macro-micro system, remaining formally equivalent to the original kinetic equations (\ref{kinscaled}):
\begin{align}
\label{basic}
\epsilon\partial_t \rho + \nabla_x \cdot  u =0 \\
\epsilon\partial_t u + \epsilon \nabla_x \cdot <\A(v)  g> + \nabla_x (T+\rho) = 0 \nonumber\\
\epsilon \partial_t  T + \frac{2}{d}	\nabla_x \cdot u  + \epsilon  \frac{1}{d}\nabla_x \cdot < \B(v) g >= 0
%\partial_t T +\nabla_x \cdot< v \frac{v^2-3}{3} f_1 >= 0\\
% \partial_t S +\nabla_x \cdot< B(v) g >= 0\nonumber
\end{align}

\begin{align}
\label{g}
	\epsilon^2 \partial_t g+   \nabla_x \cdot  [\A(v)  u + \frac{\B(v)}{2}   T + \epsilon (I-P)(v g)]= -\frac{1}{\tau} g +   Q.
\end{align}

The final  relaxation system  is now obtained using the following closure procedure.
Adapted to the incompressible limit, we neglect  the term $\epsilon (I-P)(v g)$ in (\ref{g}), integrate  the equation for $g$ with respect to the polynomials $\A(v)$
and $\B(v)$ and approximate the moments of 
$Q (f)$ with respect to these polynomials by their limits for $\epsilon \rightarrow 0$ as given in (\ref{convQA}) and (\ref{convQB})

%all kinetic terms outside  the subspaces  spanned by the  collision invariants $1,v,v^2$ and  by the polynomials $A(v)$
%and $B(v)$. 
%Considering the above equations with the LBM-type collision operator and integrating the $g$-equation wrto
%$A(v)$ and $B(v)$  
Denoting the moments $\Theta = <\A(v) g>$ and $\Gamma = <\B(v ) g>$ and using the above approximations and properties of
$\A$ and $\B$ gives
 the hyperbolic relaxation system 
\begin{align}
	\label{relax1}
	\epsilon \partial_t \rho + \nabla_x \cdot  u  =0 \\
	\epsilon \partial_t u + \epsilon \nabla_x \cdot  \Theta  + \nabla_x  (\rho+T) = 0
	\label{relax2}\\
	%\partial_t T +\nabla_x \cdot< v \frac{v^2-3}{3} f_1 >= 0\\
	\partial_t  T + \frac{2}{\epsilon d}	\nabla_x \cdot u  +  \frac{1}{d}\nabla_x \cdot\Gamma= 0 \label{relax3}
\end{align}
and
\begin{align}
	\label{relax4}
	\epsilon^2 \partial_t \Theta+  \D(\nabla_x u ) &=\ - \frac{1}{\tau}  (\Theta -\A(u)) \\
\label{relax5}	\epsilon^2 \partial_t \Gamma+  (d+2) \nabla_x T&=\ -  \frac{1}{\tau} (\Gamma  - (d+2) u  T) 
\end{align}

We have obtained a $(d+2)+d+ (\frac{d(d-1)}{2}+d-1)=\frac{d^2}{2}+\frac{5d}{2} +1$ moment model for $\rho,u,T,\Theta,\Gamma$. That means,
we obtain  in 1-D a $4$ moment model, in  2-D a $8$ moment model and in 3D a $13$ moment model
(note that $\Theta$ is symmetric with trace zero).
In the following we investigate this system analytically and numerically discussing structural stability conditions for the system in Section 
\ref{strucstab} and presenting an asymptotic preserving numerical scheme for the system in Sections  \ref{ap}, \ref{ho} and \ref{spat}.

\begin{remark}
Note that hyperbolicity of the system follows directly from symmetrizability, which is proven in the next section.
The eigenvalue structure is simple. We have, for example, in 1D two positive and two negative simple eigenvalues
and  in 2D a double zero eigenvalue, a double positive and negative  as well as  a simple positive and negative  eigenvalue,

\end{remark}

\begin{remark}
Comparing the hyperbolic main part of our system to the one of the  linear Grad moment system 	\cite{struchtrup2003regularization, sarna2020convergence}, we observe, that they are equivalent in 1D, but not the same in higher dimensions. 
In contrast to the linear Grad approach the present approach is not a full expansion in terms of Hermite polynomials to the corresponding order. Instead, by neglecting the term $\epsilon (I-P)(v g)$, we have  neglected terms in the Hermite expansion. 
In particular, this results in missing terms in equations (\ref{relax4}) and (\ref{relax5}) in higher dimensions.
In the following we show analytically and numerically,  that the system is still appropriate as a relaxation system for the incompressible Navier-Stokes limit, for appropriately small values of the scaling parameter $\epsilon$.
\end{remark}

\begin{remark}
	Comparing the  system to the ones derived from discrete velocity models as in the Lattice-Boltzmann method,
	see \cite{JUNK,ChenDoolen1998,1stPaper2025},  one  observe the following. Again,  in 1D,   one can rewrite the model as a discrete velocity model with Gauss-Hermite points as discrete velocities. In contrast, this is not possible any more in higher dimensions.
	This can be easily seen, for example, in 2D, by observing, that the system matrices are not any more simultaneously diagonalisable, since they are not commuting.  This  would be the case, if they would have been derived by  a one-to-one relation of  discrete velocity and moment variables.
	This behaviour is shared with the linear Grad moment system.
\end{remark}

\section{Stability of the relaxation system}\label{sec3}
\label{strucstab}
In this section we discuss structural  stability properties of the relaxation system in the limit as $\epsilon \rightarrow  0$.
With $\tau=1$ the original model reads
\begin{align*}
	\epsilon \partial_t \rho + \nabla_x\cdot u&=0\\[2mm]
	\epsilon \partial_t u + \epsilon \nabla_x \cdot \Theta +\nabla_x(\rho+T)&=0\\[2mm]
	\partial_t T + \frac{2}{\epsilon d} \nabla_x\cdot u+\frac{1}{d}\nabla_x\cdot \Gamma &= 0\\[2mm]
	\epsilon^2 \partial_t \Theta + \nabla_x u + (\nabla_x u)^T - \frac{2}{d} \nabla_x\cdot u I & = -\Theta+u\otimes u -\frac{|u|^2}{d}I\\[2mm]
	\epsilon^2 \partial_t \Gamma + (d+2)\nabla_x T & = -\Gamma+(d+2)uT.
\end{align*}
By introducing $\widetilde{\Gamma}=\epsilon \Gamma$ and 
$\widetilde{\Theta}=\epsilon \Theta$ and dividing the equations by $\epsilon$, we obtain
\begin{align}
	\partial_t \rho + \frac{1}{\epsilon}\nabla_x\cdot u&=0 \label{eq1}\\[2mm]
	\partial_t u + \frac{1}{\epsilon} \nabla_x \cdot \widetilde{\Theta} + \frac{1}{\epsilon} \nabla_x(\rho+T)&=0\label{eq2}\\[2mm]
	\partial_t T + \frac{2}{\epsilon d} \nabla_x\cdot u+\frac{1}{\epsilon d}\nabla_x\cdot \widetilde{\Gamma} &= 0\label{eq3}\\[2mm]
	\partial_t \widetilde{\Theta} + \frac{1}{\epsilon}\left(\nabla_x u + (\nabla_x u)^T - \frac{2}{d} \nabla_x\cdot u I\right) & = \frac{1}{\epsilon^2}\left(-\widetilde{\Theta}+ \epsilon u\otimes u - \epsilon \frac{|u|^2}{d}I\right)\label{eq4}\\[2mm]
	\partial_t \widetilde{\Gamma} + \frac{d+2}{\epsilon}\nabla_x T & = \frac{1}{\epsilon^2}\Big(-\widetilde{\Gamma}+\epsilon(d+2)uT\Big).\label{eq5}
\end{align}
For simplicity, we drop $\widetilde{~}$ ~in what follows. 

Note that $\Theta(x,t)$ is symmetric and trace free if the initial data $\Theta(x,0)$ is so. We denote the variables $(\rho,u,T,\Theta,\Gamma)$ as $U$ and rewrite the system in matrix form as 
\begin{align}
	\label{relaxcompact}
	U_t+\frac{1}{\epsilon}\sum_{j=1}^dA_jU_{x_j}=\frac{1}{\epsilon^2}Q(\epsilon,U).
\end{align}
In the following we investigate the system and prove that the stability conditions from \cite{Yong99,Cor-Yong,Peng}
are satisfied.
For the right-hand side, we have
$$
Q(\epsilon,U)=\begin{pmatrix}
	0_{(d+2)\times1} \\[2mm]
	-{\Theta}+ \epsilon u\otimes u - \epsilon \dfrac{|u|^2}{d}I\\[3mm]
	-{\Gamma}+\epsilon(d+2) uT
\end{pmatrix}. 
$$
Clearly, the source term satisfies
$$
Q(0,U)=\begin{pmatrix}
	0\\
	q(U)
\end{pmatrix}\qquad q(U)=-\begin{pmatrix}
	\Theta\\
	\Gamma
\end{pmatrix}.
$$
Accordingly, we write $U^I=(\rho,u,T)$ and $U^{II}=(\Theta,\Gamma)$. Then we check that
$$
q(U)=0\Longleftrightarrow U^{II}=0,\qquad \frac{\partial q(U^I,0)}{\partial U^{II}}=-I.
$$ 
%Therefore the condition (ii) is satisfied. 

For the left-hand side, we denote $A(\xi)=\sum_{j=1}^d\xi_jA_j$ with $\xi=(\xi_1,\xi_2,\dots,\xi_d)^T\in\mathbb{R}^d$ and try to write the expression of $A(\xi)$. From \eqref{eq1}--\eqref{eq5}, we have 
\begin{align*}
	A(\xi)=\begin{pmatrix}
		~0_{1\times 1} & ~\xi^T & 0 & 0 & ~0 \\[2mm]
		~\xi & ~0_{d\times d} & \xi & C & ~0 \\[2mm]
		~0 & ~\frac{2}{d}\xi^T & 0_{1\times 1} & 0 & ~\frac{1}{d}\xi^T \\[2mm]
		~0 & ~D & 0 & 0_{d^2\times d^2} & ~0 \\[2mm]
		~0 & ~0 & (d+2)\xi & 0 & ~0_{d\times d} 
	\end{pmatrix}
\end{align*}
with
\begin{align*}
	C=\begin{pmatrix}
		\xi^T & 0 & 0 & 0 \\[2mm]
		0 & \xi^T & 0 & 0\\[1mm]
		0 & 0 & \ddots & 0\\[2mm]
		0 & 0 & 0 & \xi^T
	\end{pmatrix}\in \mathbb{R}^{d\times d^2}.
\end{align*}
\begin{remark}
	Note that due to the symmetric and trace-free property, the expression for $C$ is not unique.  
\end{remark}

In order to express $D$, we write 
$$
\nabla_x u + (\nabla_x u)^T - \frac{2}{d} \nabla_x\cdot u I = \sum_{k=1}^d\mathcal{U}^{(k)}
$$
with 
$$
\mathcal{U}^{(k)}_{ij} = \left\{
\begin{array}{lc}
	(2-\dfrac{2}{d})\partial_{x_k}u_k  &  \quad i=j=k\\[4mm]
	-\dfrac{2}{d}\partial_{x_k}u_k  &  \quad i=j\neq k\\[5mm]
	~\partial_{x_j}u_k &  \quad i\neq j,~i=k\\[3mm]
	~\partial_{x_i}u_k&  \quad i\neq j,~j=k.
\end{array}
\right.
$$
The other components are zero.
Then we express the $((i-1)\times d+j)$-th row and the $k$-th column component of $D$ by
$$
D_{(i-1)\times d+j,k} = \left\{
\begin{array}{lc}
	(2-\dfrac{2}{d})\xi_k  &  \quad i=j=k\\[5mm]
	-\dfrac{2}{d}\xi_k  &  \quad i=j\neq k\\[5mm]
	\xi_j &  \quad i\neq j,~i=k\\[3mm]
	\xi_i &  \quad i\neq j,~j=k.
\end{array}
\right.
$$
For example, when $d=2$ we have 
$$
D^T=\begin{pmatrix}
	\xi_1 & \xi_2 & \xi_2 & -\xi_1 \\[1mm]
	-\xi_2 & \xi_1 & \xi_1 & \xi_2
\end{pmatrix}
$$
and when $d=3$ we have 
$$
D^T=
\begin{pmatrix}
	\frac{4}{3}\xi_1 & \xi_2 & \xi_3 & \xi_2 & -\frac{2}{3}\xi_1 & 0 & \xi_3 & 0 & -\frac{2}{3}\xi_1 \\[2mm]
	-\frac{2}{3}\xi_2 & \xi_1 & 0 & \xi_1 & \frac{4}{3}\xi_2 & \xi_3 & 0 & \xi_3 & -\frac{2}{3}\xi_2 \\[2mm]
	-\frac{2}{3}\xi_3 & 0 & \xi_1 & 0 & -\frac{2}{3}\xi_3& \xi_2 & \xi_1 & \xi_2 & \frac{4}{3}\xi_3 
\end{pmatrix}.
$$
% Denote the $((i-1)\times d+j)$-th row and the $k$-th column component of $2C^T$ by $C_{ijk}$. Then we have
% $$
% C_{ijk} + C_{jik} = \left\{
% \begin{array}{lc}
	%    2\xi_j  &  \quad i=k\\[3mm]
	%    0  &  \quad i\neq k.
	% \end{array}
% \right.
% $$
% Due to the symmetric property, we can introduce $\widetilde{C}_{ijk}$ such that 
% $$
% \widetilde{C}_{ijk} = {C}_{ijk} - a \qquad \widetilde{C}_{jik} = {C}_{jik} + a
% $$ 
% for any constant $a$ and $i\neq j$.
% On the other hand, using the trace-free property, we write 
% $$
% \widetilde{C}_{iik} = {C}_{iik} - a \qquad \widetilde{C}_{jjk} = {C}_{jjk} + a_j ~~\text{for all $j\neq i$ such that}~~a=\sum_{j\neq i}a_j. 
% $$ 
Note that due to the symmetric and trace-free property, 
we may also replace the matrix $C$ by $\widetilde{C}$ such that
$$
\widetilde{C}=\frac{1}{2}D^T.
$$
For example, in the 2D case, we can use the symmetric and trace-free property $\widetilde{\Theta}_{12}=\widetilde{\Theta}_{21}$ and $\widetilde{\Theta}_{11}=-\widetilde{\Theta}_{22}$ to rewrite the term 
	$$
	\nabla_x\cdot \widetilde{\Theta}=\begin{pmatrix}
		\partial_{x_1}\widetilde{\Theta}_{11} + \partial_{x_2}\widetilde{\Theta}_{21}\\[1mm]
		\partial_{x_1}\widetilde{\Theta}_{12} + \partial_{x_2}\widetilde{\Theta}_{22}
	\end{pmatrix}=\begin{pmatrix}
		\frac{1}{2}\partial_{x_1}\widetilde{\Theta}_{11} +\frac{1}{2}\partial_{x_2}\widetilde{\Theta}_{21}+\frac{1}{2}\partial_{x_2}\widetilde{\Theta}_{12}- \frac{1}{2}\partial_{x_1}\widetilde{\Theta}_{22}\\[2mm]
		-\frac{1}{2} \partial_{x_2}\widetilde{\Theta}_{11}+\frac{1}{2}\partial_{x_1}\widetilde{\Theta}_{21} + \frac{1}{2}\partial_{x_1}\widetilde{\Theta}_{12} + \frac{1}{2}\partial_{x_2}\widetilde{\Theta}_{22} 
	\end{pmatrix}.
	$$
	Now we see that the corresponding matrix is $\widetilde{C}=\frac{1}{2}D^T$. The 3D case is similar. We find that $C$ and $D$ occupy symmetric block positions in the matrix $A(\xi)$. Thus this expression of $\widetilde{C}$ is more convenient to construct the symmetrizer. 
Indeed, defining 
$$
A_0=\text{diag}\left(I_{d+1},~\frac{d}{2},~\frac{1}{2}I_{d^2},~\frac{1}{2(d+2)}I_d\right)
$$ 
we  check that $A_0A(\xi)$ is symmetric. Moreover, we have
$$
A_0Q_U(0,U)=-\text{diag}\left(0_{d+2},~\frac{1}{2}I_{d^2},~\frac{1}{2(d+2)}I_d\right).
$$
To summarize, we have shown that the relaxation system (\ref{relaxcompact})
fulfills the following 
 structural stability conditions, see  \cite{Yong99,Cor-Yong,Peng}
\begin{itemize}
	\item[(i)] There exists a symmetric positive definite matrix $A_0$, called symmetrizer, such that for all $j=1,2,\dots,d$,
	$$
	A_0A_j=A_j^TA_0.
	$$
	\item[(ii)] The source term satisfies 
	$$
	Q(0,U)=\begin{pmatrix}
		0\\[1mm]
		q(U)
	\end{pmatrix},\qquad q:G\rightarrow \mathbb{R}^r~\text{a smooth function},~1\leq r\leq n.
	$$
	With the same partition, we denote $U=\begin{pmatrix}
		U^I\\
		U^{II}
	\end{pmatrix}$ with $U^I\in \mathbb{R}^{n-r}$ and $U^{II}\in \mathbb{R}^{r}$. The function $q(U)$ satisfies 
	$$
	q(U)=0\Longleftrightarrow U^{II}=0,\qquad \frac{\partial q(U^I,0)}{\partial U^{II}} ~~\text{is invertible for all}~U^I.
	$$ 
	\item[(iii)] The symmetrizer $A_0$ and the source term $Q(U)$ are coupled by the inequality
	$$
	A_0Q_U(0,U^I,0)+Q_U^T(0,U^I,0)A_0 \leq -c_0\begin{pmatrix}
		0 & \\
		& I_r
	\end{pmatrix}.
	$$
\end{itemize}

\begin{remark}
	According to \cite{Peng} a proof of convergence as $\epsilon $ goes to $0$  of the relaxation system to the incompressible Navier-Stokes equations requires besides the above structural stability conditions further  assumptions.  In particular, note that (i) shows that the system is symmetrizable, (ii) is equivalent to
	conditions (1.4) and (1.5) in \cite{Peng} .
	The assumptions (H1) and (H3) in \cite{Peng} are satisfied automatically since all $A_j$ are constant matrices. (iii) is equivalent to condition (H2) in \cite{Peng} . 
	  The missing Assumptions (H4)--(H6) in \cite{Peng}  are technical assumptions on the initial conditions and further conditions on the regularity of the asymptotic expansion, which are not discussed here.
\end{remark}
\begin{remark}
	\label{hyplimit}
	The situation is completely different, if we would consider the relaxation system with $\epsilon=1$ in the limit  $\tau$ going to $0$.
	In this case the corresponding stability conditions are not the same as in the present case and are in general not fulfilled  by the 
    relaxation system derived here. Moreover, it is easy to see, that for such a limit the formal limit equations are not the nonlinear Euler equations and not even always hyperbolic. 
\end{remark}

\section{Asymptotic-preserving algorithms - Time discretization}\label{sec4}
\label{ap}

In the following  an asymptotic preserving algorithm for the relaxation system working uniformly  in the Navier-Stokes limit  is developed and numerically investigated. We consider in this section first-order time discretizations following \cite{2ndPaper}. Higher order approaches and the spatial discretization are described in the next section.
The relaxation system for $\rho,u,T,\Theta,\Gamma$ with the notation  $p=\rho+T$ and $S= dp-(d+2) \rho= dT-2\rho$ is given by 
\begin{align*}
	%\label{Pf0}
	\partial_t \rho + \frac{1}{\epsilon}\nabla_x \cdot  u  =0 \\
% \partial_t p + \frac{d+2}{\epsilon d }\nabla_x \cdot  u  +  \frac{1}{d}\nabla_x \cdot \Gamma=0 \\
\partial_t u  + \frac{1}{\epsilon}\nabla_x  p +  \nabla_x \cdot \Theta   = 0\\
%\partial_t T +\nabla_x \cdot< v \frac{v^2-3}{3} f_1 >= 0\\
\partial_t  T + \frac{2}{\epsilon d}	\nabla_x \cdot u  +  \frac{1}{d}\nabla_x \cdot \Gamma= 0
\end{align*}
and
\begin{align*}
	\epsilon^2 \partial_t \Theta+  \D(\nabla_x u )&=\ - \frac{1}{\tau}  \Theta + \frac{1}{\tau} \A(u) \\
	\epsilon^2 \partial_t \Gamma+  (d+2) \nabla_x T&=\ -  \frac{1}{\tau} \Gamma  + \frac{1}{\tau} (d+2) u  T.
\end{align*}
We consider two  semi-implicit algorithms, treating first  the small Knudsen number and  low Mach number limit implicitly
and second including  an additional implicit discretization of the diffusive terms to take into account the parabolic stiffness of the problem in the limit.

%\subsection{Asymptotic preserving time discretization}

%We consider the following  semi-implicit time discretizations.

\subsection{Semi-implicit discretization for  low Mach and small Knudsen numbers}
Here, only the stiff low-Mach limit and the stiffness in the Knudsen number will be treated implicitly. The scheme for the limit equations is a projection scheme treating the viscous terms explicitly. 
We separate  the implicit and explicit components as follows
\begin{equation}
\begin{split} \label{implicit-explicit}
\partial_t \rho +  \underbrace{\frac{1}{\varepsilon}\nabla_x \cdot u}_{implicit} &= 0 \\
\partial_t u + \underbrace{ \nabla_x \cdot  \Theta + \frac{1}{\varepsilon} \nabla_x p}_{implicit} &= 0 \\
\partial_t T +  \underbrace{\frac{2}{\varepsilon d} \nabla_x \cdot u + \frac{1}{d} \nabla_x \cdot \Gamma}_{implicit} &= 0\\
\partial_t p +  \underbrace{\frac{d+2}{\varepsilon d} \nabla_x \cdot u + \frac{1}{d} \nabla_x \cdot \Gamma}_{implicit} &= 0\\
\varepsilon^2 \partial_t \Theta +  \underbrace{\D(\nabla_x u )}_{explicit} &= -  \underbrace{\frac{1}{\tau} \Theta}_{implicit} + \underbrace{\frac{1}{\tau} \mathcal{A}(u)}_{explicit}\\
\varepsilon^2 \partial_t \Gamma +  \underbrace{(d+2) \nabla_x T}_{explicit} &= -  \underbrace{\frac{1}{\tau} \Gamma}_{implicit} + \underbrace{\frac{1}{\tau} (d+2) uT}_{explicit}
\end{split}
\end{equation}
Thus, for the pressure we have 
\begin{equation}
	\begin{split} 
		\partial_t p +  \underbrace{\frac{d+2}{\varepsilon d} \nabla_x \cdot u + \frac{1}{d} \nabla_x \cdot \Gamma}_{implicit} &= 0
	\end{split}
\end{equation}

%For a simpler discussion we use 
%the splitting
%
%
%
%{\bf Step 1:}
%\begin{align*}
%	%\label{Pf0}
%	 \partial_t \rho   =0 \\
%%	\partial_t p   +  \frac{1}{d}\nabla_x \cdot  \Gamma=0 \\
%	\partial_t u  +  \nabla_x \cdot \Theta   = 0\\
%	%\partial_t T +\nabla_x \cdot< v \frac{v^2-3}{3} f_1 >= 0\\
%	\partial_t  T  +  \frac{1}{d}\nabla_x \cdot \Gamma= 0
%\end{align*}
%
%
%
%{\bf Step  2:}
%\begin{align}
%	\partial_t \rho + \frac{1}{\epsilon}\nabla_x \cdot  u  =0 \\
%%	\partial_t p + \frac{d+2}{d\epsilon}\nabla_x \cdot  u  =0 \\
%	\partial_t u  + \frac{1}{\epsilon}\nabla_x  p   = 0\\
%	%\partial_t T +\nabla_x \cdot< v \frac{v^2-3}{3} f_1 >= 0\\
%	\partial_t  T + \frac{2}{\epsilon d}	\nabla_x \cdot u  = 0
%\end{align} 
%Note that for this step  $ \partial_t S =0$ with  $S= d p - (d+2)\rho$.

Writing a first order time discretization in a stepwise form we have  

\begin{align}
%\label{Pf0}
\frac{1}{\Delta t } (u^{n+1/2}- u^{n})+   \nabla_x \cdot \Theta^{n+1} =0 \nonumber\\
\frac{1}{\Delta t } (u^{n+1}- u^{n+1/2})+ \frac{1}{\epsilon} \nabla_x p^{n+1}=0\label{p1}\\
\frac{1}{\Delta t } (T^{n+1/2}- T^{n})+ \frac{1}{d}\nabla_x\cdot \Gamma^{n+1} = 0\nonumber\\
\frac{1}{\Delta t } (T^{n+1}- T^{n+1/2})+ \frac{2}{\epsilon d}\nabla_x\cdot u^{n+1} = 0\nonumber\\
\frac{1}{\Delta t } (\rho^{n+1}- \rho^{n})+ \frac{1}{\epsilon }\nabla_x\cdot u^{n+1} = 0\nonumber.
\end{align}
This leads for the pressure to
\begin{align}
	\frac{1}{\Delta t } (p^{n+1/2}- p^{n})+ \frac{1}{d}\nabla_x\cdot \Gamma^{n+1} &= 0\nonumber\\
	\frac{1}{\Delta t } (p^{n+1}- p^{n+1/2})+ \frac{(d+2)}{\epsilon d}\nabla_x\cdot u^{n+1} &= 0\label{p2}.
\end{align}
Moreover,
\begin{align*}
\frac{	\epsilon^2}{\Delta t }    (\Theta^{n+1}-\Theta^n)+ \D (\nabla_x u^n)=\ - \frac{1}{\tau}  \Theta^{n+1} + \frac{1}{\tau} \A(u^n) 
\end{align*}
and
\begin{align*}
\frac{	\epsilon^2}{\Delta t } (\Gamma^{n+1}-\Gamma^n)+  (d+2)  \nabla_x T^{n} = -\frac{1}{\tau}  \Gamma^{n+1} +   \frac{1}{\tau} (d+2)u^nT^n .
\end{align*}

which gives 
with $\gamma = \frac{1}{\epsilon^2 + \frac{\Delta t}{\tau} }=  \frac{\tau }{\epsilon^2 \tau + \Delta t}$

\begin{align}
\label{Gamma}
\Gamma^{n+1}= \epsilon^2 \gamma \Gamma^n+ (d+2)\Delta t \gamma [\frac{1}{\tau} u^n T^n- 
\nabla_x  T^n ]
\end{align}

\begin{align}
\label{Theta}
\Theta^{n+1}= \epsilon^2 \gamma\Theta^n+ \Delta t \gamma \frac{1}{\tau} \A(u^n)- 
\Delta t \gamma \D(\nabla_x u^n) 
\end{align}

The discrete equations can be rearranged in a more efficient way:
Compute first $u^{n+1/2},T^{n+1/2}$ using 
\begin{align*}
%\label{Pf0}
\frac{1}{\Delta t } (u^{n+1/2}- u^{n})+   \nabla_x \cdot \Theta^{n+1} =0 \nonumber\\
\frac{1}{\Delta t } (T^{n+1/2}- T^{n})+ \frac{1}{d}\nabla_x\cdot \Gamma^{n+1} = 0\nonumber
\end{align*}
or 
\begin{align}
\label{e1}
\frac{1}{\Delta t } (u^{n+1/2}- u^{n})+  \epsilon^2 \gamma  \nabla_x \cdot \Theta^n - \Delta t \gamma  
[\Delta_x u^n   +  \frac{d-2}{d} \nabla_x (\nabla_x \cdot  u^n)]+ \Delta t  \frac{1}{\tau} \gamma \nabla_x \cdot \A(u^n )=0 
\end{align}
and
\begin{align}
\label{e2}
\frac{1}{\Delta t } (T^{n+1/2}- T^{n})+ \frac{\epsilon^2 \gamma}{d}\nabla_x \cdot\Gamma^n - \Delta t \gamma  \frac{d+2}{d}\Delta_x    T^{n}  
+ \Delta t   \gamma  \frac{d+2}{d} \frac{1}{\tau}\nabla_x\cdot (uT)^n	= 0.
\end{align}
Then, we obtain $p^{n+1/2}$ simply via 
\begin{align}
\label{pn12}
p^{n+1/2}- p^{n}=T^{n+1/2}- T^{n}.
\end{align}
Moreover,
\begin{align}
\label{rhop}
\frac{d+2}{d} (\rho^{n+1}- \rho^{n})=p^{n+1}- p^{n+1/2}.
\end{align}
Combining (\ref{p1}) and (\ref{p2})  gives
\begin{align}
\label{ppoisson}
\frac{\epsilon^2}{\Delta t^2} p^{n+1}- \frac{d+2}{d}\Delta_x  p^{n+1}  = \frac{\epsilon^2}{\Delta t^2} p^{n+1/2}   - \frac{(d+2) \epsilon }{d \Delta t } \nabla_x \cdot u^{n+1/2}.
\end{align}
Then, (\ref{rhop}) can again be used to compute $\rho^{n+1}$ and 
$T^{n+1}$ via
\begin{align}
\label{Tp}
T^{n+1}- T^{n+1/2}=\frac{2}{d} (\rho^{n+1}- \rho^{n}).
\end{align}

The first order numerical algorithm is finally given in Algorithm \ref{alg1}.

\begin{algorithm}\caption{First order AP algorithm for relaxation scheme.}
	\label{alg1}
%	Start with $\rho^n,u^n,T^n$ and $\Gamma^{n},\Theta^{n}$ 
\textbf{Input:} State variables $\rho^n,u^n,T^n$ and $\Gamma^{n},\Theta^{n}$ at time $t^n$.
	\begin{enumerate}
		\item Compute $\Gamma^{n+1},\Theta^{n+1}$  via (\ref{Gamma}) and (\ref{Theta}).\\
		\item Compute
		$u^{n+1/2}, T^{n+1/2}$ explicitly via 	(\ref{e1}) and (\ref{e2}).
	 Compute $p^{n+1/2}$ via (\ref{pn12}).
		\item Compute  $\tilde  p^{n+1}=\frac{p^{n+1}}{\epsilon}$ using (\ref{ppoisson})  via 
\begin{align}
\frac{\epsilon^2}{\Delta t^2} \tilde p^{n+1}- \frac{d+2}{d}\Delta_x  \tilde p^{n+1}  = \frac{\epsilon}{\Delta t^2} p^{n+1/2}   - \frac{(d+2) }{d \Delta t } \nabla_x \cdot u^{n+1/2}
\end{align}
and $u^{n+1} $ as %using  (\ref{u}) as
$$u^{n+1}= u^{n+1/2}-  \Delta t \nabla_x  \tilde p^{n+1}.$$
%		
%		\begin{align*}
%			\frac{\epsilon^2}{\Delta t^2} p^{n+1}- \frac{d+2}{d}\Delta_x p^{n+1}  = \frac{\epsilon^2}{\Delta t^2}p^{n+1/2}   - \frac{(d+2) \epsilon}{d \Delta t } \nabla_x \cdot u^{n+1/2}
%		\end{align*}
%		and
%		\begin{align}
%			u^{n+1} =  u^{n+1/2} -  \frac{\Delta t}{\epsilon} \nabla_x p^{n+1}
%		\end{align}
\item 
	Compute  $\rho^{n+1}$ via 
	\begin{align}
		\frac{d+2}{d} (\rho^{n+1}- \rho^{n})=\epsilon \tilde p^{n+1}- p^{n+1/2}
	\end{align}
	and 
	$ T^{n+1}$ via  (\ref{Tp}).
%	 as
%	\begin{align}
%	\rho^{n+1} = \frac{d}{d+2}\epsilon \bar p^{n+1} - \frac{1}{d+2}S^{n+1/2} \\
%	T^{n+1}=\epsilon \bar p^{n+1} -\rho^{n+1}
%	\end{align}
%		
	\end{enumerate}	
\end{algorithm}

\begin{remark}

As $\epsilon \rightarrow 0$  we obtain   with $\gamma \rightarrow \frac{\tau}{\Delta t }$ 
		\begin{align*}
			\Gamma^{n+1}=  (d+2) [ u^n T^n- \tau
			\nabla_x  T^n ]
		\end{align*}
		
		\begin{align*}
			\Theta^{n+1}= \A(u^n)- 
			\tau\D(\nabla_x u^n )
		\end{align*}

	Then,
		$T^{n+1/2}$ is computed via 	
		\begin{align*}
		T^{n+1/2} 
		= T^{n} -  \Delta t \frac{d+2}{d}[\nabla_x \cdot (u^nT^n)-  \tau \Delta_x T^n]
		\end{align*}
		and
		\begin{align}
		u^{n+1/2} =  u^n - \Delta t [\nabla_x \cdot \A(u^n ) - \tau \Delta_x u^n]
		\end{align}
%			\begin{align*}
%		p^{n+1/2} =  T^{n+1/2} + \rho^{n} \\
%		S^{n+1/2}=   d T^{n+1/2} -2 \rho^{n}
%		\end{align*}

Then,
\begin{align*}
			 \Delta_x \tilde p^{n+1}  =  \frac{1 }{ \Delta t } \nabla_x \cdot u^{n+1/2}
		\end{align*} 
	and 
	$$
	u^{n+1} = u^{n+1/2}- \Delta t \nabla_x \tilde p
	$$	
	which gives 
	$\nabla_x \cdot u^{n+1} =0$.
	Finally, 
	\begin{align*}
		T^{n+1}-T^n &=	T^{n+1}-T^{n+1/2}+T^{n+1/2}-T^n \\
		&= (-\frac{2}{d+2}+1)(T^{n+1/2}-T^n )=-  \Delta t[\nabla_x \cdot (u^nT^n)-  \tau \Delta_x T^n]
	\end{align*}
	due to
	\begin{align*}
		T^{n+1}-T^{n+1/2} 
		= -\frac{2}{d+2}(T^{n+1/2}-T^n ).
	\end{align*}
	in the limit $\epsilon \rightarrow 0$ .
%	using 
%	\begin{align*}
%		S^{n+1/2} = d T^{n+1/2} -2 \rho^n
%	\end{align*}
% gives
%		\begin{align*}
%	S^{n+1} -S^n=  S^{n+1/2} -S^n
%= d (T^{n+1/2} -T^{n})
%	\end{align*}
%	and therefore 
%		\begin{align*}
%		S^{n+1} = S^n-  \Delta t (d+2)[\nabla_x \cdot (u^nT^n)-  \tau \Delta_x T^n].
%	\end{align*}
%Then
%\begin{align*}
%		T^{n+1}=\epsilon \bar p^{n+1} -\rho^{n+1}
%	\end{align*}
Thus, 	for $\epsilon \rightarrow 0$
	\begin{align*}
	 T^{n+1}= T^n-  \Delta t[\nabla_x \cdot (u^nT^n)-  \tau \Delta_x T^n]
	\end{align*}
	We have  obtained 	a standard projection scheme for the incompressible thermal Navier-Stokes system with an explicit treatment of the diffusive terms in the limit equations.

\end{remark}

\subsection{Additional implicit discretization treating the parabolic stiffness}
\label{pstiff}

	The algorithm is the same as before, but instead of 	(\ref{e1}) and (\ref{e2}) we compute
	
			\begin{align*}
			&	\frac{1}{\gamma \Delta t^2}T^{n+1/2}- \frac{d+2}{d} \Delta_x T^{n+1/2}=
			\frac{1}{\gamma \Delta t^2}T^{n}-\frac{\epsilon^2}{d \Delta t } \nabla_x \cdot \Gamma^n- \frac{d+2}{ d \tau}
			\nabla_x \cdot (u^n T^n)
		\end{align*}
	and 
			\begin{align*}
			\frac{1}{\Delta t^2 \gamma}u^{n+1/2}- \Delta_x u^{n+1/2}=	\frac{1}{\Delta t^2 \gamma}u^{n}
			- \frac{\epsilon^2}{\Delta t } \nabla_x \cdot \Theta^n +
			\frac{d-2}{d}\nabla_x (\nabla_x \cdot u^{n}) - \frac{1}{\tau  } \nabla_x \cdot \A(u^n)
		\end{align*}
and then 
\begin{align*}
	\Gamma^{n+1}= \epsilon^2 \gamma \Gamma^n+ (d+2)\Delta t \gamma [\frac{1}{\tau} u^n T^n- 
	\nabla_x  T^{n+1/2} ]
\end{align*}
and
\begin{align*}
	\Theta^{n+1}= \epsilon^2 \gamma\Theta^n+ \Delta t \gamma \frac{1}{\tau} \A(u^n)- 
	\Delta t \gamma\D(\nabla_x u^{n+1/2} ) 
\end{align*}

with $\gamma = \frac{1}{\epsilon^2 + \frac{\Delta t}{\tau} }$.
%%%%%%%%%%%%%%

\begin{remark}
For $\epsilon \rightarrow 0$ this is with $\gamma \rightarrow  \frac{\tau}{\Delta t}$
		\begin{align*}
			&	T^{n+1/2}=
			T^{n}- \Delta t  \frac{d+2}{d}
			\nabla_x \cdot(u^n T^n) + \Delta t \frac{d+2}{d} \tau \Delta_x T^{n+1/2}
		\end{align*}
	and
		\begin{align*}
			u^{n+1/2}&=	u^{n}
			 - \Delta t  \nabla_x \cdot \A(u^n)+ \Delta t \tau \Delta_x u^{n+1/2}
		\end{align*}
		and otherwise the same scheme as before.
%\item Compute  
%\begin{align*}
%	d T^{n+1}-2 \rho^{n+1} = d T^{n+1/2} - 2 \rho^n = d T^{n} - 2 \rho^n \\ - 	 \Delta t  (d+2)
%	\nabla_x \cdot(u^n T^n) + \Delta t (d+2) \tau \Delta_x T^{n+1/2} 
%\end{align*}	

%\begin{align*}
%	T^{n+1} = -\rho^{n+1}  
%\end{align*}
%this is 
%\begin{align*}
%	T^{n+1}= T^{n}  -	 \Delta t  
%	\nabla_x \cdot(u^n T^n) + \Delta t \tau \Delta_x T^{n+1/2} 
%\end{align*}	
	
%	\end{enumerate}
Thus, we obtain a projection scheme with implicit treatment of the diffusive terms in the limit.

\end{remark}

\section{AP scheme - Higher order time-discretizations}\label{sec5}
\label{ho}
The main idea is to use a suitable combination of implicit–explicit (IMEX) Runge–Kutta discretizations \cite{IMEX_PR,Bosc}. The main advantage of the proposed strategy is that, compared to a fully explicit method, the time step is not restricted by the stiffness of the system, while compared to a fully implicit scheme, the solution procedure is more straightforward, since the system to be solved is less involved. We will show that this choice is sufficient to obtain a projection scheme for the incompressible Navier–Stokes equations in the mean–free–path limit. The general IMEX–RK schemes can be conveniently represented by a double Butcher tableau of the following form:
\begin{align}
	\text{Explicit: } 
	\begin{array}{c|c}
		\tilde{c} & \tilde{A} \\ \hline \\ [-1.9ex]
		& \tilde{w}^T
	\end{array}
	\qquad
	\text{Implicit: }
	\begin{array}{c|c}
		c & A \\ \hline \\[-1.9ex]
		& w^T
	\end{array}
\end{align}
For the explicit scheme, $\tilde{A}$ is an $s \times s$ lower–triangular matrix with zero diagonal entries.
For the implicit scheme, the $s \times s$ matrix $A$ is restricted to the class of diagonally implicit Runge–Kutta (DIRK) methods, which means that $a_{ij} = 0$ for $j > i$. The vectors $w$ and $\tilde{w}$ contain the quadrature weights used to combine the internal stages of the Runge–Kutta method, while the coefficients $c$ and $\tilde{c}$ are defined as
\[c_i = \sum_{j=1}^{i} a_{ij}, \quad \tilde{c}_i = \sum_{j=1}^{i-1} \tilde{a}_{ij}.\]
The following definitions characterize the IMEX schemes and the properties we aim to consider in this work. For further details we refer for instance to \cite{IMEX_PR}.
\begin{definition}
	An IMEX RK method is said to be of type A, if the matrix $A \in \mathbb{R}^{s \times s}$ is invertible, or equivalently $a_{ii} \neq 0$, $i = 1,\dots,s$.
\end{definition}
\begin{definition}
	An IMEX RK method is said to be of type CK, if the matrix $A \in \mathbb{R}^{s \times s}$ can be written as
	\begin{align*}A=\left(
		\begin{matrix}
			0 & 0 \\ a & \hat{A}
		\end{matrix}\right)
	\end{align*}
	with $a = (a_{21},\dots,a_{s1})^T \in \mathbb{R}^{(s-1)}$ and the submatrix $\hat{A} \in \mathbb{R}^{(s-1) \times (s-1)}$ invertible. If additionally $a = 0$, the scheme is said to be of type ARS.
\end{definition}
\begin{definition}
	An IMEX RK method is called implicitly stiffly accurate (ISA), if the corresponding DIRK method is stiffly accurate, which means that
	\begin{align*}
		a_{si} = w_i, \qquad i = 1,\dots,s.
	\end{align*}
	If additionally the explicit method satisfies
	\begin{align}
		\tilde{a}_{si} = \tilde{w}_i, \qquad i = 1,\dots,s,
	\end{align}
	the IMEX RK method is called globally stiffly accurate (GSA).
\end{definition}
Note that the numerical solution of a GSA IMEX RK scheme coincides exactly with the last internal stage \cite{IMEX_PR},\cite{Hairer}. We now apply the above general IMEX-RK approach to the moment system.
We proceed by separating the implicit and the explicit components as before \eqref{implicit-explicit}.
%
%)following:
%\begin{equation}
%	\begin{split} %\label{implicit-explicit}
%		\partial_t \rho +  \underbrace{\frac{1}{\varepsilon}\nabla_x \cdot u}_{implicit} &= 0 \\
%		\partial_t u + \underbrace{ \nabla_x \cdot \langle \Theta + \frac{1}{\varepsilon} \nabla_x p}_{implicit} &= 0 \\
%		\partial_t T +  \underbrace{\frac{2}{\varepsilon d} \nabla_x \cdot u + \frac{1}{d} \nabla_x \cdot \Gamma}_{implicit} &= 0\\
%		\partial_t p +  \underbrace{\frac{d+2}{\varepsilon d} \nabla_x \cdot u + \frac{1}{d} \nabla_x \cdot \Gamma}_{implicit} &= 0\\
%		\varepsilon^2 \partial_t \Theta +  \underbrace{\D(\nabla_x u )}_{explicit} &= -  \underbrace{\frac{1}{\tau} \Theta}_{implicit} + \underbrace{\mathcal{A}(u)}_{explicit}\\
%		\varepsilon^2 \partial_t \Gamma +  \underbrace{(d+2) \nabla_x T}_{explicit} &= -  \underbrace{\frac{1}{\tau} \Gamma}_{implicit} + \underbrace{(d+2) uT}_{explicit}
%	\end{split}
%\end{equation}
So, we obtain for the internal stages: 
\begin{equation}
	\begin{split} \label{InternalStages}
		\rho^{(i)} &= \rho^n - \Delta t \frac{1}{\varepsilon} \sum_{j = 1}^i a_{ij} \nabla_x \cdot u^{(j)} \\
		u^{(i)} &= u^n - \Delta t  \sum_{j = 1}^i \left( a_{ij} \nabla_x \cdot \Theta^{(j)} + \frac{1}{\varepsilon} a_{ij} \nabla_x p^{(j)}\right) \\
		T^{(i)} &= T^n - \Delta t  \sum_{j = 1}^i \left( \frac{1}{d} a_{ij} \nabla_x \cdot \Gamma^{(j)}  + \frac{2}{d \varepsilon} a_{ij} \nabla_x \cdot u^{(j)}\right) \\
		p^{(i)} &= p^n - \Delta t  \sum_{j = 1}^i \left( \frac{1}{d} a_{ij} \nabla_x \cdot \Gamma^{(j)}  + \frac{d+2}{d \varepsilon} a_{ij} \nabla_x \cdot u^{(j)}\right) \\
		\Theta^{(i)} &= \Theta^n - \frac{\Delta t}{\varepsilon^2 \tau} \sum_{j=1}^{i-1} \tilde{a}_{ij}\left( \tau \D(\nabla_x u^{(j)} ) - \mathcal{A}(u^{(j)}) \right) - \frac{\Delta t}{\varepsilon^2 \tau} \sum_{j=1}^i a_{ij} \Theta^{(j)}\\
		\Gamma^{(i)} &= \Gamma^n - \frac{\Delta t}{\varepsilon^2\tau} \sum_{j=1}^{i-1} \tilde{a}_{ij}\left( (d+2)\tau \nabla_x T^{(j)} - (d+2)u^{(j)}T^{(j)} \right) - \frac{\Delta t}{\varepsilon^2 \tau} \sum_{j=1}^i a_{ij} \Gamma^{(j)}
	\end{split}
\end{equation}
\begin{theorem}\label{thm:5_1}
	If the IMEX method is of type A (of type CK/ARS) and satisfies the GSA property (and the initial velocity field is divergence free), then in the fluid dynamic limit $\varepsilon \rightarrow 0$ the IMEX RK scheme \eqref{InternalStages} becomes a consistent discretization for the incompressible Navier Stokes equations.
\end{theorem}
%\begin{proof}
%Analogously to the proof in \cite{2ndPaper}.
%\end{proof}

\begin{proof}
To establish the result, we closely follow the framework developed in \cite[Thm 4.4 and Thm 4.5]{2ndPaper}. 

The moment reduction mechanism ensures that as $\epsilon$ vanishes, higher-order kinetic moments are systematically slaved to the macroscopic variables through the dominant relaxation terms, yielding the correct macroscopic limit equations.
\end{proof}

\begin{remark}
\label{rem:ap_properties}
When interpreting Theorem~\ref{thm:5_1}, it is essential to maintain a strict conceptual distinction between:
\begin{enumerate}
    \item the \emph{structural stability} of the continuous and semi-discrete systems,
    \item the \emph{formal continuous limit} as $\epsilon \to 0$, and
    \item the \emph{time-semidiscrete AP property} ensuring consistency at the discrete level.
\end{enumerate}
Furthermore, while the scheme achieves \emph{high-order} accuracy, assertions regarding uniform numerical stability are strictly bounded by the established CFL and dissipation constraints, avoiding generalized claims of unconditional uniform stability.
\end{remark}

The final algorithm for solving the moment system is given by Algorithm \ref{highalg}.

\begin{algorithm}
	\caption{An Asymptotic Preserving GSA IMEX Runge-Kutta scheme for the relaxation system}\scriptsize
	\label{highalg}
	\textbf{Input:} State variables $(\Gamma^{n}, \Theta^{n}, u^{n}, \rho^{n}, T^{n}, p^{n})$ at time $t^n$, time step $\Delta t$.
	
	\vspace{2pt}
	\textbf{Stage Initialization:}
	\begin{itemize}
		\item[] \textbf{if} tableau is CK/ARS-type ($a_{11} = 0$) \textbf{then} set $(\Gamma^{(1)}, \Theta^{(1)}, u^{(1)}, \rho^{(1)}, T^{(1)}, p^{(1)}) \gets (\Gamma^{n}, \Theta^{n}, u^{n}, \rho^{n}, T^{n}, p^{n})$ and $i_{\min} = 2$.
		\item[] \textbf{else} (Type A, $a_{11} \neq 0$) set $i_{\min} = 1$.
	\end{itemize}

	\vspace{2pt}
	\textbf{for} $i = i_{\min}, \dots, s$ \textbf{do}:
	\begin{enumerate}
		\item Compute $\Gamma^{(i)}$ by solving
		\begin{equation}
			\begin{split} \label{FirstStep}
				\Gamma^{(i)} &= \frac{\varepsilon^2 \tau}{\varepsilon^2 \tau + a_{ii} \Delta t}\Gamma^n - \frac{\Delta t (d+2) }{\varepsilon^2 \tau + a_{ii} \Delta t} \sum_{j=1}^{i-1} \tilde{a}_{ij}\left( \tau \nabla_x  T^{(j)}  - u^{(j)} T^{(j)}  \right) \\&- \frac{\Delta t}{\varepsilon^2 \tau + a_{ii} \Delta t} \sum_{j=1}^{i-1} a_{ij} \Gamma^{(j)}.
			\end{split}
		\end{equation} 
		\item Compute $\Theta^{(i)}$ by solving
		\begin{equation}
			\begin{split} \label{SecondStep}
				\Theta^{(i)} &= \frac{\varepsilon^2 \tau}{\varepsilon^2 \tau + a_{ii} \Delta t}\Theta^n - \frac{\Delta t }{\varepsilon^2 \tau + a_{ii} \Delta t} \sum_{j=1}^{i-1} \tilde{a}_{ij}\left(  \tau \D(\nabla_x u^{(j)} )  - \mathcal{A}(u^{(j)}) \right) \\&- \frac{\Delta t}{\varepsilon^2 \tau + a_{ii} \Delta t} \sum_{j=1}^{i-1} a_{ij} \Theta^{(j)}.
			\end{split}
		\end{equation} 
		
		\item Compute ${\tilde p^{(i)} = p^{(i)}/\varepsilon}$ by solving the following Helmholtz equation
		\begin{equation}\begin{split}\label{ThirdStep}
			&	\frac{2+d}{d}   \Delta_x {\tilde p^{(i)}} - \frac{ \varepsilon^2}{\Delta t^2 a_{ii}^2} {\tilde p^{(i)}} = -\frac{ \varepsilon}{\Delta t^2 a_{ii}^2}{p^{(i)}}^*
			\\&	+ \frac{1}{\Delta t a_{ii}^2}  \sum_{j = 1}^{i-1} \left( \frac{2+d}{d } a_{ij} \nabla_x \cdot u^{(j)}\right)  +  \frac{ 1}{\Delta t a_{ii}} \frac{2+d}{d} \nabla_x \cdot {u^{(i)}}^* - \frac{1}{a_{ii}} \frac{2+d}{d} \sum_{j=1}^{i-1} a_{ij} \Delta_x \tilde p^{(j)}
			\end{split}
		\end{equation}
		with
		\begin{equation}
			\begin{split}
				{p^{(i)}}^* &= p^n - \frac{\Delta t}{d} \sum_{j = 1}^i  a_{ij}  \nabla_x \cdot \Gamma^{(j)} , \\
				{u^{(i)}}^* &= u^n - \Delta t  \sum_{j = 1}^i  a_{ij} \nabla_x \cdot  \Theta^{(j)}  .
			\end{split}
		\end{equation}
		\item Compute $u^{(i)}$ by solving
		\begin{equation}
			\begin{split}\label{FourthStep}
				{u^{(i)}} = u^n - \Delta t  \sum_{j = 1}^i  a_{ij} \nabla_x \cdot  \Theta^{(j)}  - \Delta t \sum_{j=1}^{i} a_{ij} \nabla_x \tilde p^{(j)} 
			\end{split}
		\end{equation}
			\item Compute $\rho^{(i)}$ by solving
		\begin{equation}
		\begin{split}\label{FifthStep}
		{\rho^{(i)}} =   \rho^n  - \frac{\Delta t}{\varepsilon} \sum_{j=1}^{i} a_{ij} \nabla_x \cdot u^{(j)}
		= \rho^n  + \frac{d}{d+2} (\epsilon \tilde p^{(i)}-{p^{(i)}}^*)
		\end{split}
		\end{equation}
		\item Compute $T^{(i)}$ by solving
		\begin{equation}
			\begin{split}\label{SixthStep}
				{T^{(i)}} =   				T^n - \frac{\Delta t}{d}  \sum_{j = 1}^i  a_{ij}  \nabla_x \cdot \Gamma^{(j)} - \frac{2 \Delta t}{d \varepsilon} \sum_{j=1}^{i} a_{ij} \nabla_x \cdot u^{(j)}
%				=T^n +{p^{(i)}}^*- p^n - \frac{2 \Delta t}{d \varepsilon} \sum_{j=1}^{i} a_{ij} \nabla_x \cdot u^{(j)}\\
				=T^n +{p^{(i)}}^*- p^n + \frac{2}{d}  (	{\rho^{(i)}} -  \rho^n) 
			\end{split}
		\end{equation}

	\end{enumerate}
	\textbf{end for}
	
	\vspace{2pt}
\textbf{Output:} Advanced state $(\Gamma^{n+1}, \Theta^{n+1}, u^{n+1}, \rho^{n+1}, T^{n+1}, p^{n+1}) \gets (\Gamma^{(s)}, \Theta^{(s)}, u^{(s)}, \rho^{(s)}, T^{(s)}, p^{(s)})$.
	\label{alg1RelSys}
\end{algorithm}

\section{Numerical Implementation and Spatial Discretization Details}
\label{sec6}
\label{spat}

To complement the description of the AP scheme, we provide specific details concerning the spatial discretization and solver configurations. 

\subsection{Spatial Discretization and Time-Step Restrictions}
The numerical solver utilizes IMEX RK tableaus (specifically ARS-type schemes featuring GSA properties) to treat the source term stiffness efficiently. The time-step size $\Delta t$ is governed by a standard advective CFL condition and is independent of the microscopic scaling parameter $\epsilon$. 

Regarding the spatial operators:
\begin{itemize}
    \item First-order derivatives (outside of the divergence/gradient coupling) are discretized using central finite differences of second or fourth order.
    \item The flux term in step~4 of Algorithm~\ref{highalg}, in particular $\nabla_x \cdot \Theta$, is handled via Lax-Friedrichs (LF) numerical fluxes combined with WENO \cite{Shu} reconstructions (WENO3/WENO5).
    \item The discrete Laplacian $\Delta_x \tilde{p}$ in the Helmholtz problem is solved using standard second/fourth-order central differences on the grid.
\end{itemize}

We note also that the choice of the spatial discretization is not restricted to the above described Finite-Difference approach. We have also implemented a generalized Finite Difference discretization for arbitrary  grids based on a least square determination of the 
derivatives yielding similar results.

\subsection{Collocated Grid Formulation and Oscillation Suppression}
The scheme operates on a collocated grid where all primitive variables ($\rho, u, T, p$) share identical nodes. In naive collocated formulations, such an arrangement is inherently prone to unphysical checkerboard pressure oscillations due to grid decoupling. In the present AP scheme, these spurious modes are suppressed owing to the compact stencils combined with LF numerical fluxes. Specifically, the numerical dissipation inherent in the LF flux formulation introduces local inter-node coupling between adjacent cells, effectively damping high-frequency odd-even decoupling and ensuring stable pressure-velocity coupling without requiring separate grid staggering.

\subsection{Treatment of Diffusive Stiffness}
It should be noted that while the present AP formulation successfully eliminates stiffness related to the relaxation source terms and $\epsilon$, the explicit treatment of diffusion terms in the limit equations (or specific substeps in Eq.~\ref{InternalStages}) maintains a standard parabolic stability constraint of the form $\Delta t \le \mathcal{O}(\Delta x^2)$. 

\subsection{Code Availability}
The implementation scripts and numerical routines associated with this study are openly accessible in the project's GitHub repository: \url{https://github.com/sudarshan825/}.

\section{Numerical results}\label{sec7}
\label{sec:numerics}

In this section we present a numerical investigation of the algorithm described above and a comparison of the solutions of the BGK equation \eqref{BGKoriginal} and of the relaxation system \eqref{relax1}--\eqref{relax5}.

\subsection{Numerical methods for the  BGK equation}

In order to compare the numerical results to the kinetic solution, we use two  numerical schemes for 
the full BGK equation \eqref{BGKoriginal} from \cite{TKR} and \cite{2ndPaper}.
The first scheme developed in \cite{TKR} is based on a  classical discretization approach for the BGK equation using an implicit treatment of the collision terms and  a second order IMEX-WENO method. The scaled advection term still requires a time step of the order $\epsilon$
chosen according to 
\begin{equation}
\label{CFL}
\Delta t  = CFL \frac{ \Delta x \epsilon }{v_{max} }.
\end{equation}
Here, $v_{max} $ is the maximal absolute value of the velocites used for the discretization of the velocity space in the BGK model.
 %For the spatial discretization  a  WENO type second order scheme.  
The second method developed in \cite{2ndPaper}  is  based on an AP-method similar to the one described for the relaxation system above. In contrast to the solution of the relaxation system,  the method computes the  full kinetic evolution.
Compared to a classical BGK solver, the scheme avoids the dependence of  the time step 
on $\epsilon$ and leads in the limit to a classical projection scheme for the Navier-Stokes equations as in the case of the relaxation system discussed here. 
%For the numerical comparisons we use  for the relaxation system with arbitrary value of $\epsilon$
%the same time step as for the BGK equation with $\epsilon=1$. 
%The values of $\rho,u,T$ computed by the solution of the relaxation system are then compared to  the density $\rho_F$, mean velocity $u_F$ and temperature $T_F$ computed from the BGK equation using (\ref{relation}).
%

\subsection{Results in 1D}

We consider the following periodic test case in 1D. Let $x \in [0, 1]$ and consider the following initial conditions, similar as in  \cite{DegondTang2011}:
\begin{align*}
	\rho(0,x) &=
	\left\{ 
	\begin{array}{l}  
		1, \qquad \qquad x \in [0, 0.2] ,
		\\ \nonumber 
		1 + \epsilon ^2 , \qquad x \in (0.2, 0.3], 
		\\ \nonumber 
		1, \qquad \qquad x \in (0.3, 0.7],
		\\ \nonumber
		1 - \epsilon^2, \qquad x \in (0.7,0.8],\\
		1 \;\; \qquad \qquad x \in (0.8, 1],
	\end{array}
	\right.
\end{align*}
Moreover, we choose $
u(0, x) = 0, \;
T(0, x) = 0, \;
\Gamma (0,x) = 0$.
%The aim of this setup is to investigate how accurately the numerical scheme handles sharp transitions in density while maintaining stability in the diffusive limit.
We use the numerical scheme for the BGK model described in  \cite{TKR} and  consider 1000 spatial grid points for the BGK simulation, as well as for the spatial discretization of the relaxation system. % This results in a grid spacing $\Delta x = 1/1000$.
We investigate test cases with different parameters $\epsilon = 1, 0.1, 0.05$ and relaxation times $\tau = 0.1, 0.01$ in Figures~\ref{1d_eps_1_t_0dot15} to \ref{1d_eps_0dot05_t_0dot15}. 
% and the length of radius $h = 3.1 dx$ for neighbor searching. 
%Moreover, periodic boundary conditions are used.
%The number of velocity grid points $N_v$ for the BGK equation has to be chosen according to the value of the parameters $\tau$ and $\epsilon$. For larger values of $\tau$ and $\epsilon$ the number 
%$N_v$ has to be chosen larger to avoid spurious oscillations, see, e.g. \cite{TKR}.
In our computations we have used  a velocity  discretization with  $N_v = 200$.
We consider a finite velocity interval  $[-v_{max}, v_{max}]$, with $v_{max} = 10$. 
%resulting in 
%\begin{align*}
%\rho_F(0, x) &= 1 + \epsilon \rho(0.x), \\\ u_F(0,x) &= \frac{\epsilon u(0,x)}{\rho_{F}(0,x)} = 0\\
%T_F(0,x)  &= 1 + \epsilon \frac{T(0,x)}{1 + \epsilon \rho(0,x)} - \epsilon^2 \frac{|u|^2}{2(1 + \epsilon \rho(0,x))^2} = 1.
%\end{align*} 
The associated initial conditions for the BGK equation are given by the Maxwellian distribution \eqref{Maxwelloriginal}, where the macroscopic parameters $\rho_F, u_F,T_F$ are computed from $\rho,u,T$ using the relation \eqref{relation}.

In Figures~\ref{1d_eps_1_t_0dot15} to \ref{1d_eps_0dot05_t_0dot15}, we show density $\rho$, velocity $u$ and temperature $T$ obtained from kinetic equation and relaxation model at different runtimes and for different parameters $\tau$ and $\epsilon$.
 One observes clearly, looking at Figure \ref{1d_eps_1_t_0dot15}, that for larger values of $\epsilon$ the approximation quality deteriorates due to the pronounced wave pattern present in the hyperbolic relaxation system, whereas the solution of the BGK equation has a much smoother behaviour.
 
 On the contrary, Figures~\ref{1d_eps_0dot1_t_0dot15} and   \ref{1d_eps_0dot05_t_0dot15} show, that for smaller values of $\epsilon$ the approximation quality of the relaxation model  is much better, showing only small deviations from the kinetic solution.

\begin{figure}
	\centering
	\includegraphics[keepaspectratio=true, width=4cm]{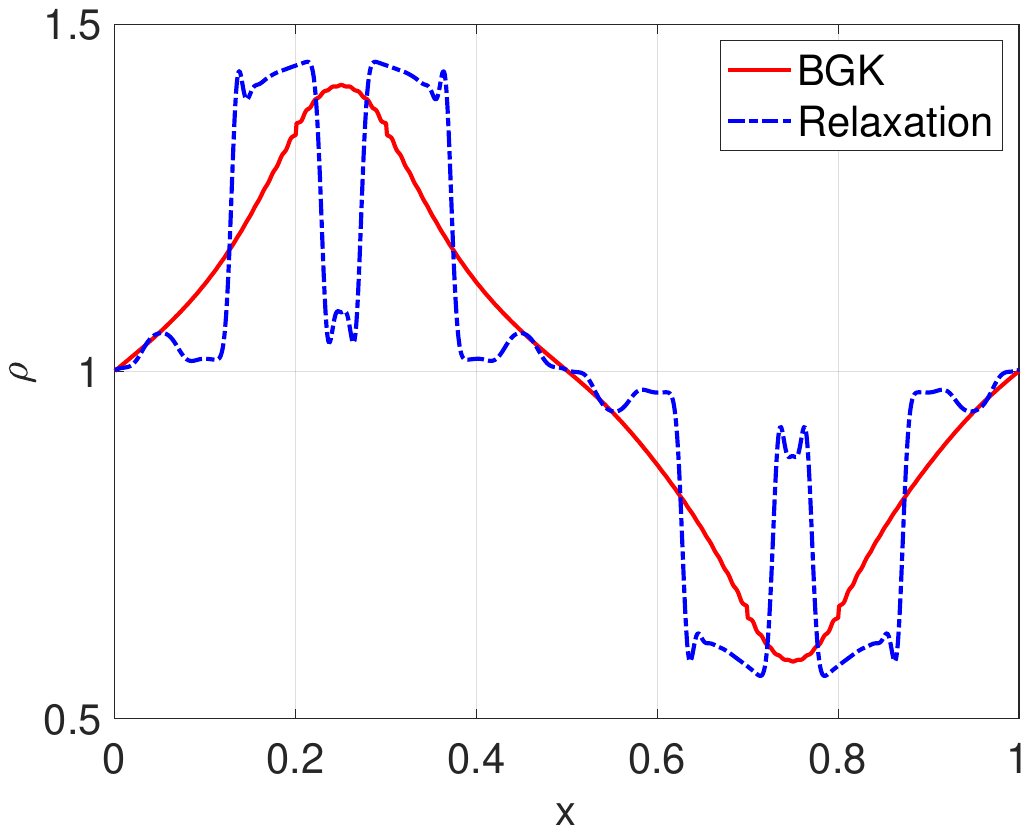}
	\includegraphics[keepaspectratio=true, width=4cm]{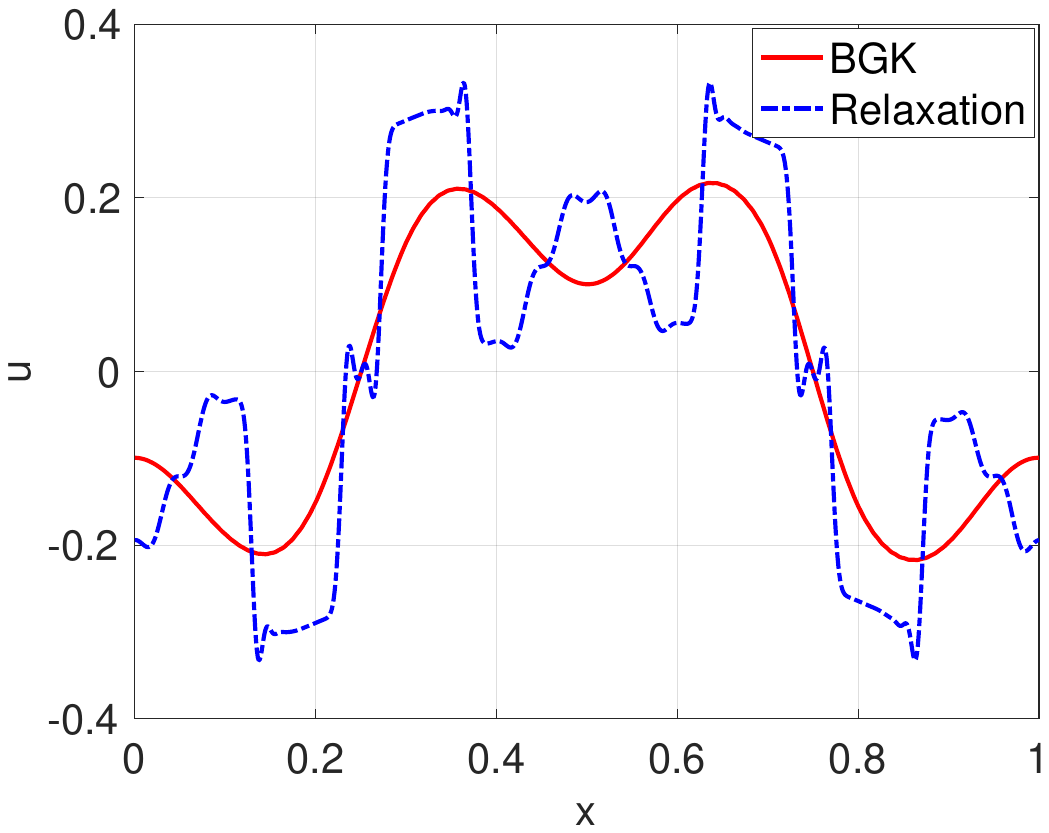}
	\includegraphics[keepaspectratio=true, width=4cm]{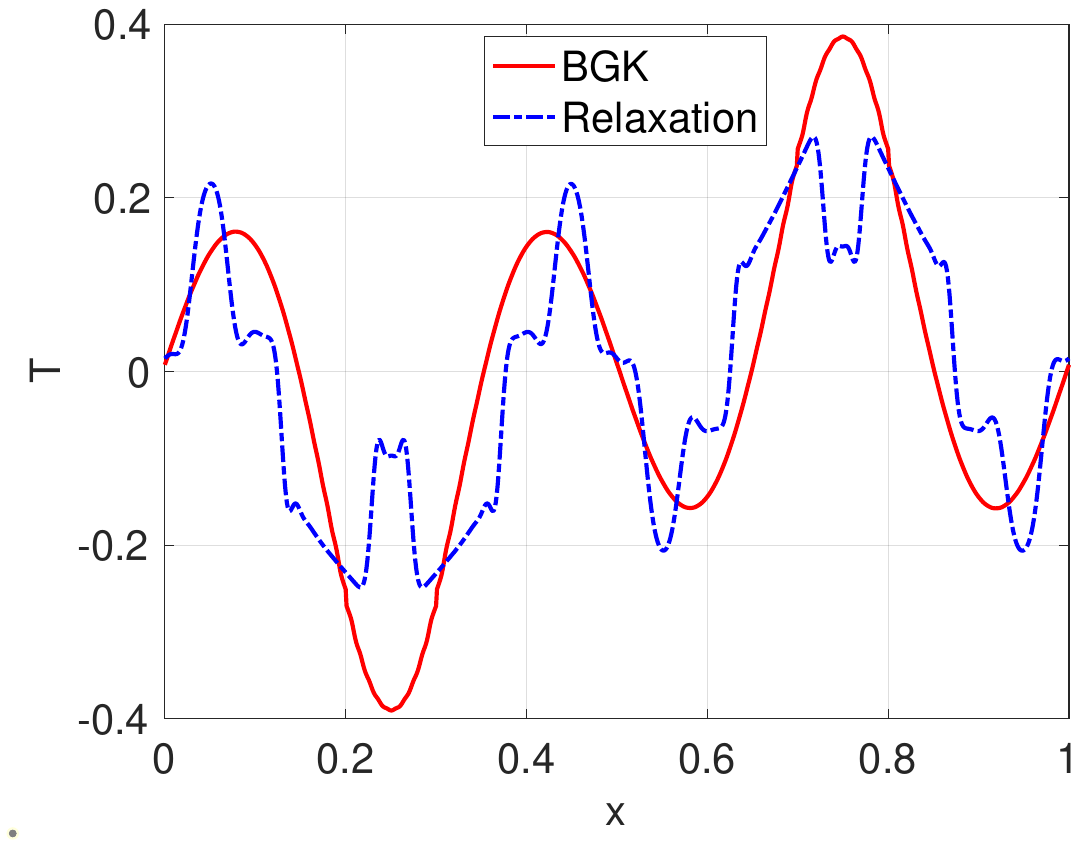}
	\\[-2cm]
	\includegraphics[keepaspectratio=true, width=4cm]{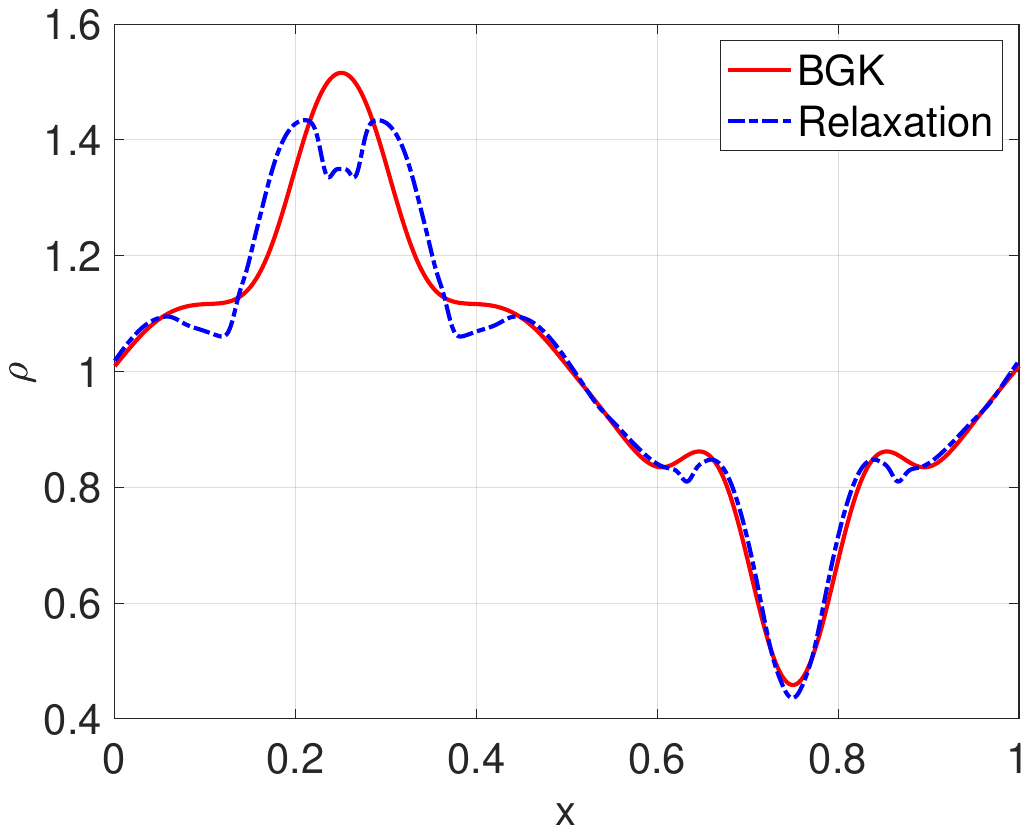}
	\includegraphics[keepaspectratio=true, width=4cm]{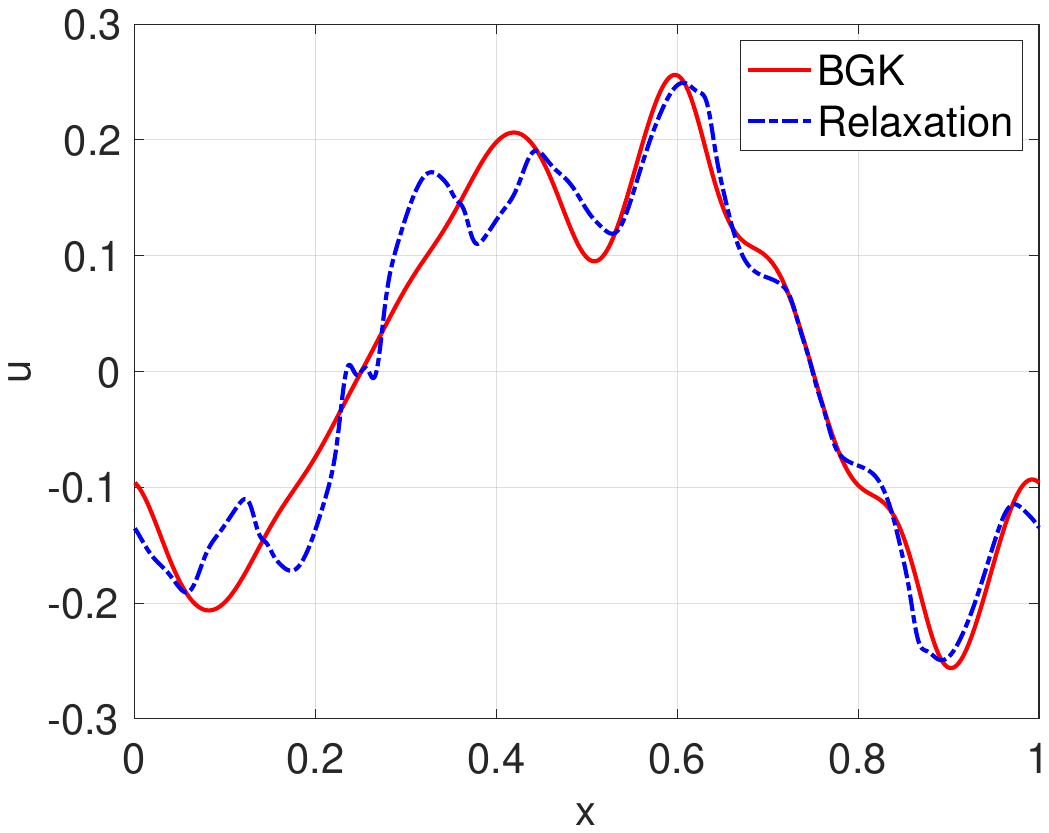}
	\includegraphics[keepaspectratio=true, width=4cm]{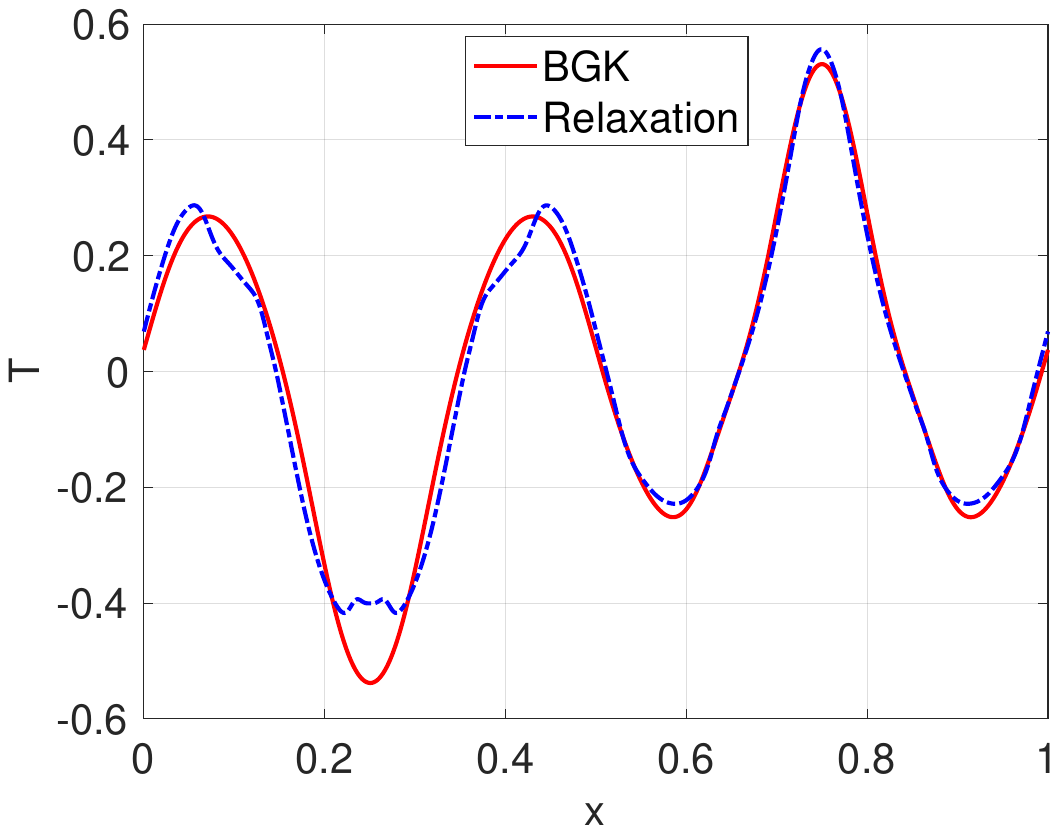}
	\\[-1.5cm]
	\caption{ Comparision of the  density, velocity and temperature obtained from the BGK equation and the relaxation system at final time $t_f = 0.1$ for $\epsilon = 1$, $\tau = 0.1$(first row) and $\tau = 0.01$ (second row). }
	\label{1d_eps_1_t_0dot15}
\end{figure}  
\begin{figure}
	\centering
	\includegraphics[keepaspectratio=true, width=4cm]{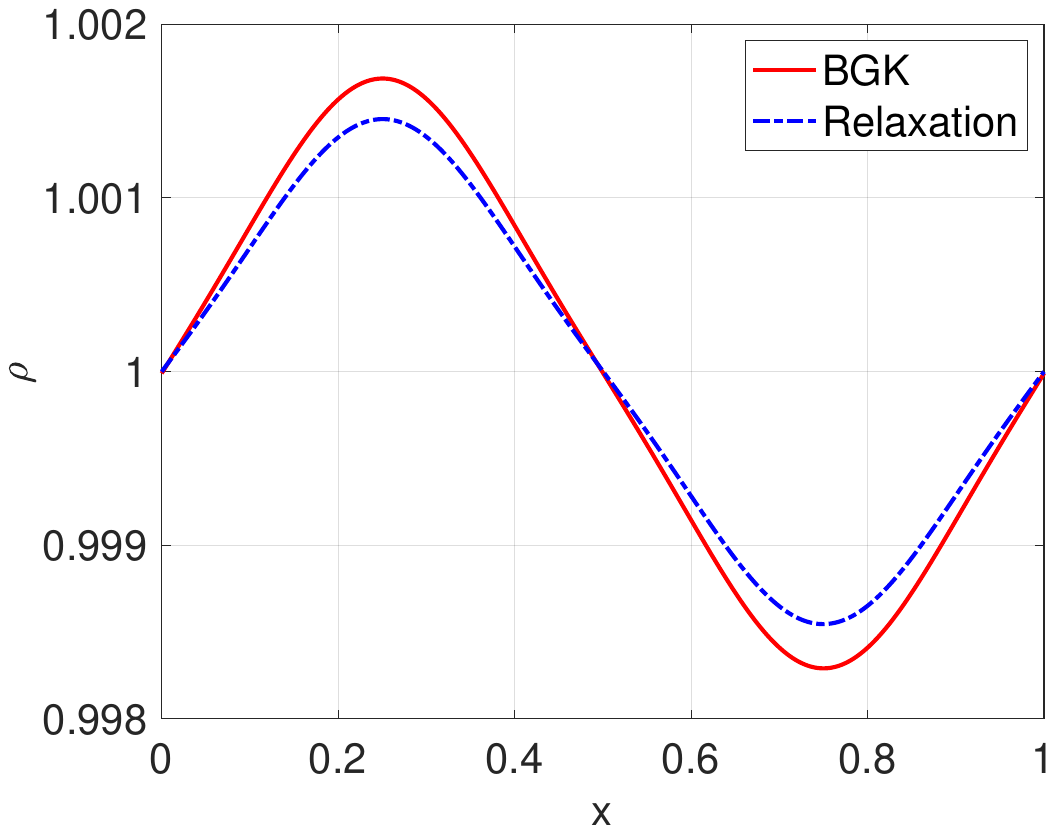}
	\includegraphics[keepaspectratio=true, width=4cm]{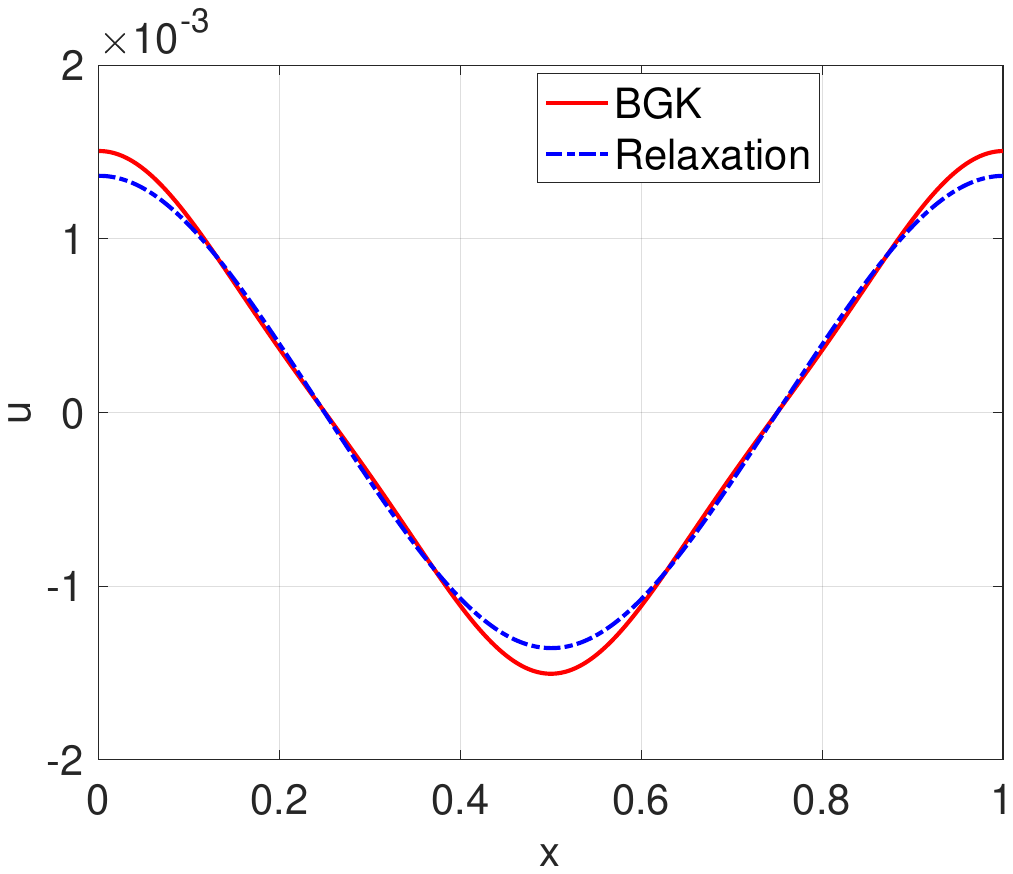}
	\includegraphics[keepaspectratio=true, width=4cm]{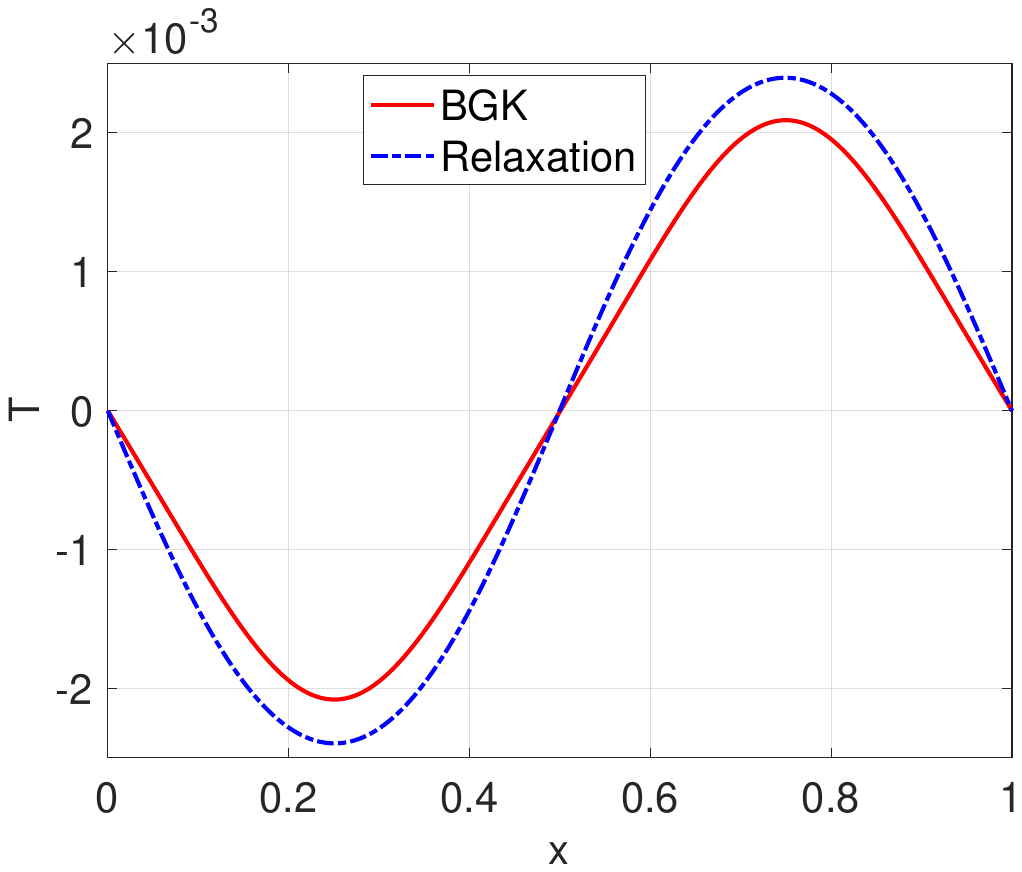}
	\\[-2cm]
	\includegraphics[keepaspectratio=true, width=4cm]{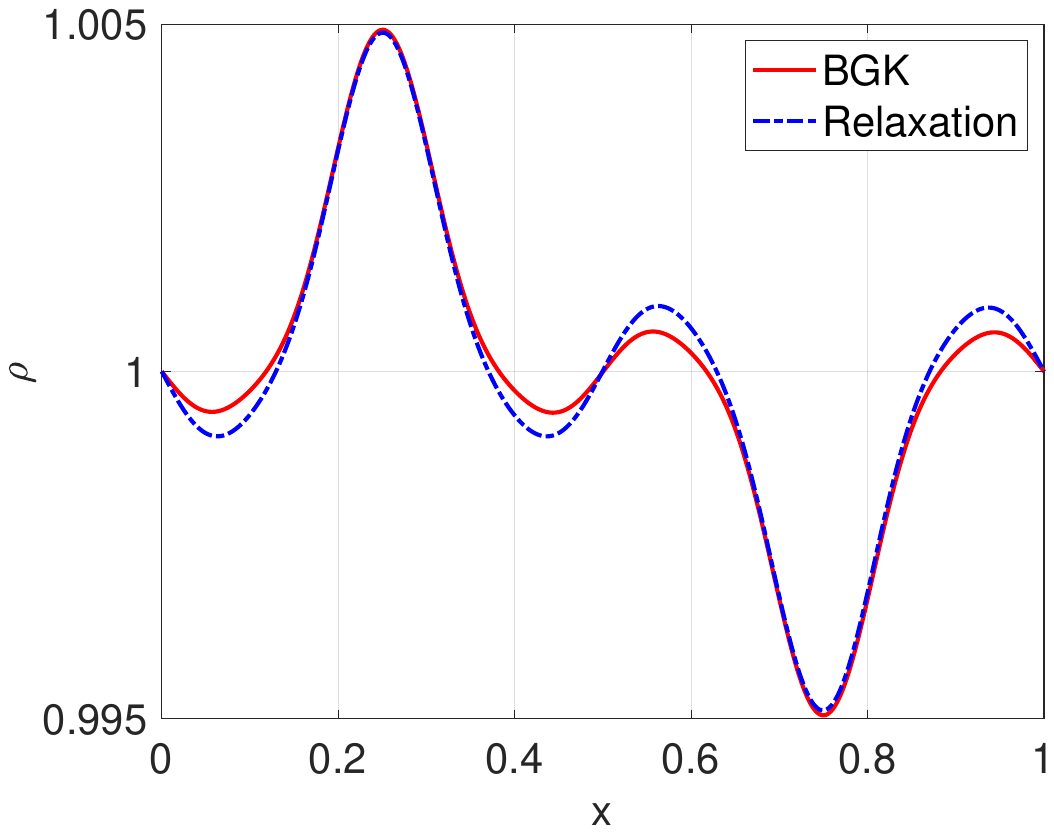}
	\includegraphics[keepaspectratio=true, width=4cm]{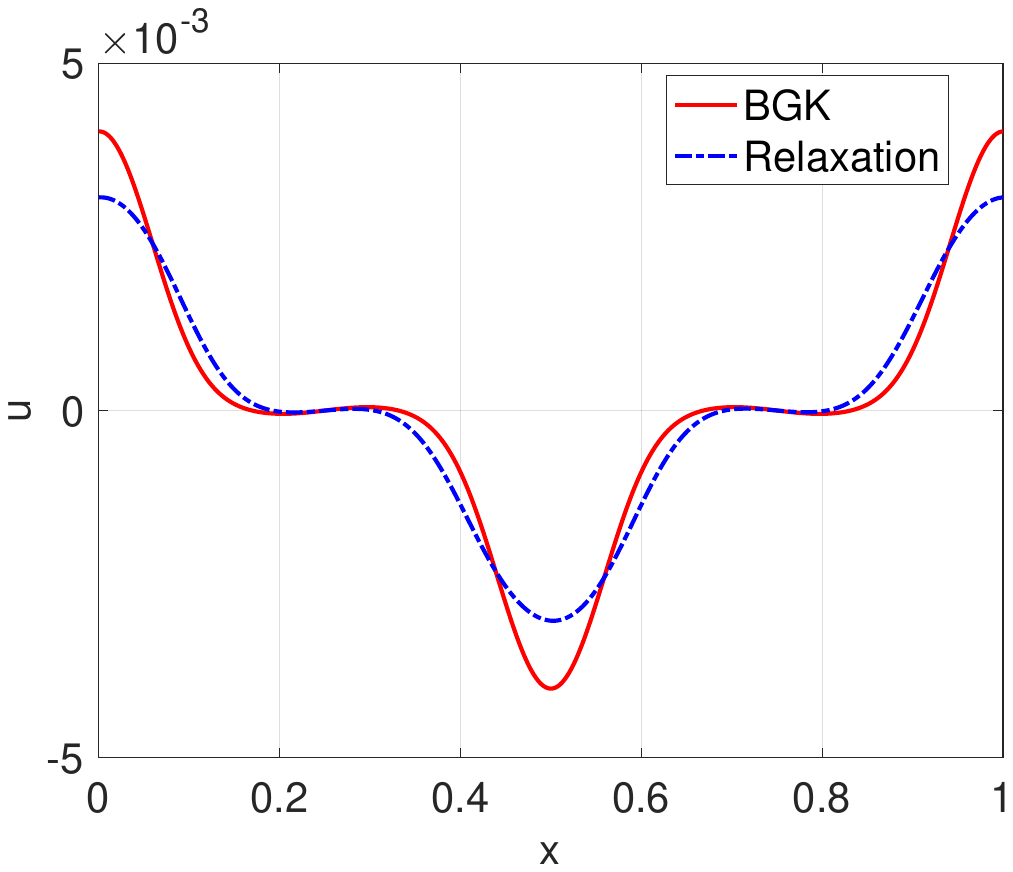}
	\includegraphics[keepaspectratio=true, width=4cm]{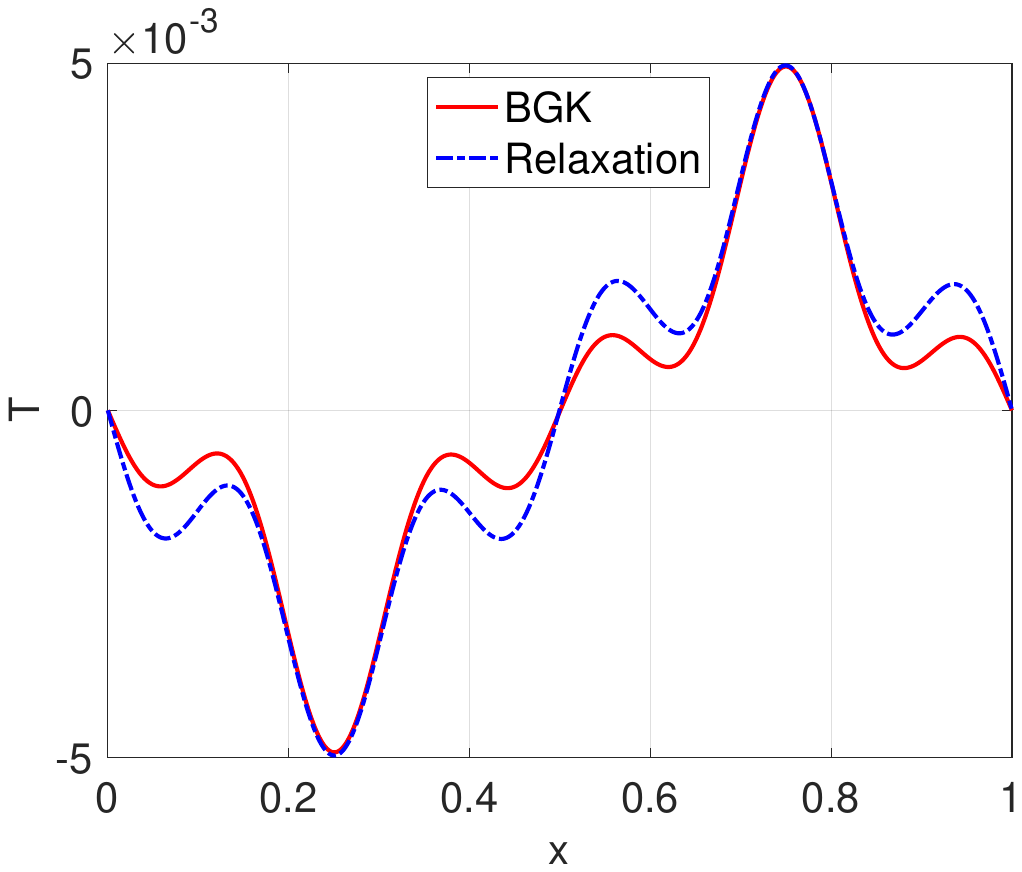}
	\\[-1.5cm]
	\caption{ Comparision of the  density, velocity and temperature obtained from the BGK equation and the relaxation system at final time $t_f = 0.1$ for $\epsilon = 0.1$, $\tau = 0.1$(first row) and $\tau = 0.01$ (second row). }
	\label{1d_eps_0dot1_t_0dot15}
\end{figure}  
\begin{figure}
	\centering
	\includegraphics[keepaspectratio=true, width=4cm]{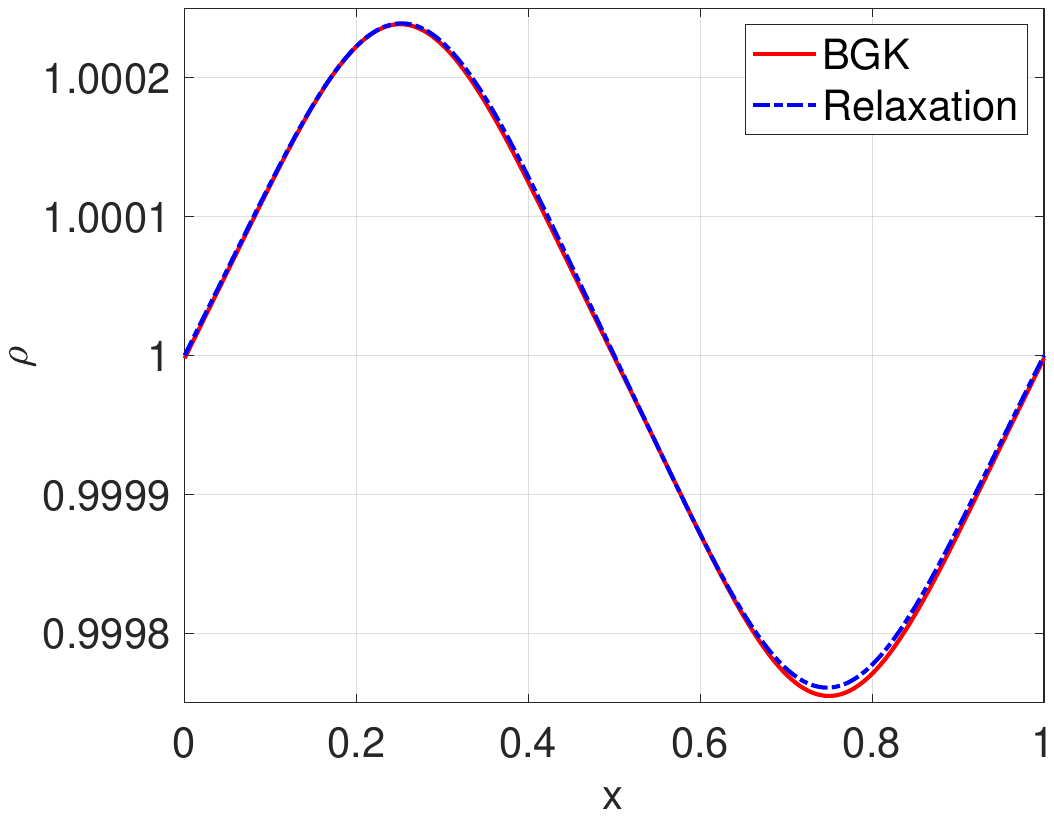}
	\includegraphics[keepaspectratio=true, width=4cm]{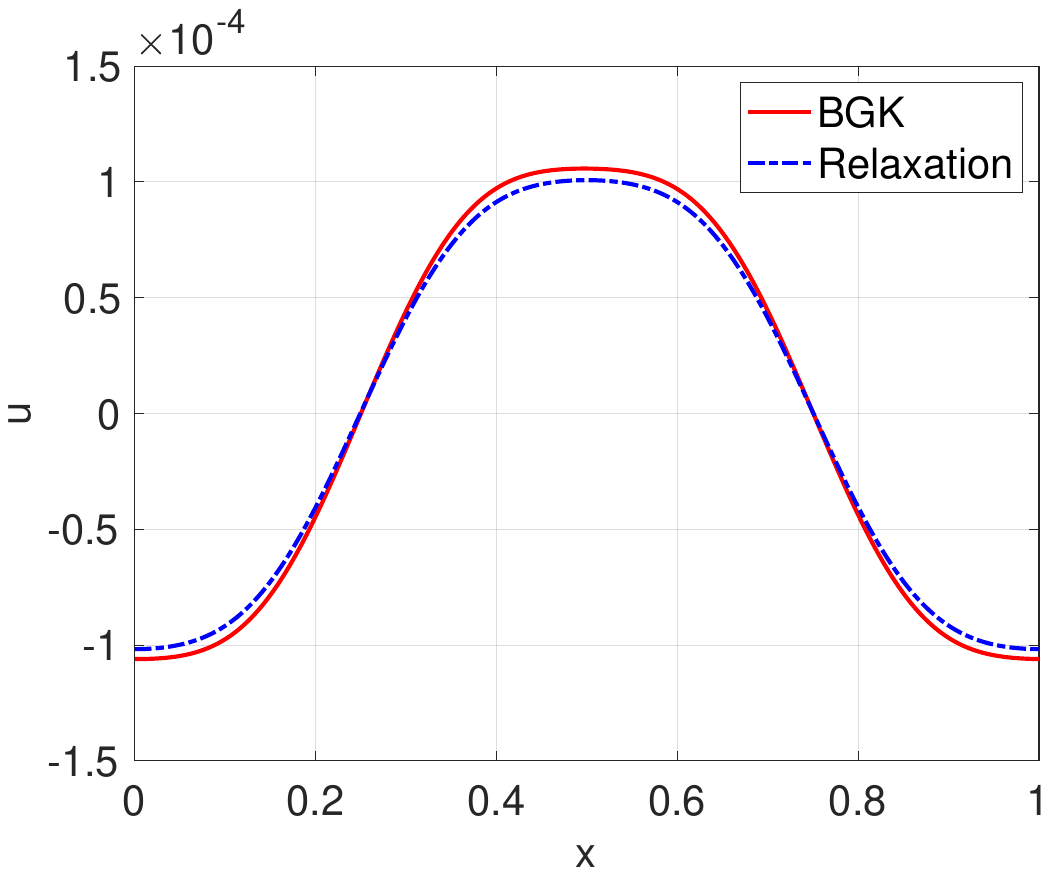}
	\includegraphics[keepaspectratio=true, width=4cm]{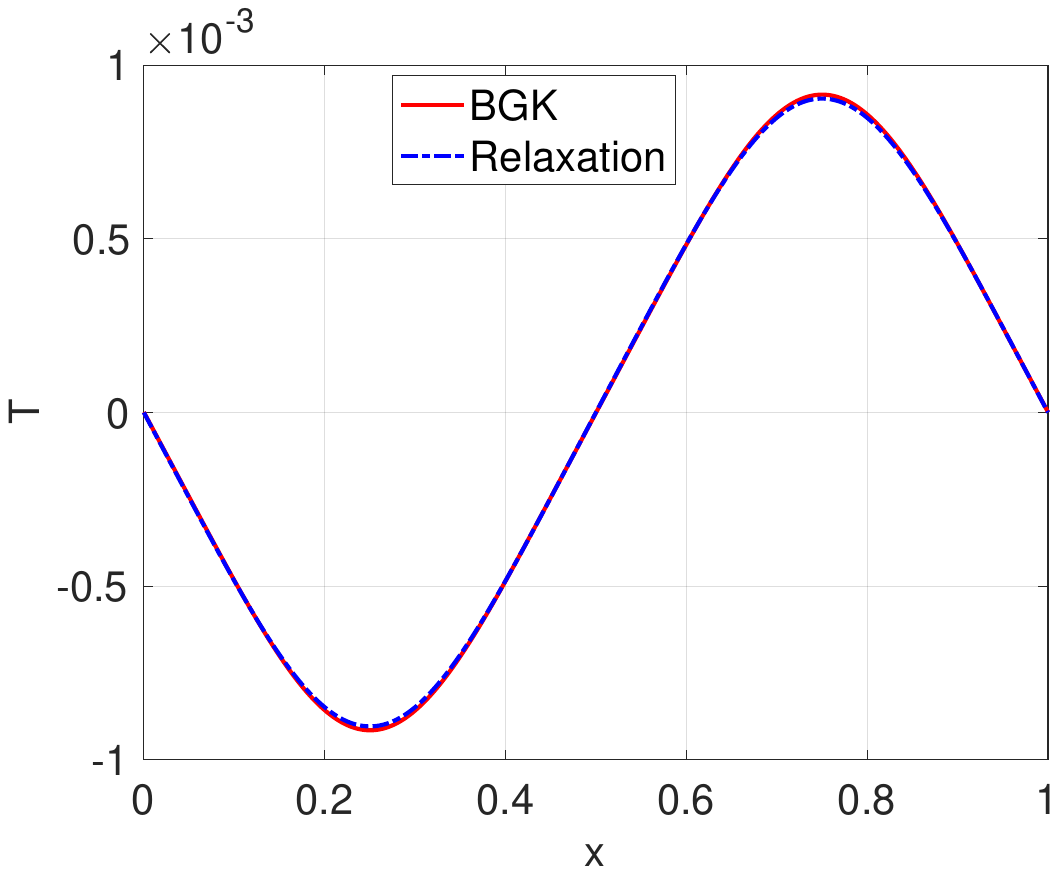}
	\\[-2cm]
	\includegraphics[keepaspectratio=true, width=4cm]{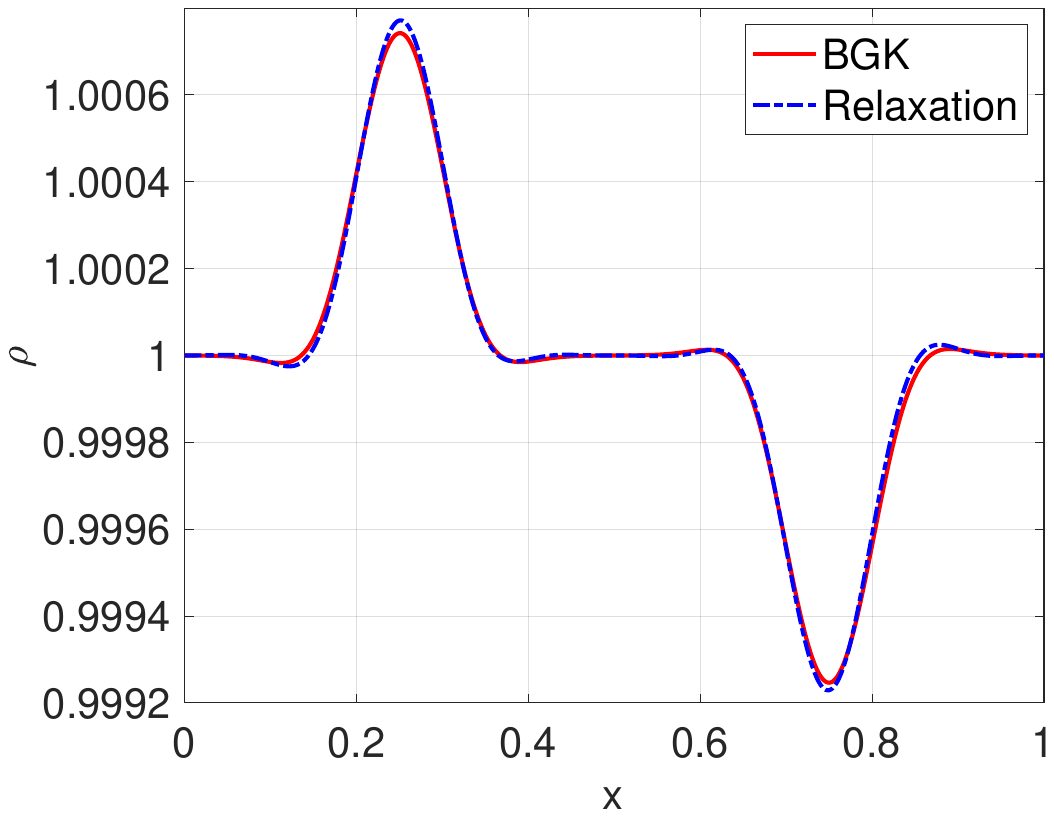}
	\includegraphics[keepaspectratio=true, width=4cm]{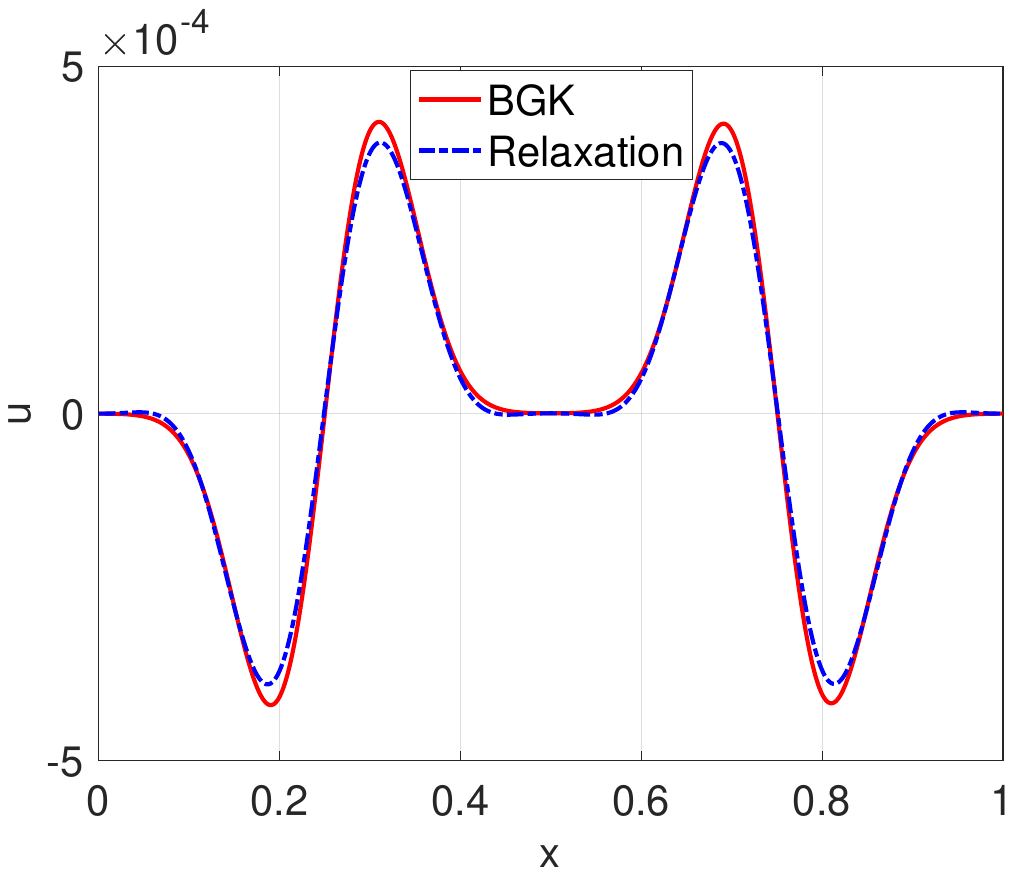}
	\includegraphics[keepaspectratio=true, width=4cm]{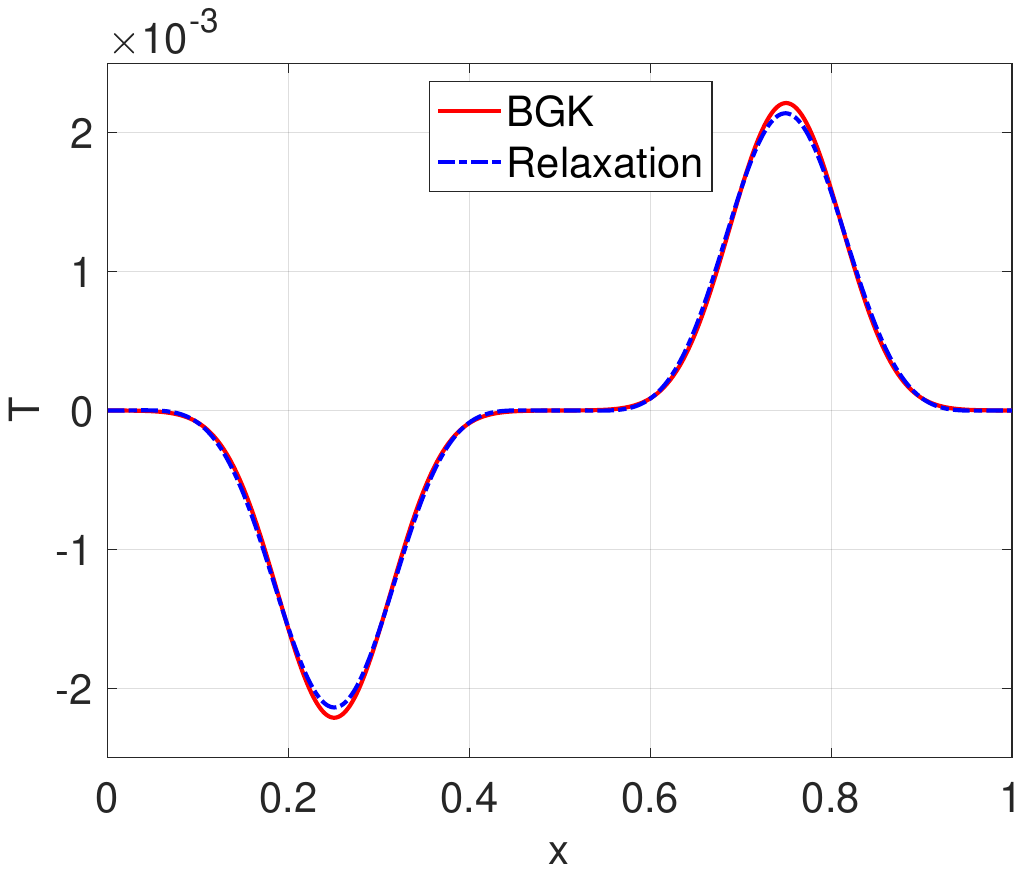}
	\\[-1.5cm]
	\caption{ Comparision of the  density, velocity and temperature obtained from the BGK equation and the relaxation system at final time $t_f = 0.1$ for $\epsilon = 0.05$, $\tau = 0.1$(first row) and $\tau = 0.01$ (second row). }
	\label{1d_eps_0dot05_t_0dot15}
\end{figure}  

\begin{remark}[Comparison of CPU times]
On the one hand, the CPU time for the BGK simulation \cite{TKR} obviously is proportional to the number $N_v$ of velocity discretization points. 
Moreover, at a fixed spatial discretization and fixed final time $t_{final}$, the computational cost for this conventional CFL-limited solver is inversely proportional to $\epsilon$ due to the restrictive CFL condition \eqref{CFL}.
On the other hand, the computation of the relaxation model, using the above-described AP scheme, is independent of $\epsilon$ and obviously does not depend on $N_v$.

For the comparison of the CPU times, we have performed simulations of the BGK equation and the relaxation model for $\tau= 0.01$, using a total number of $1000$ spatial grid points and a final time $t_{final} = 0.01$. 
For the conventional BGK model at $\epsilon=1$, we considered a CFL number of $0.1$, which corresponds to a time step $\Delta t = 10^{-5}$, taking $N_v = 20$ velocity points. The same time discretization has been used for the relaxation system for comparison.
As shown in Table \ref{tab1}, for these parameters the CPU times for the relaxation system and the BGK equation are essentially the same (101 s vs. 110 s).

However, for smaller values of $\epsilon$, the computation times for the BGK equation become increasingly longer. 
According to Table \ref{tab1}, decreasing $\epsilon$ from $1$ to $0.1$ increases the BGK computation time from 110 seconds to 19 minutes (a factor of roughly 10.4). Decreasing $\epsilon$ further to $0.05$ results in a CPU time of 38.8 minutes (a factor of roughly 21.2). 
This behavior perfectly reflects the expected inverse proportionality $O(1/\epsilon)$.
The same linear increase in cost would apply if the velocity discretization were changed from $N_v=20$ to $N_v=200$.

Ultimately, for small values of $\epsilon$ or large values of $N_v$, the proposed AP scheme leads to a considerable speedup for the relaxation model compared to the conventional CFL-limited solver \cite{TKR}. We emphasize that this speedup claim applies specifically to the standard explicit discretization; no timing comparison is made here with implicit AP BGK solvers such as \cite{2ndPaper}. 
The simulations were performed on macOS 26.5 with an Apple M5 Pro processor and 48 GB of memory. 
\end{remark}

%\begin{table}
%	\begin{center}
%		\begin{tabular}{|r|r|r|r}
%			\hline
%			$\epsilon$ & BGK &  Relaxation   \\ \hline
%			1  &   110 s&   101 s                                             \\ 
%			0.1    &  190 min &     101 s      \\
%			0.05     &  388 min &  101 s       \\                                  
%			\hline
%		\end{tabular}
%		\caption{Comparison of computation times BGK equation with relaxation and kinetic relaxation system.}
%		\label{tab1}
%	\end{center}
%\end{table}

\begin{table}[ht]
    \centering
    \begin{tabular}{|c|ccc|cc|c|}
        \hline
        $\epsilon$ & $N_v$ & $\Delta t$ & $t_{final}$ & BGK & Relaxation & Time Data \\
        \hline
        1    & 20 & $10^{-5}$ & 0.01 & 110 s   & 101 s & Measured \\
        0.1  & 20 & $10^{-6}$ & 0.01 & 19 min & 101 s & Measured \\
        0.05 & 20 & $5 \cdot 10^{-7}$ & 0.01 & 38.8 min & 101 s & Measured  \\
        \hline
    \end{tabular}
    \caption{Comparison of computation times of the conventional CFL-limited BGK solver \cite{TKR} and the proposed AP relaxation scheme. Note that the relaxation scheme time-stepping remains constant as it does not resolve the microscopic velocity space or the restrictive $\epsilon$-dependent CFL condition.}
    \label{tab1}
\end{table}

%%%%%%%%%%%%%%%%%%%

%\newpage

\subsection{2D: Doubly periodic double shear layer}

For 2D, we  consider the following periodic two--dimensional model problem. Let $(x,y) \in \left[ 0,1\right]^2$ and let the initial flow to consist of a horizontal shear-layer of finite thickness, perturbed by a small amplitude vertical velocity as in \cite{BM} 

\begin{align*}
	\rho(0, x) &= 1, \\
	T(0, x) &= 0, \\ 
	u_1(0,x) &=
	\left\{ 
	\begin{array}{l}  
		\tanh(k(y-0.25)), \qquad \qquad \mbox{for} \quad \quad y \le 0.5 ,
		\\ 
		\tanh(k(0.75-y)), \qquad \qquad \mbox{for} \quad \quad y > 0.5,
	\end{array}
	\right.
	\\
	u_2 (0,x)& = \delta \sin(2\pi x),
\end{align*}

with $k = 30$ and $\delta = 0.05$. 

Note that Figures~\ref{fig:2d-1}--\ref{fig:2d-7} display the contours of the vorticity $\omega$, which is evaluated from the velocity components $u_1$ and $u_2$ as:$$\omega = \frac{\partial u_2}{\partial x_1} - \frac{\partial u_1}{\partial x_2}.$$
\subsubsection{Accuracy}
A convergence study for the numerical solution of the relaxation model  is shown in Tables~\ref{tab:convergence_u1}--\ref{tab:convergence_p} for the second and third order IMEX-RK methods (see \citep{2ndPaper} for further details on the specific IMEX-RK tableaux used here) together with third and fifth order WENO reconstructions as well as second and fourth order central difference approximations for the spatial discretization. 
Unlike conventional explicit BGK solvers, which are bound to a restrictive, $\epsilon$-dependent CFL condition of the form $\Delta t = \text{CFL} \frac{\Delta x \epsilon}{v_{\max}}$, the proposed AP relaxation method is independent of $\epsilon$ and does not need to resolve the microscopic velocity scale or stiff $\epsilon$-scaled CFL constraints. 
Due to the parabolic time-step restriction, the time step scales quadratically with the grid spacing as $\Delta t(h) = \mathcal{O}(h^2)$ during the spatial and temporal refinement studies (varying grid spacing $h = \Delta x$) to maintain stability and the designed order of accuracy.

In order to evaluate the convergence properties and accuracy of the numerical schemes, the global discretization error is measured by $L^1$, $L^2$ and $L^\infty$ norms. The reference solution is obtained from the numerical approximation computed on a fine mesh of $513 \times 513$ grid points. 
It should be noted that the convergence results presented in Tables~\ref{tab:convergence_u1}--\ref{tab:convergence_p} strictly refer to the asymptotic regime with $\epsilon = 10^{-6}$ and $\tau = 0$. These results do not automatically imply universal, high-order accuracy across all physical regimes. The limiting case $\tau = 0$ is evaluated directly via the limiting algebraic scheme.
The results reported in Tables~\ref{tab:convergence_u1}--\ref{tab:convergence_p} confirm the expected order of convergence.

\subsubsection{Comparison to BGK}
For the numerical examples in this subsection, we use again the BGK algorithm from \cite{TKR} with  a $100\times100$ spatial grid and evaluate the performance for   $\epsilon = 0.5$ and $\epsilon = 0.1$ in order to compare  the relaxation system on the same spatial grid with a reference solution computed by the BGK solver. As already mentioned before, the reference solution computed by the BGK solver is only second order accurate. Therefore, we utilize a comparable second-order version of our algorithm for validation of our results. More specifically, we obtain this second-order version by using second-order central difference approximations combined with a minmod limiter for spatial reconstruction instead of using Lax Friedrichs flux functions.
% and WENO reconstructions like in our higher order implementation. 

%As shown in Figures~\ref{fig:2d-1}--\ref{fig:2d-8}, the relaxation system provides a good approximation, especially for smaller values of $\epsilon$, without the need of further discretization of a velocity space. 

In  Figures~\ref{fig:2d-1}--\ref{fig:2d-3}  we compare the solution of the relaxation system with our reference solution, obtained by the BGK solver \cite{TKR}, using the  parameters $\epsilon = 0.5, 0.1, $ and $\tau = 0.01,0.001$.    
The agreement is in general good for $\epsilon=0.1$ and also for $\epsilon =0.5$ and  $\tau=0.01$ as in 
Figure~\ref{fig:2d-1}.  For $\epsilon=0.5$ and $\tau=0.001$ as in Figure \ref{fig:2d-2}  on the other hand, larger deviations of the relaxation system solution from the BGK solution is observed in good agreement with the analytical observation, see
Remark \ref{hyplimit}.

%%%%%%%%%%%% TABLES %%%%%%%%%%%%%
\begin{table}[htbp]
    \centering
    \small
    \begin{tabular}{@{} l c cc cc cc @{}}
        \toprule
        \multirow{2}{*}{\textbf{Method}} & \multirow{2}{*}{$N$} & \multicolumn{2}{c}{\textbf{$L^1$ Error}} & \multicolumn{2}{c}{\textbf{$L^2$ Error}} & \multicolumn{2}{c}{\textbf{$L^\infty$ Error}} \\
        \cmidrule(lr){3-4} \cmidrule(lr){5-6} \cmidrule(lr){7-8}
        & & Error & EOC & Error & EOC & Error & EOC \\
        \midrule
        % --- METHOD C ---
        \multirow{4}{*}{\textbf{IMEX RK2}} 
& 33 & 5.56e-02 & NaN & 8.71e-02 & NaN & 1.95e-01 & NaN \\
& 65 & 1.70e-02 & 1.71 & 2.91e-02 & 1.58 & 7.53e-02 & 1.38 \\
& 129 & 3.28e-03 & 2.38 & 5.69e-03 & 2.35 & 1.61e-02 & 2.22 \\
& 257 & 5.71e-04 & 2.52 & 1.02e-03 & 2.47 & 2.82e-03 & 2.51 \\
        
        \addlinespace[0.3cm]
        % --- METHOD A ---
        \multirow{4}{*}{\textbf{IMEX RK3}} 
& 33 & 5.13e-02 & NaN & 8.23e-02 & NaN & 1.92e-01 & NaN \\
& 65 & 6.53e-03 & 2.97 & 1.17e-02 & 2.82 & 3.17e-02 & 2.60 \\
& 129 & 6.44e-04 & 3.34 & 1.42e-03 & 3.03 & 6.40e-03 & 2.31 \\
& 257 & 1.37e-05 & 5.56 & 3.52e-05 & 5.34 & 1.69e-04 & 5.24 \\
 \\
%        \addlinespace[0.3cm]
        
%        % --- METHOD C ---
%        \multirow{4}{*}{\textbf{"Standard" INS Solver}} 
%        & 33  & 2.278e-1 & --   & 9.80e-3 & --   & 1.40e-2 & --   \\
%       & 65  & 6.84e-2 & 1.7353 & 6.12e-4 & 1.6829 & 8.75e-4 & 1.6181 \\
%       & 129  & 3.71e-5 & 1.9317 & 3.83e-5 & 1.9053 & 5.47e-5 & 1.8576 \\
%        & 257  & 2.32e-6 & 1.9879 & 2.39e-6 & 1.9838 & 3.42e-6 & 1.9770 \\
        
        \bottomrule
    \end{tabular}
    \caption{Convergence analysis for the double shear layer test case with $\varepsilon = 10^{-6}$ and $\tau = 0$ at $t = 1$ showing variable $u_1$ and using two different numerical methods, namely IMEX RK2 with WENO3 and IMEX RK3 with WENO5. EOC denotes the Experimental Order of Convergence. Global discretization errors ($L^1$-, $L^2$-, and $L^\infty$-norms) are evaluated by comparing coarser grid solutions against a high-resolution numerical reference solution computed on a fine grid of $513 \times 513$ points.}
    \label{tab:convergence_u1}
\end{table}

\begin{table}[htbp]
    \centering
    \small
    \begin{tabular}{@{} l c cc cc cc @{}}
        \toprule
        \multirow{2}{*}{\textbf{Method}} & \multirow{2}{*}{$N$} & \multicolumn{2}{c}{\textbf{$L^1$ Error}} & \multicolumn{2}{c}{\textbf{$L^2$ Error}} & \multicolumn{2}{c}{\textbf{$L^\infty$ Error}} \\
        \cmidrule(lr){3-4} \cmidrule(lr){5-6} \cmidrule(lr){7-8}
        & & Error & EOC & Error & EOC & Error & EOC \\
        \midrule
        % --- METHOD C ---
        \multirow{4}{*}{\textbf{IMEX RK2}} 
& 33 & 3.19e-03 & NaN & 5.10e-03 & NaN & 1.89e-02 & NaN \\
& 65 & 1.85e-03 & 0.78 & 3.24e-03 & 0.65 & 1.11e-02 & 0.78 \\
& 129 & 3.95e-04 & 2.23 & 7.05e-04 & 2.20 & 3.31e-03 & 1.74 \\
& 257 & 4.93e-05 & 3.00 & 1.00e-04 & 2.81 & 5.53e-04 & 2.58 \\
        
        \addlinespace[0.3cm]
        % --- METHOD A ---
        \multirow{4}{*}{\textbf{IMEX RK3}} 
& 33 & 2.21e-03 & NaN & 3.69e-03 & NaN & 1.32e-02 & NaN \\
& 65 & 5.61e-04 & 1.98 & 9.36e-04 & 1.98 & 3.50e-03 & 1.91 \\
& 129 & 2.36e-05 & 4.57 & 4.78e-05 & 4.29 & 2.66e-04 & 3.72 \\
& 257 & 5.49e-07 & 5.42 & 1.19e-06 & 5.32 & 6.86e-06 & 5.28 \\
 \\
%        \addlinespace[0.3cm]
        
%        % --- METHOD C ---
%        \multirow{4}{*}{\textbf{"Standard" INS Solver}} 
%        & 33  & 2.278e-1 & --   & 9.80e-3 & --   & 1.40e-2 & --   \\
%       & 65  & 6.84e-2 & 1.7353 & 6.12e-4 & 1.6829 & 8.75e-4 & 1.6181 \\
%       & 129  & 3.71e-5 & 1.9317 & 3.83e-5 & 1.9053 & 5.47e-5 & 1.8576 \\
%        & 257  & 2.32e-6 & 1.9879 & 2.39e-6 & 1.9838 & 3.42e-6 & 1.9770 \\
        
        \bottomrule
    \end{tabular}
    \caption{Convergence analysis for the double shear layer test case with $\varepsilon = 10^{-6}$ and $\tau = 0$ at $t = 1$ showing variable $u_2$ and using two different numerical methods, namely IMEX RK2 with WENO3 and IMEX RK3 with WENO5. EOC denotes the Experimental Order of Convergence. Global discretization errors ($L^1$-, $L^2$-, and $L^\infty$-norms) are evaluated by comparing coarser grid solutions against a high-resolution numerical reference solution computed on a fine grid of $513 \times 513$ points.}
        \label{tab:convergence_u2}
\end{table}

\begin{table}[htbp]
    \centering
    \small
    \begin{tabular}{@{} l c cc cc cc @{}}
        \toprule
        \multirow{2}{*}{\textbf{Method}} & \multirow{2}{*}{$N$} & \multicolumn{2}{c}{\textbf{$L^1$ Error}} & \multicolumn{2}{c}{\textbf{$L^2$ Error}} & \multicolumn{2}{c}{\textbf{$L^\infty$ Error}} \\
        \cmidrule(lr){3-4} \cmidrule(lr){5-6} \cmidrule(lr){7-8}
        & & Error & EOC & Error & EOC & Error & EOC \\
        \midrule
        % --- METHOD C ---
        \multirow{4}{*}{\textbf{IMEX RK2}} 
& 33 & 4.93e-02 & NaN & 5.43e-02 & NaN & 1.48e-01 & NaN \\
& 65 & 3.92e-02 & 0.33 & 6.09e-02 & -0.16 & 2.26e-01 & -0.61 \\
& 129 & 9.50e-03 & 2.04 & 1.89e-02 & 1.68 & 9.45e-02 & 1.26 \\
& 257 & 1.62e-03 & 2.55 & 3.07e-03 & 2.63 & 1.66e-02 & 2.50 \\
        
        \addlinespace[0.3cm]
        % --- METHOD A ---
        \multirow{4}{*}{\textbf{IMEX RK3}} 
& 33 & 4.39e-02 & NaN & 5.05e-02 & NaN & 1.26e-01 & NaN \\
& 65 & 1.01e-02 & 2.12 & 1.50e-02 & 1.75 & 6.76e-02 & 0.90 \\
& 129 & 5.43e-04 & 4.22 & 8.63e-04 & 4.12 & 3.46e-03 & 4.29 \\
& 257 & 8.13e-05 & 2.74 & 8.29e-05 & 3.38 & 1.92e-04 & 4.17 \\
 \\
%        \addlinespace[0.3cm]
        
%        % --- METHOD C ---
%        \multirow{4}{*}{\textbf{"Standard" INS Solver}} 
%        & 33  & 2.278e-1 & --   & 9.80e-3 & --   & 1.40e-2 & --   \\
%       & 65  & 6.84e-2 & 1.7353 & 6.12e-4 & 1.6829 & 8.75e-4 & 1.6181 \\
%       & 129  & 3.71e-5 & 1.9317 & 3.83e-5 & 1.9053 & 5.47e-5 & 1.8576 \\
%        & 257  & 2.32e-6 & 1.9879 & 2.39e-6 & 1.9838 & 3.42e-6 & 1.9770 \\
        
        \bottomrule
    \end{tabular}
    \caption{Convergence analysis for the double shear layer test case with $\varepsilon = 10^{-6}$ and $\tau = 0$ at $t = 1$ showing variable $p$ and using two different numerical methods, namely IMEX RK2 with WENO3 and IMEX RK3 with WENO5. EOC denotes the Experimental Order of Convergence. Global discretization errors ($L^1$-, $L^2$-, and $L^\infty$-norms) are evaluated by comparing coarser grid solutions against a high-resolution numerical reference solution computed on a fine grid of $513 \times 513$ points.}
        \label{tab:convergence_p}
\end{table}
\begin{figure}[H]
	\centering
	\begin{subfigure}{0.45\textwidth}
		\centering
		\includegraphics[width=\linewidth]{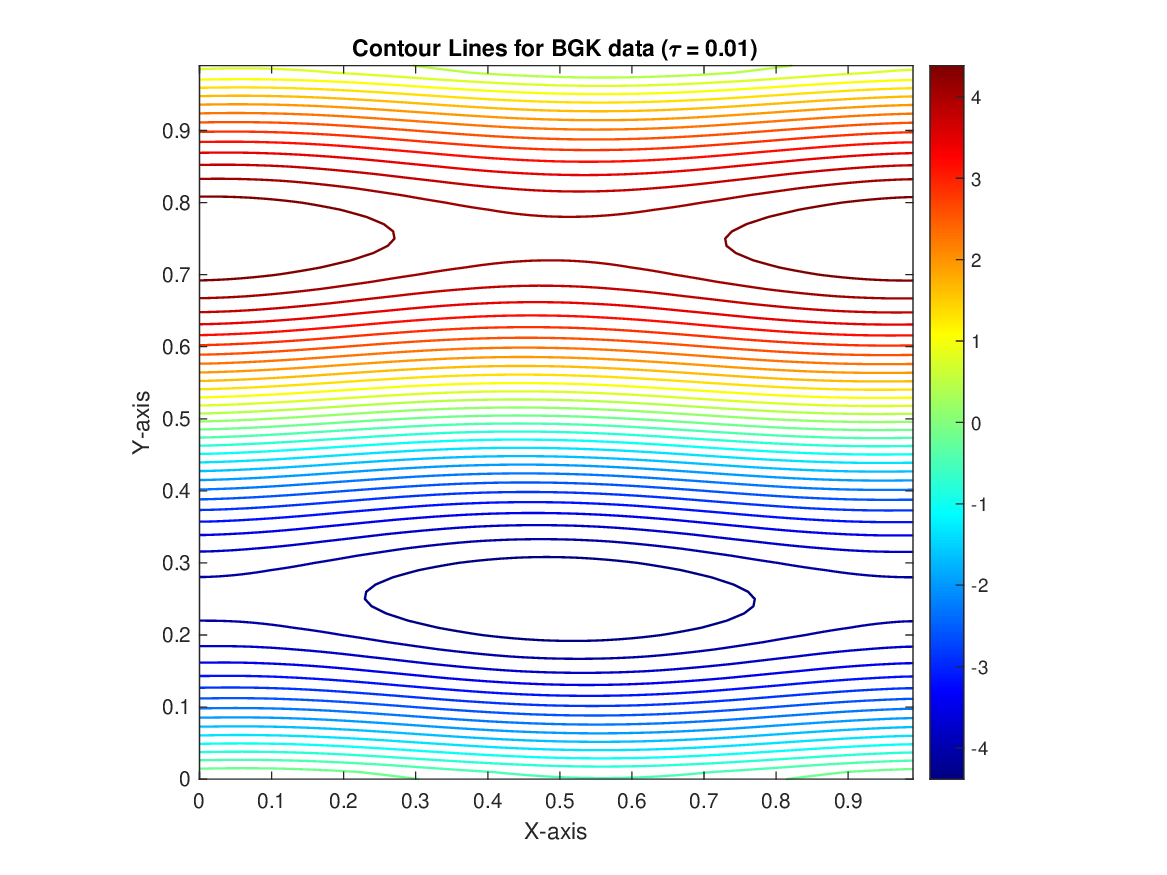}
%		\caption{$\varepsilon = 0.5$,v-Raum, Diff exp. - Algo 1}
	\end{subfigure}
	\centering
	\begin{subfigure}{0.45\textwidth}
		\centering
	\includegraphics[width=\linewidth]{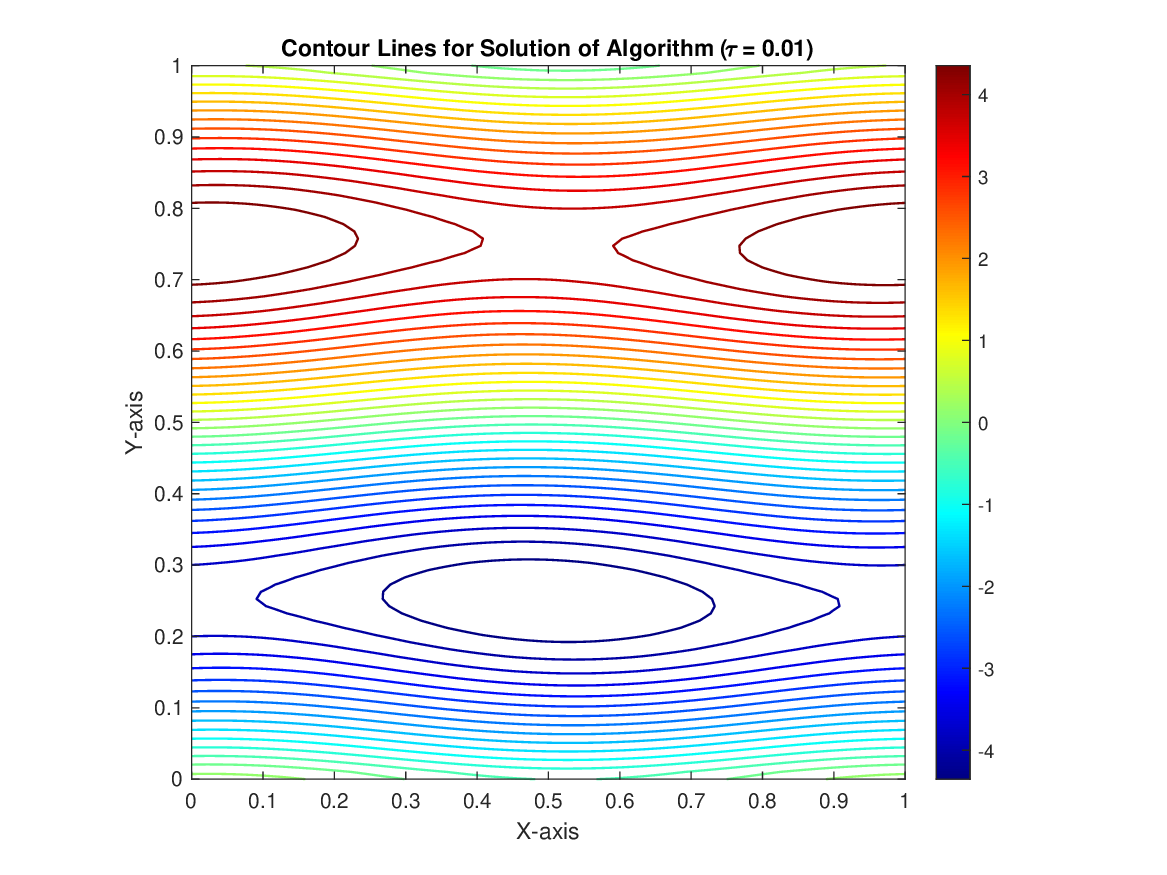}
	\end{subfigure}
	\caption{Contour lines of the vorticity $\omega$, computed from the velocity components $u_1$ and $u_2$, for the shear layer test case for $\epsilon = 0.5$ and $\tau = 0.01$ at $t = 1.5$. Left: BGK \cite{TKR}, right: relaxation system with IMEX RK2, minmod}
	\label{fig:2d-1}
\end{figure}

%\begin{figure}[H]
%	\centering
%	\begin{subfigure}{0.45\textwidth}
%		\centering
%		\includegraphics[width=\linewidth]{figures/shear_layer/eps_0dot5/tau_0dot01/shear_layer_RK2WENO3eps_0dot5_tau_0dot01_t_1dot5.eps}
%%		\caption{$\varepsilon = 0.5$,v-Raum, Diff exp. - Algo 1}
%	\end{subfigure}
%%	\centering
%%	\begin{subfigure}{0.45\textwidth}
%%		\centering
%%%		\includegraphics[width=\linewidth]{Images/AlgoSoleps1tau00001D2}
%%	\includegraphics[width=\linewidth]{figures/shear_layer/eps_0dot1/tau_1em3/shear_layer_VielDiff_eps_0dot1_tau_0dot001_t_1dot5.eps}
%%
%%	\end{subfigure}
%	\caption{Solution of the shear layer test case for $\epsilon = 0.5$ and $\tau = 0.01$ at $t = 1.5$. Left: IMEX RK2, WENO3, right: IMEX RK3, WENO5}
%	\label{fig:2d-2}
%\end{figure}

%%%%%%%%%%%%%%%%%%%%%%%%
\begin{figure}[H]
	\centering
	\begin{subfigure}{0.45\textwidth}
		\centering
		\includegraphics[width=\linewidth]{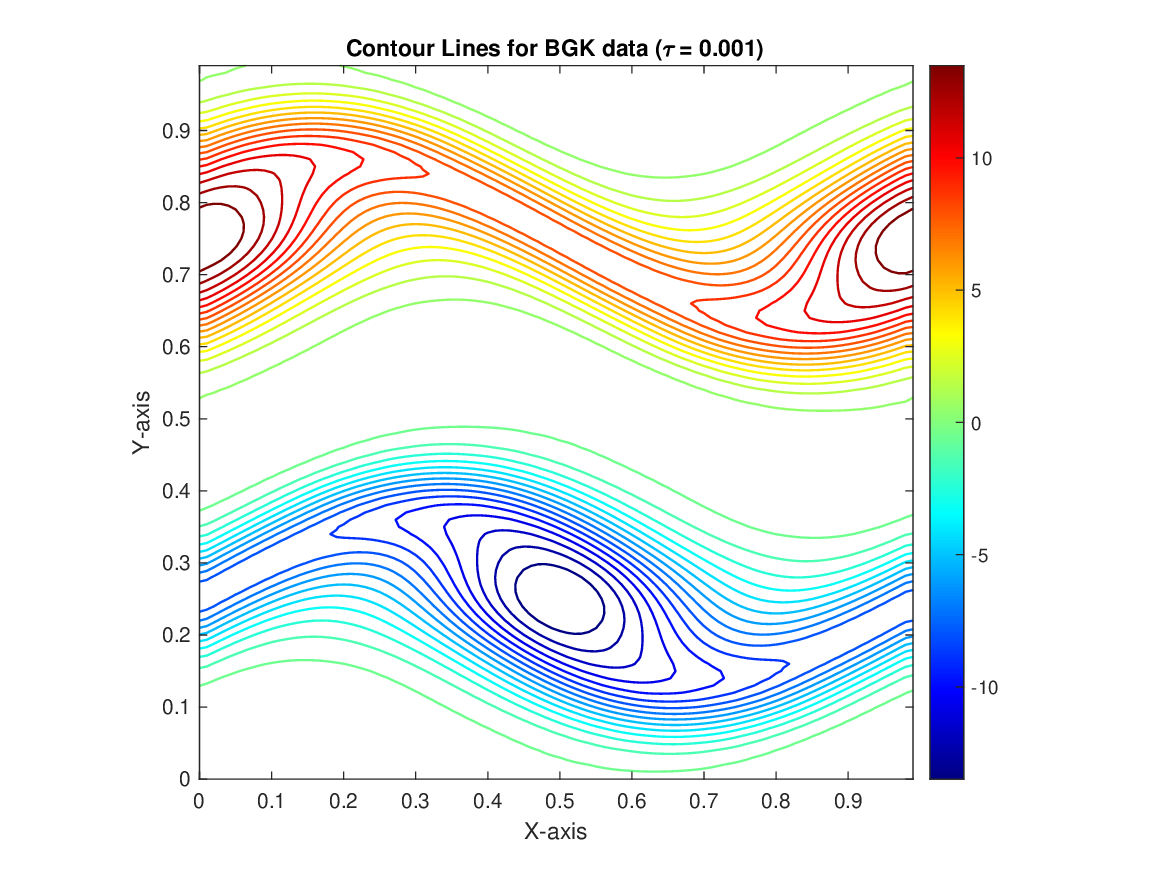}
%		\caption{$\varepsilon = 0.5$,v-Raum, Diff exp. - Algo 1}
	\end{subfigure}
	\centering
	\begin{subfigure}{0.45\textwidth}
		\centering
	\includegraphics[width=\linewidth]{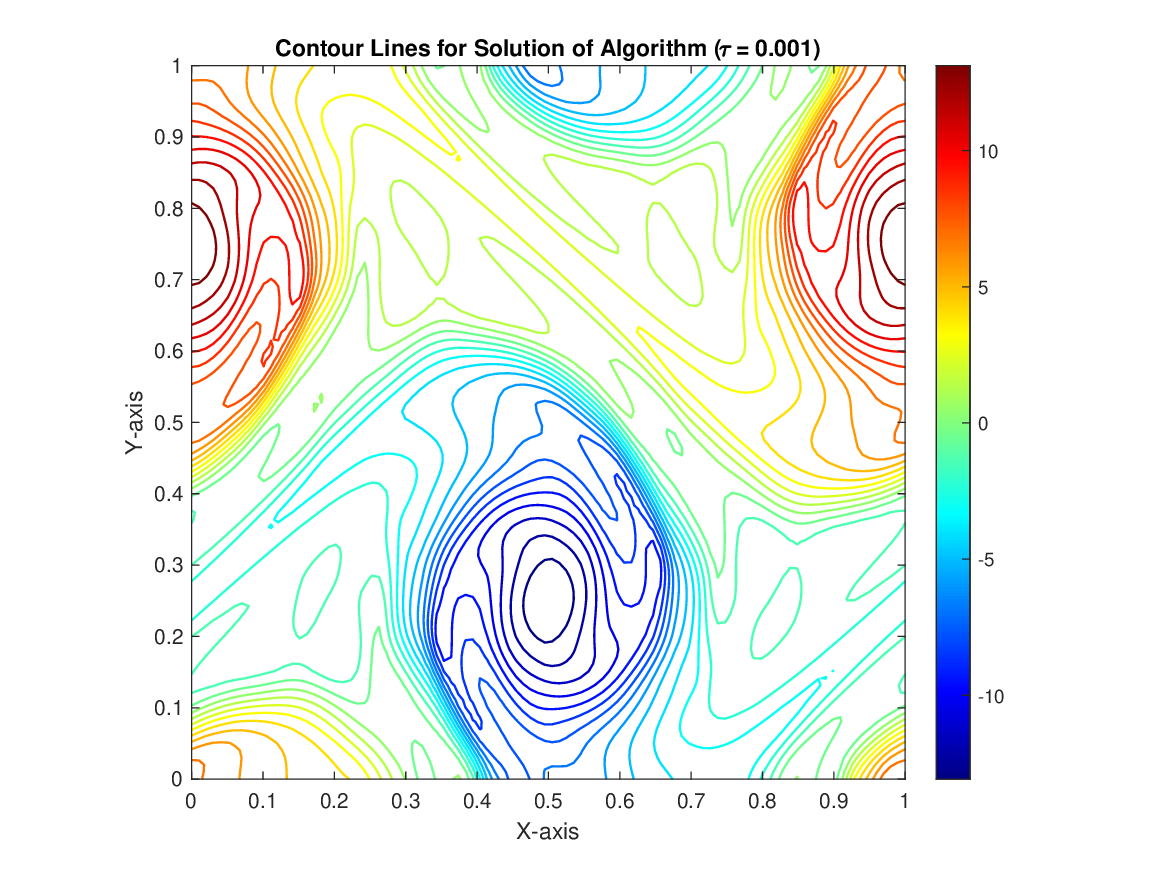}

	\end{subfigure}
	\caption{Contour lines of the vorticity $\omega$, computed from the velocity components $u_1$ and $u_2$, for the shear layer test case for $\epsilon = 0.5$ and $\tau = 0.001$ at $t = 1.5$. Left: BGK \cite{TKR}, right: relaxation system with IMEX RK2, minmod}
	\label{fig:2d-2}
\end{figure}

%\begin{figure}[H]
%	\centering
%	\begin{subfigure}{0.45\textwidth}
%		\centering
%		\includegraphics[width=\linewidth]{figures/shear_layer/eps_0dot5/tau_0dot001/shear_layer_RK2WENO3_eps_0dot5_tau_0dot001_t_1dot5.eps}
%%		\caption{$\varepsilon = 0.5$,v-Raum, Diff exp. - Algo 1}
%	\end{subfigure}
%%	\centering
%%	\begin{subfigure}{0.45\textwidth}
%%		\centering
%%%		\includegraphics[width=\linewidth]{Images/AlgoSoleps1tau00001D2}
%%	\includegraphics[width=\linewidth]{figures/shear_layer/eps_0dot1/tau_1em3/shear_layer_VielDiff_eps_0dot1_tau_0dot001_t_1dot5.eps}
%%
%%	\end{subfigure}
%	\caption{Solution of the shear layer test case for $\epsilon = 0.5$ and $\tau = 0.001$ at $t = 1.5$. Left: IMEX RK2, WENO3, right: IMEX RK3, WENO5}
%	\label{fig:2d-4}
%\end{figure}

%%%%%%%%%%%%%%%%%%%%% 0.1 %%%%%%%%%%%%%%%%%%%%%%%%%%%%%%%%%%%%%%%%%

%%%%%%%%%%%%%%%%%%
\begin{figure}[H]
	\centering
	\begin{subfigure}{0.45\textwidth}
		\centering
	\includegraphics[width=\linewidth]{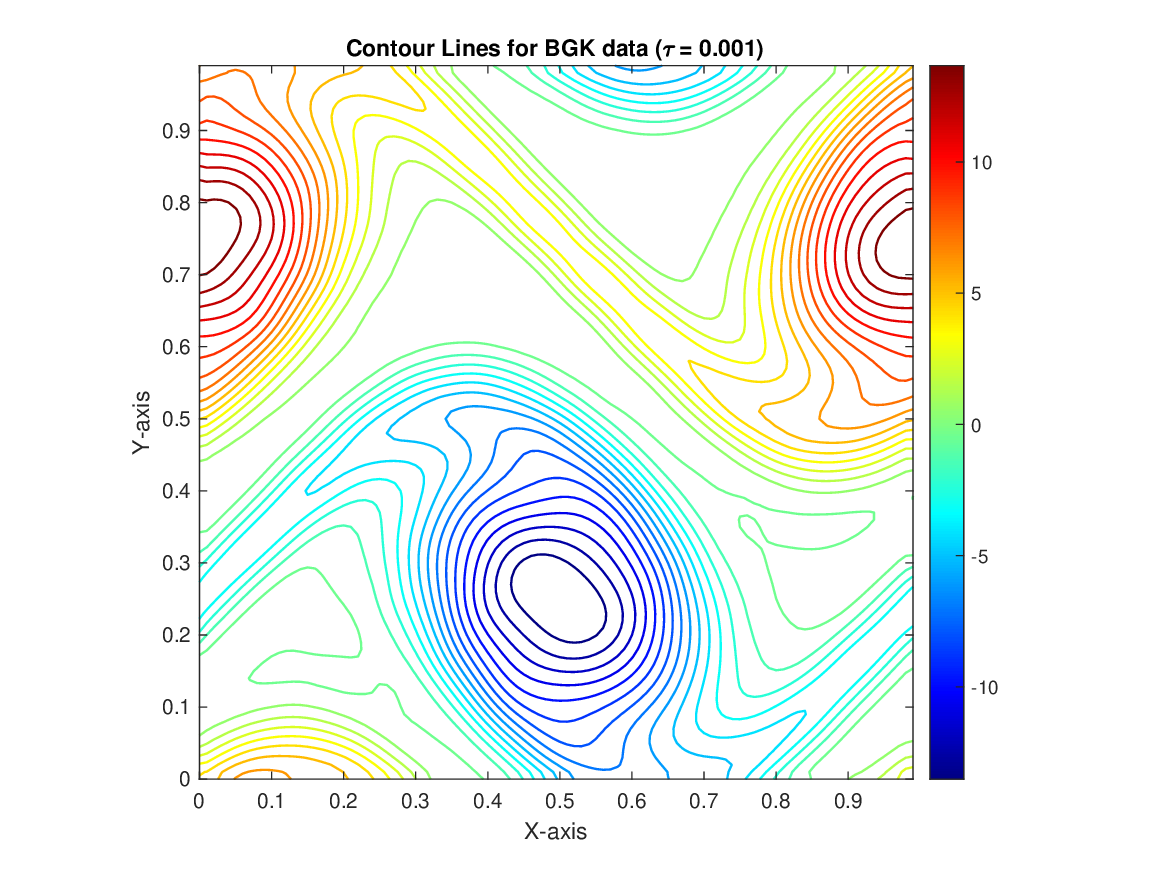}
%		\caption{$\varepsilon = 0.5$,v-Raum, Diff exp. - Algo 1}
	\end{subfigure}
	\centering
	\begin{subfigure}{0.45\textwidth}
		\centering
	\includegraphics[width=\linewidth]{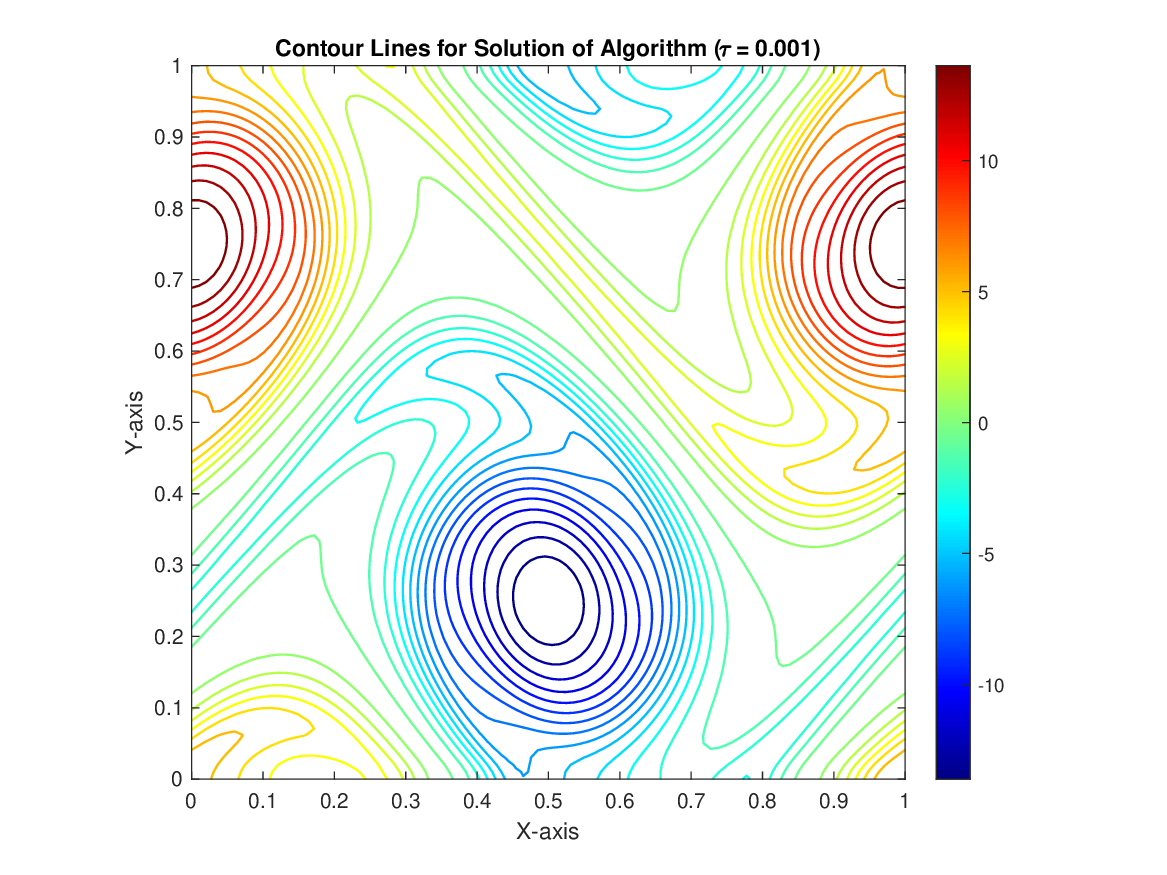}

	\end{subfigure}
	\caption{Contour lines of the vorticity $\omega$, computed from the velocity components $u_1$ and $u_2$, for the shear layer test case for $\epsilon = 0.1$ and $\tau = 0.001$ at $t = 1.5$. Left: BGK \cite{TKR}, right: relaxation system with  IMEX RK2, minmod}
	\label{fig:2d-3}
\end{figure}

\subsubsection{Comparison of higher order results}
While the comparison of a classical BGK solver with the second order solution of our model confirms the validity, further strength is lying in its high--order capabilities, that are able to capture peak amplitudes and sharp gradients, that are usually smoothed out by a second-order approach, especially for smaller values of the parameters 
$\epsilon$ and $\tau$. For these cases we use the BGK algorithm from \cite{2ndPaper} and compare it to the relaxation system. 
%For both algorithms we use a IMEX RK2 time discretization and a third order WENO spatial discretization.

We start the investigation by considering a situation with $\epsilon =0.1$ and $\tau=0.0001$. We use for both problems a discretization based on
IMEX RK2 and WENO3.
Figure~\ref{fig:2d-4} shows already a very good agreement between the solutions of the BGK equation and the relaxation system.

Moreover, we investigate the hydrodynamic limit using smaller  values of $\epsilon$ and different values of $\tau$. For these parameters, standard BGK simulations are no longer computationally feasible. The AP method from \cite{2ndPaper} works in contrast uniformly for all ranges of parameters. 
Figures~\ref{fig:2d-5} and \ref{fig:2d-6} show a situation with $\epsilon =0.01$, $\tau=0.001$ and $\tau=0.0001$ respectively.
Again IMEX RK2 and WENO3 discretizations are used. 
In these cases one observes again  a very good fit between the kinetic solution and the solution of the relaxation model.
 %The simulations confirm that our solver is working as expected also in the hydrodynamic regime and the results coincide well with those of incompressible Navier--Stokes simulations with different values of the diffusivity and with the incompressible Euler limit. As already stated, setting $\tau = 0.01$ yields the incompressible Navier--Stokes solution, while reducing it to $\tau = 0$ causes Algorithm~\ref{alg1RelSys} to approach the incompressible Euler limit.
 
%Finally, we investigate a situation with $\epsilon = 10^{-6}$ and  $\tau = 0$ at $t = 1.5$, i.e. the incompressible Euler limit. As expected, the kinetic solution and the solution of the relaxation model give completely consistent results . 
%Finally, we compare the solutions of the relaxation system based on a IMEX RK2 and WENO3 discretization and a solution using IMEX RK3 and WENO5, 
%see Figure \ref{fig:2d-7}, demonstrating the high-order capabilities of the proposed method.

To conclude, we investigate the incompressible Euler limit by setting $\epsilon = 10^{-6}$ and $\tau = 0$ at $t = 1.5$. As expected, the kinetic solution and the solution of the relaxation model yield completely consistent results. Furthermore, a comparison between the relaxation system solved via IMEX RK2 with WENO3 and IMEX RK3 with WENO5 (see Figure \ref{fig:2d-7}) clearly demonstrates the high-order capabilities of the proposed method: The higher-order IMEX RK3/WENO5 method introduces less numerical diffusion, allowing it to preserve steeper gradients and capture higher peak amplitudes of the vorticity field compared to the lower-order scheme.

\begin{figure}[H]
	\centering
	\begin{subfigure}{0.45\textwidth}
		\centering
		\includegraphics[width=\linewidth]{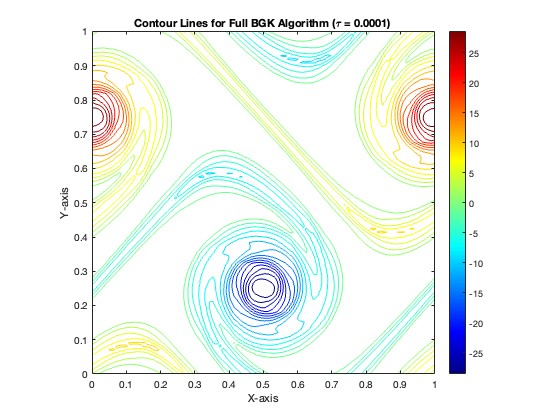}
		%		\caption{$\varepsilon = 0.5$,v-Raum, Diff exp. - Algo 1}
	\end{subfigure}
	\centering
	\begin{subfigure}{0.45\textwidth}
		\centering
		\includegraphics[width=\linewidth]{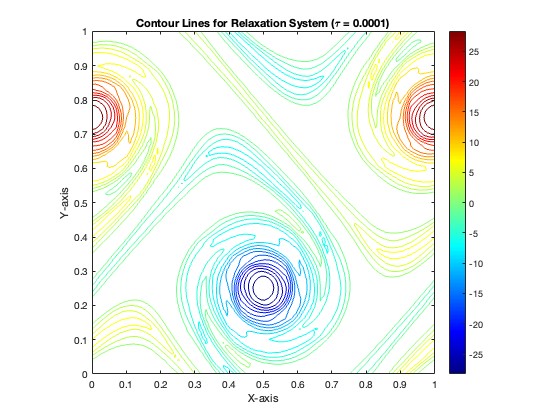}
		%		\caption{$\varepsilon = 0.5$,v-Raum, Diff imp. - Algo 2}
	\end{subfigure}
	\caption{Contour lines of the vorticity $\omega$, computed from the velocity components $u_1$ and $u_2$, for the shear layer test case for $\epsilon = 0.1$ and $\tau = 0.0001$ at $t = 1.5$. Left: BGK \cite{2ndPaper}, right: relaxation system. Both methods use   IMEX RK2, WENO3.}
	\label{fig:2d-4}
\end{figure}

\begin{figure}[H]
	\centering
	\begin{subfigure}{0.45\textwidth}
		\centering
		\includegraphics[width=\linewidth]{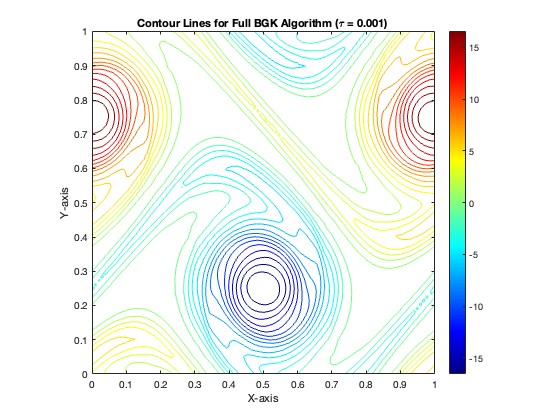}
		%		\caption{$\varepsilon = 0.5$,v-Raum, Diff exp. - Algo 1}
	\end{subfigure}
	\centering
	\begin{subfigure}{0.45\textwidth}
		\centering
		\includegraphics[width=\linewidth]{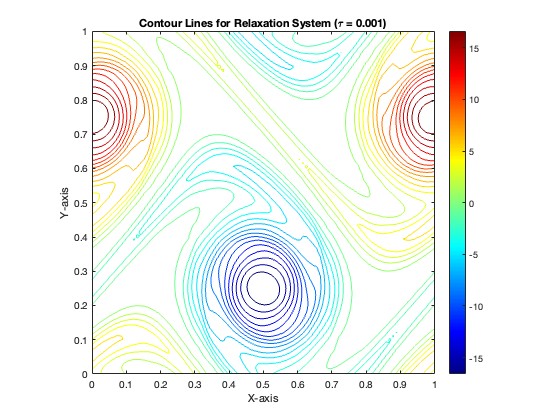}
		%		\caption{$\varepsilon = 0.5$,v-Raum, Diff imp. - Algo 2}
	\end{subfigure}
	\caption{Contour lines of the vorticity $\omega$, computed from the velocity components $u_1$ and $u_2$, for the shear layer test case for $\epsilon = 0.01$ and $\tau = 0.001$ at $t = 1.5$. Left: BGK \cite{2ndPaper}, right: relaxation system. Both methods use   IMEX RK2, WENO3.}
	\label{fig:2d-5}
\end{figure}

\begin{figure}[H]
	\centering
	\begin{subfigure}{0.45\textwidth}
		\centering
		\includegraphics[width=\linewidth]{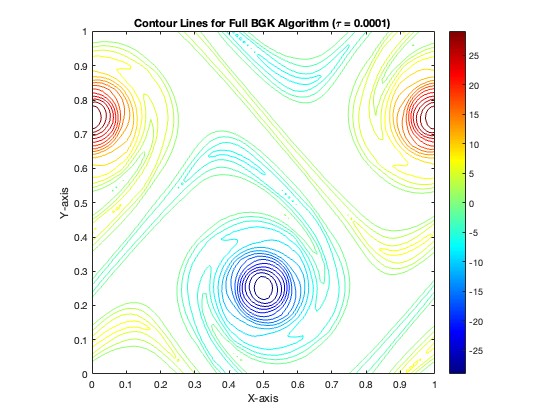}
		%		\caption{$\varepsilon = 0.5$,v-Raum, Diff exp. - Algo 1}
	\end{subfigure}
	\centering
	\begin{subfigure}{0.45\textwidth}
		\centering
		\includegraphics[width=\linewidth]{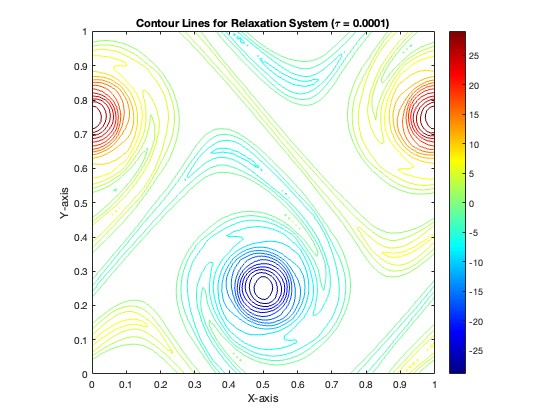}
		%		\caption{$\varepsilon = 0.5$,v-Raum, Diff imp. - Algo 2}
	\end{subfigure}
	\caption{Contour lines of the vorticity $\omega$, computed from the velocity components $u_1$ and $u_2$, for the shear layer test case for $\epsilon = 0.01$ and $\tau = 0.0001$ at $t = 1.5$. Left: BGK \cite{2ndPaper}, right: relaxation system.
		Both methods use   IMEX RK2, WENO3.}
	\label{fig:2d-6}
\end{figure}

\begin{figure}[H]
	\centering
	\begin{subfigure}{0.45\textwidth}
		\centering
		\includegraphics[width=\linewidth]{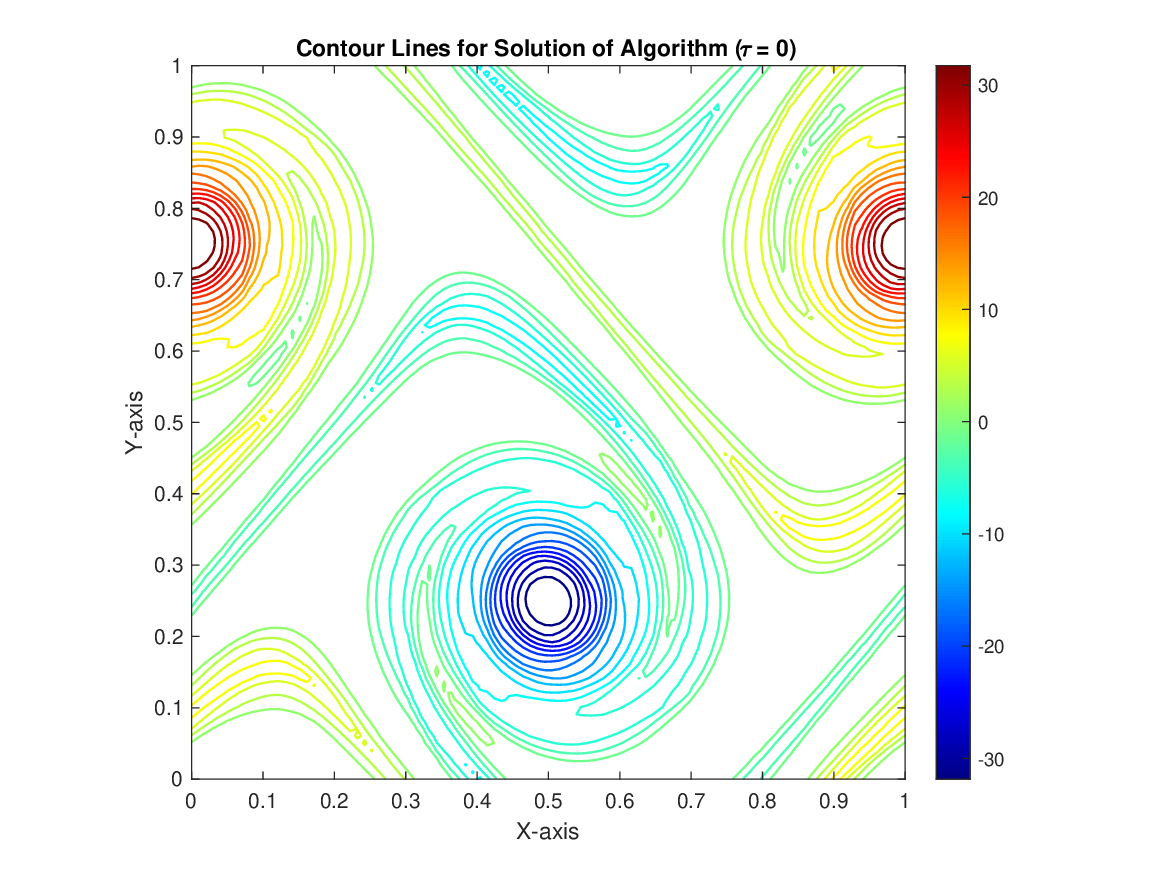}
		%		\caption{$\varepsilon = 0.5$,v-Raum, Diff exp. - Algo 1}
	\end{subfigure}
	\centering
	\begin{subfigure}{0.45\textwidth}
		\centering
		\includegraphics[width=\linewidth]{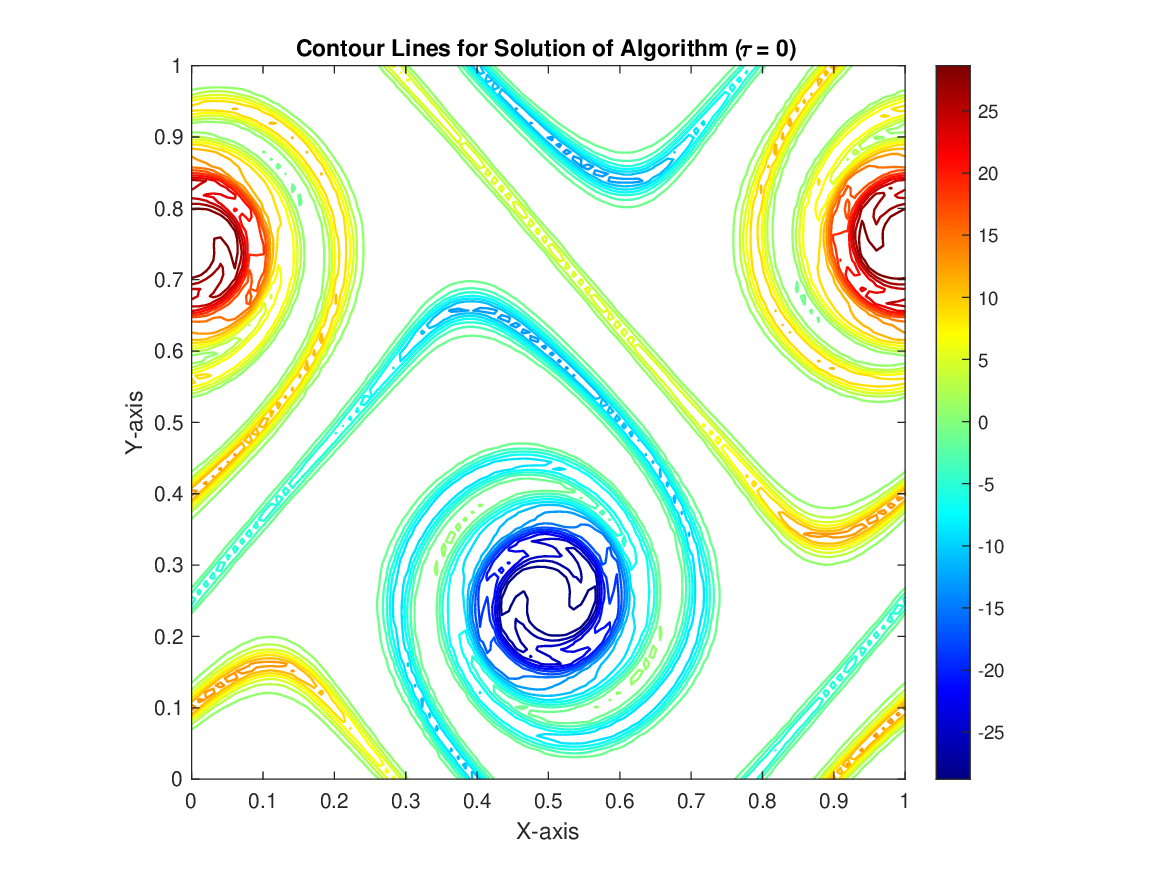}
		%		\caption{$\varepsilon = 0.5$,v-Raum, Diff imp. - Algo 2}
	\end{subfigure}
	\caption{Contour lines of the vorticity $\omega$, computed from the velocity components $u_1$ and $u_2$, for the relaxation system for the shear layer test case for $\epsilon = 10^{-6}$ and  $\tau = 0$ at $t = 1.5$. Left: IMEX RK2,WENO3, right: IMEX RK3,WENO5. The higher-order IMEX RK3/WENO5 method is able to preserve steeper gradients and capture higher peak amplitudes of the vorticity field compared to the IMEX RK2/WENO3 scheme.}
	\label{fig:2d-7}
\end{figure}

\begin{figure}[H]
	\centering
	\begin{subfigure}{1\textwidth}
		\centering
		\includegraphics[width=\linewidth]{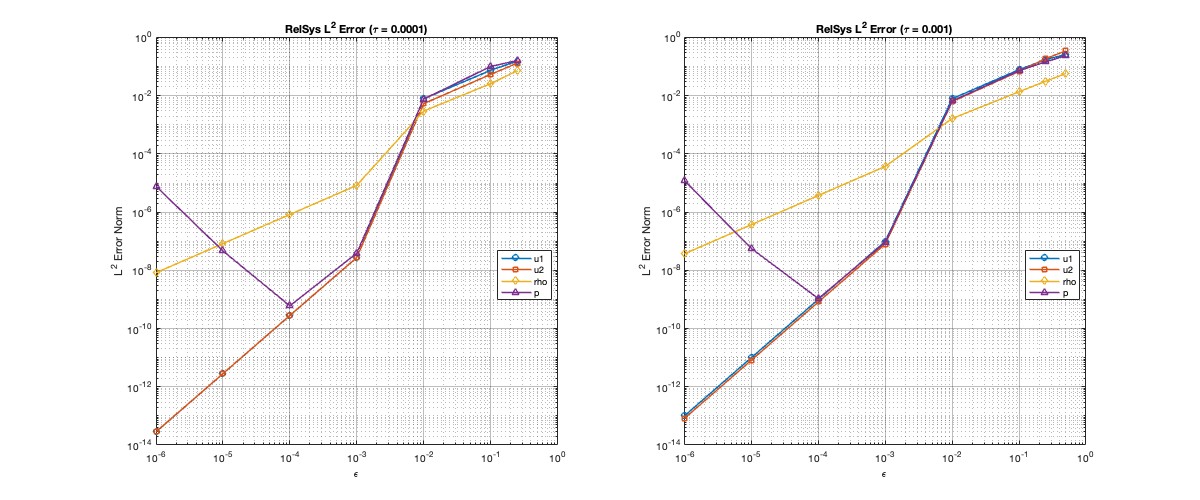}
		%		\caption{$\varepsilon = 0.5$,v-Raum, Diff exp. - Algo 1}
	\end{subfigure}
	\caption{Comparison of the $L^2$ errors (Shear Layer test case at $t = 1.5$) of the Relaxation System against the Full BGK reference solution (\cite{2ndPaper}) across different scaling parameters $\epsilon$ for $\tau = 0.0001$ (left) and $\tau = 0.001$ (right).}
	\label{fig:2d-Error}
\end{figure}

\subsubsection{Model errors}
Last but not least, Figure~\ref{fig:2d-Error} displays the $L^2$ norm errors of $u_1, u_2, p$, and $\rho$ computed by the relaxation system compared to a BGK reference solution obtained using the scheme described in \cite{2ndPaper}. The errors are plotted as a function of the scaling parameter $\varepsilon$ on a log-log scale for two different values of $\tau$ ($\tau = 0.0001$ and $\tau = 0.001$). 

The error response to the scaling parameter $\varepsilon$ exhibits two distinct behavioral regimes, noting that the transition threshold around $\varepsilon = 10^{-3}$ is specific to the present test configurations rather than serving as a universal regime boundary:
\begin{itemize}
\item Continuum-dominated regime (for smaller values, e.g., $\varepsilon \le 10^{-3}$): The velocity components $u_1$ and $u_2$ show excellent monotonic asymptotic convergence. The density $\rho$ also converges monotonically. The pressure $p$ behaves non-monotonically, attaining an error minimum at $\varepsilon = 10^{-4}$ before increasing again for smaller $\varepsilon$.
A possible explanation for this degradation at extremely small $\varepsilon$ is the influence of floating-point round-off effects, such as explicit $1/\varepsilon$-scalings in the predictor steps causing small physical fluctuations to approach the double-precision threshold when scaled back, though establishing this definitively would require further precision and tolerance sensitivity tests.
\item Kinetic regime (for larger values, e.g., $\varepsilon > 10^{-3}$): Deviations from the reference solution scale of order $\mathcal{O}(\varepsilon)$. Because the relaxation system truncates the infinite moment hierarchy inherent to kinetic theory, it cannot capture the strong non-equilibrium effects that emerge as $\varepsilon$ grows, leading to significant structural model errors for larger values of $\varepsilon$, as expected.
\end{itemize}

\section{Conclusions}\label{sec8}
In this work, we have introduced a hyperbolic relaxation system for the simulation of kinetic equations in the low-Mach number limits. By employing a micro-macro decomposition of the BGK model and appropriate simplifications, we derived a closed moment-based system whose hyperbolic core simplifies Grad's linear moment framework while bypassing the high-dimensional discretization of velocity space. 

Theoretically, we established the structural stability of the model in the incompressible limit via symmetrizability and proved its consistency with the INS equations. To bridge the continuum and rarefied regimes, we developed a high-order AP numerical framework combining IMEX RK schemes with high-order spatial reconstructions (WENO and central differences) under established stability constraints. In the asymptotic limit ($\epsilon \to 0$), the scheme successfully degenerates into a consistent projection method for the INS equations on collocated grids.

Numerical experiments in 1D and 2D validate our theoretical findings across distinct physical regimes. In the continuum-dominated regime (noting that thresholds such as $\varepsilon \le 10^{-3}$ are specific to the present test configurations rather than serving as universal regime boundaries), the relaxation model exhibits excellent agreement with full BGK reference solutions, with numerical comparisons demonstrating close alignment down to relative scales of $\mathcal{O}(10^{-14})$. At extreme asymptotic limits ($\varepsilon < 10^{-4}$), a possible explanation for the non-monotonic error degradation observed in the pressure field is the influence of floating-point round-off effects and catastrophic cancellation inherent to $1/\varepsilon$ predictor scalings, though establishing this definitively would require further precision tests. In the kinetic regime ($\varepsilon > 10^{-3}$), the expected growing deviations reflect the truncated moment hierarchy of the reduced-order system outside the hydrodynamic limit.

Crucially, our continuous moment-based approach achieves consistent asymptotic convergence toward the high-fidelity reference solution as $\varepsilon \to 0$, completely bypassing the computational bottleneck of high-dimensional velocity space discretizations. Furthermore, the relaxation framework is naturally formulated for and seamlessly applicable to arbitrary spatial dimensions without requiring specialized stencil or velocity discrete constructions. Coupled with computational speedups of several orders of magnitude in stiff regimes compared to full kinetic solvers, the proposed framework provides a robust and cost-effective tool for multiscale fluid dynamics, while remaining subject to the model-validity qualifications discussed earlier in this work.

Looking ahead, several promising avenues for future research emerge from the results presented in this work. One primary direction involves extending the current relaxation framework to incorporate more complex collision operators, such as the full Boltzmann equation or ES-BGK models, to more accurately capture physical phenomena involving non--constant Prandtl numbers. 
Furthermore, to improve the model's fidelity in more rarefied regimes (i.e., for larger values of $\varepsilon$), the system could be extended to include higher kinetic moments, akin to Grad's extended moment systems. This would systematically reduce the truncation error of the physical model outside the hydrodynamic limit.
Additionally, the development and implementation of high--order AP boundary conditions for complex geometries will be crucial for applying this methodology to practical engineering problems. We refer to \cite{torrilhon2008boundary} for investigations in this direction for moment systems. Finally, the significant reduction in computational cost offered by our relaxation system makes it an ideal candidate for Uncertainty Quantification (UQ) studies.

\bibliographystyle{plain}
\bibliography{reference_last2.bib}

\end{document}